\documentclass[11pt,onecoulme]{article}
\usepackage{amsfonts}
\usepackage{mathrsfs}
\usepackage{amssymb}
\usepackage{color, amsmath,amssymb, amsfonts, amstext,amsthm, latexsym}

\usepackage{amssymb, epsfig, amssymb, latexsym}
\usepackage{amsmath}
\usepackage{graphicx}
\usepackage{longtable}
\usepackage{appendix}
\DeclareMathOperator{\esssup}{esssup}

\numberwithin{equation}{section}

\allowdisplaybreaks

\newtheorem{definition}{Definition}[section]
\newtheorem{theorem}[definition]{Theorem}
\newtheorem{lemma}[definition]{Lemma}
\newtheorem{remark}[definition]{Remark}

\newtheorem{example}[definition]{Example}
\newtheorem{assumption}[definition]{Assumption}

  \renewcommand\appendix{\par
    \setcounter{section}{0}
    \setcounter{subsection}{0}
    \gdef\thesection{ Appendix \Alph{section}}}

\title{Optimal Control of Unbounded Stochastic Evolution Systems in Hilbert Spaces}
\author{Shanjian Tang\footnote{ Institute of Mathematical Finance and Department of Finance and Control Sciences,
	School of Mathematical Sciences, Fudan University,
	Shanghai 200433, P. R. China, sjtang@fudan.edu.cn. This author is partially supported by National Science Foundation of China (Grant
	No. 11631004) and National Key R\&D Program of China (Grant No. 2018YFA0703903)} \quad Jianjun Zhou\footnote{Corresponding author. College of Science,
	Northwest A\&F University, Yangling 712100, Shaanxi, P. R.
	China, zhoujianjun@nwsuaf.edu.cn. This author is partially supported by  the Shaanxi Natural Science Foundation
	(Grant No. 2025JC-YBMS-021)}  }

\date{}

\begin{document}

\maketitle

\pagestyle{plain}

\begin{abstract}
	Optimal control and the associated second-order Hamilton-Jacobi-Bellman (HJB) equation are studied for unbounded stochastic evolution systems in Hilbert spaces. A new notion of viscosity solution, featured by absence of $B$-continuity,  is introduced for the second-order HJB equation in the sense of Crandall and Lions, and is shown to coincide with the classical solutions and to satisfy a stability property. The value functional is proved to be the unique continuous viscosity solution to the second-order HJB equation, with the coefficients being not necessarily  $B$-continuous. Our result provides a new theory of viscosity solutions to the HJB equation for optimal control of stochastic evolutionary equations---driven by a linear unbounded operator---in a Hilbert space, and removes the $B$-continuity assumption on the coefficients which is used in the existing literature.
	

\medskip


\end{abstract}

%
%
%
%

\section{Introduction}

\par
Let $\{W_t,t\geq0\}$ be a cylindrical Wiener process in
Hilbert space $\Xi$ on a  complete probability  space  $(\Omega,{\mathcal {F}}, \mathbb{P})$, with $\mathbb{F}=\{{\mathcal {F}}_{s}\}_{0\leq s\leq T}$ being the  natural filtration, augmented with the totality ${\mathcal{N}}$ of all $P$-null sets in $\mathcal{F}$. The process  $\mathbf{u}=\{u_s, s\in [t,T]\}$ is  $\mathbb{F}$-progressively measurable,  taking values in a Polish space $(U,d_{U})$; and the totality of such processes is denoted by  $ {\mathcal{U}}[t,T]$.
Consider a strongly continuous semi-group   $\{e^{tA}, t\geq0\}$  of bounded linear operators  in Hilbert space
$H$ with $A$ being the generator.


Assume that the pair of functionals $(b, \sigma):[0,T]\times H\times U\rightarrow H\times  L_2(\Xi, H)$   are jointly continuous with $(b,\sigma)(\cdot, u)$ being uniformly  Lipschitz continuous  for each $u\in U$.

For a fixed finite time $T>0$, consider in Hilbert space $H$ the  following controlled stochastic evolution
equation (SEE):
\begin{eqnarray}\label{state1}
\begin{cases}
dX^{t,x,\mathbf{u}}_s= AX^{t,x,\mathbf{u}}_sds+
b\left(s,X_s^{t,x,\mathbf{u}},u_s\right)ds +\sigma\left(s,X_s^{t,x,\mathbf{u}},u_s\right)dW_s,  \quad s\in (t,T];\\
\  X_t^{t,x,\mathbf{u}}=x\in H,
\end{cases}
\end{eqnarray}
with $X^{t,x,\mathbf{u}}$ being  the $H$-valued unknown process.

For given  $(t,x)\in [0,T]\times H$,   we aim  at  maximizing the following utility functional
\begin{eqnarray}\label{cost1}
J(t,x,\mathbf{u}):=Y^{t,x,\mathbf{u}}_t,\quad \mathbf{u}\in {{\mathcal 	{U}}}[t,T]
\end{eqnarray}
where the process $Y^{t,x,\mathbf{u}}$ is defined via solution of the following backward stochastic differential equation (BSDE):
\begin{eqnarray}\label{fbsde1}
\begin{aligned}
Y^{t,x,\mathbf{u}}_s=&\phi\left(X_T^{t,x,\mathbf{u}}\right)+\int^{T}_{s}q\left(l,X_l^{t,x,\mathbf{u}},Y^{t,x,\mathbf{u}}_l,Z^{t,x,\mathbf{u}}_l,u_l\right)dl\\
&-\int^{T}_{s}Z^{t,x,\mathbf{u}}_ldW_l,\quad \ a.s.\ \ \mbox{for all}\ \ s\in [t,T].
\end{aligned}
\end{eqnarray}
Here  $q: [0,T]\times H\times \mathbb{R}\times \Xi\times U\to \mathbb{R}$ and $\phi: H\to \mathbb{R}$ are given real functionals and are  uniformly  Lipschitz continuous. The optimal stochastic control problem is to look for $\bar{\mathbf{u}}\in \mathcal{U}[0,T]$ such that
$$J(t,x,\bar{\mathbf{u}})=\sup_{\mathbf{u}\in \mathcal{U}[0,T]} J(t,x, \mathbf{u}). $$
The value functional of the  optimal control problem is given by
\begin{eqnarray}\label{value1}
V(t,x):=\mathop{\esssup}\limits_{\mathbf{u}\in{\mathcal{U}}[t,T]} Y^{t,x,\mathbf{u}}_t,\ \  (t,x)\in [0,T]\times H.
\end{eqnarray}

It is well-known that a powerful basic tool to  the optimal stochastic control problem is the so-called dynamic programming method---initially due to R. Bellman---which, in particular, indicates that
the value functional  $V$ should be ``the solution" of  the following second-order Hamilton-Jacobi-Bellman (HJB) equation:
\begin{eqnarray}\label{hjb1}
\quad&\begin{cases}
V_{t}(t,x)+\langle A^*\nabla_xV(t,x), x\rangle _H
+{\mathbf{H}}(t,x,V(t,x),\nabla_xV(t,x),\nabla^2_{x}V(t,x))\\[2mm]
\qquad\qquad = 0,\quad   (t,x)\in
[0,T)\times H;\\[3mm]
V(T,x)=\phi(x),\quad x\in H,
\end{cases}
\end{eqnarray}
with
\begin{eqnarray*}
	&&{\mathbf{H}}(t,x,r,p,l)
	:=\sup_{u\in{
			{U}}}\bigg{\{}
	\langle p, b(t,x,u)\rangle_H+\frac{1}{2}\mbox{Tr}[ l \sigma(t,x,u)\sigma^*(t,x,u)]   \\
&& \qquad \ \ +q(t,x,r,p\sigma(t,x,u),u)\bigg{\}},  \quad  (t,x,r,p,l)\in [0,T]\times H\times \mathbb{R}\times H\times {\mathcal{S}}(H).
\end{eqnarray*}
Here,   $A^*$ and $\sigma^*$ are the adjoint operators of $A$ and $\sigma$,  respectively, ${\mathcal{S}}(H)$ the space of bounded, self-adjoint operators on $H$ equipped with the operator norm,  ${\langle\cdot,\cdot\rangle_{H}}$ the scalar product of $H$,  and
$V_{t}$ the time-derivative of $V$, $\nabla_x$ and $\nabla^2_{x}$  the first- and  second-order Fr\'{e}chet derivatives.

\par
A notion of viscosity solutions for second-order HJB equations in Hilbert spaces has been introduced  by Lions~\cite{lio1, lio3} for a bounded operator $A$, and by \'{S}wi\c{e}ch~\cite{swi} for an unbounded operator $A$ in the framework of the so-called $B$-continuous viscosity solutions---which was initially introduced for first-order equations by
Crandall and  Lions~\cite{cra6, cra7}. In an earlier paper,  Lions~\cite{lio2}  considered
a specific second-order HJB equation for optimal control of the Zakai
equation using  some ideas of the $B$-continuous viscosity solutions. Based on the {Crandall-Ishii lemma} of Lions~\cite{lio3}, the first comparison theorem for $B$-continuous viscosity sub- and super-solutions was proved by \'{S}wi\c{e}ch~\cite{swi}.   In the introduction of Chapter 3  in the preface (see page x), Fabbri et al.~\cite{fab1} commented:`` In this book, we focus on the notion of a so-called B-continuous viscosity solution which was introduced by Crandall and Lions in [141, 142] for first-order equations and later adapted to second-order equations in [539]. The key result in the theory is the comparison principle, which is very technical. " The reader is referred to   the monograph   of  Fabbri et al.~\cite{fab1} for a detailed account
of  the theory of viscosity solutions in a separable Hilbert space.

\par
As mentioned above, in the case of unbounded linear operator $A$, the existing  Crandall-Lions viscosity solutions have to be $B$-continuous, and  the $B$-continuity of the coefficients (which ensures the $B$-continuity of the value functional  and   the comparison theorem)
 has to be assumed  so as to overcome the difficulty caused by the unbounded operator $A$.
  {In this article, we give a new notion  of  Crandall-Lions viscosity solutions  to HJB equation~\eqref{hjb1},  featured by absence of  $B$-continuity, 
and show that the value functional
$V$  defined in  (\ref{value1}) is the  unique viscosity solution to HJB equation (\ref{hjb1}), without assuming the $B$-continuity of  the coefficients.
\par
 We observe that the  norm $|\cdot|_H$ can be added to the test function in the definition of viscosity solution in  Fabbri et al. \cite{fab1},  for  it satisfies  It\^o inequality (1.111) in~\cite[page 84]{fab1} when $A$ is the generator of a $C_0$ contraction semi-group.  While the shifted norm $|\cdot-y|_H$ with fixed $y\ne0$, could not satisfy  the It\^o inequality (1.111)
 in~\cite[page 84]{fab1},   it cannot be added to the test function. This fact constitutes  the essential difficulty in the  theory of viscosity solutions  to  HJB equations with unbounded linear term. To circumvent this difficulty, Fabbri et al. \cite{fab1} assume that the coefficients satisfy  the $B$-continuity condition in~\cite[(3.21) and (3.22), page 184]{fab1}).
 \par
 We make use of the crucial fact  that the  modified  functional
 \begin{eqnarray}\label{fung}
 g(t,x):=|x-e^{(t-\hat{t})A}\hat{y}|^4,\quad (t,x)\in [\hat{t},T]\times H
 \end{eqnarray}
   with fixed $(\hat{t},\hat{y})\in [0,T)\times H$ satisfies
It\^o inequality (\ref{jias510815jia11}) (see Lemma \ref{fbjia4}) and can be added to the test functions in our viscosity solutions.
 The introduction of this term significantly helps  us to establish the comparison principle of viscosity solutions. Indeed, we can define an auxiliary function $\Psi$ which includes the  functional $g$ in the proof of  our comparison theorem.  By this, we only need the continuity  under $|\cdot|_H$ rather than  the $B$-continuity of the value functional. In particular, the comparison theorem is  established
when the coefficients is uniformly  Lipschitz continuous  with respect to the norm $|\cdot|_H$.
We emphasize that, in our viscosity solution theory in infinite dimensional spaces, the
$B$-continuity assumption on the coefficients, existing in the literature (see Fabbri et al.~\cite[Chapter 3, p. 171-365 ]{fab1}),  is  dispensed with  in our framework.
\par
In the
proof of comparison principle (see ~\cite[Theorem 3.50, page 206]{fab1}),  the Ekeland-Lebourg theorem
(see ~\cite[Theorem 3.25, page 188]{fab1})  in $H_{-2}$ is used to ensure the existence of  a maximum point of the auxiliary function. Note that our viscosity sub-solution  and super-solution  are merely strong-continuous in the Hilbert space $H$.  Application of the Ekeland-Lebourg theorem in $H$ to  the auxiliary function,  would incur a linear functional  $\langle p,\cdot\rangle_H$ (with some fixed $p\in H$)---which unfortunately does not satisfy It\^o inequality (\ref{jias510815jia11})---to be added to the test function.
   Therefore, the Ekeland-Lebourg theorem does not work straightforwardly in our case.
   \par
 Recently, the Borwein-Preiss variational principle  (see   Borwein \& Zhu~\cite[Theorem 2.5.2]{bor1}) 
 has been found to be powerful in  the viscosity solution theory for path-dependent HJB equation; see, for example the second author~\cite{zhou5}.
 Applying it to functional defined on $ H$ with gauge-type
   functional $|x-y|_H^4$, will  incur the functional $\sum^{\infty}_{i=0}\frac{1}{2^i}|x-x^i|_H^{4}$ with fixed  $\{x^i\}_{i\geq0}\in H$---
    which also unfortunately does not satisfy  It\^o inequality (\ref{jias510815jia11})--- to be added to the test functions.
    However, applying it to functional defined on  space $[0,T]\times H$ with 
    functional
 \begin{eqnarray}\label{upsilon}
{\Upsilon}((t,x),(s,y)):=|s-t|^2+|e^{(t\vee s-t)A}x-e^{(t\vee s-s)A}y|_H^4, \ (t,x;s,y)\in ([0,T]\times H)^2,
 \end{eqnarray}
 will incur the functional  $\sum^{\infty}_{i=0}\frac{1}{2^i}[|t-t_i|^2+|x-e^{(t-t_i)A}x^i|_H^{4}]$ with fixed
  $\{(t_i,x^i)\}_{i\geq0}\in [0,\hat{t}]\times H$ and $\hat{t}\in [0,T)$---which satisfies  It\^o inequality (\ref{jias510815jia11})---
   to be added to the test functions.
  This is why we use   in our paper 
  the Borwein-Preiss variational principle with 
  $\Upsilon$ rather than the the Ekeland-Lebourg theorem.
\par
 Since our value functional is merely strong-continuous rather than necessarily $B$ continuous, the existing  Crandall-Ishii lemma for $B$-upper semi-continuous  functional defined on Hilbert space $H$, established in \cite[Theorem 3.27, page 189]{fab1} so as to prove the comparison theorem for $B$-continuous viscosity solutions, does not apply to our current situation. We give a Crandall-Ishii lemma for strongly upper semi-continuous  functional defined on space $[0,T]\times H$ in Theorem \ref{theorem0513}, so as to prove the comparison theorem for strong-continuous viscosity solutions.
\par
Regarding existence, we  show that the value functional  $V$  defined in  (\ref{value1}) is  a viscosity solution to the HJB
equation given in  (\ref{hjb1}) by  It\^o  formula, It\^o inequality  and dynamic programming principle.
\par
The rest of the paper is organized as follows. In Section 2,  we give preliminaries on notations and the
Borwein-Preiss variational principle.
In {Section} 3, we introduce  preliminary results on stochastic  optimal control problems, and  the dynamic programming principle,
 which will be used in the subsequent {sections}. An It\^o inequality for $g$ defined in (\ref{fung}) is also provided.
In {Section} 4, we define classical and viscosity solutions to
HJB equations (\ref{hjb1}) and  prove  that the value functional $V$ defined by (\ref{value1}) is a viscosity solution.   We also show
the consistency with the notion of classical solutions and the stability result.
 A Crandall-Ishii lemma  for  strongly upper semi-continuous  functional defined on space $[0,T]\times H$
is given in {Section} 5.
The uniqueness of viscosity solutions to  (\ref{hjb1}) is proved in {Section} 6.

\section{Preliminaries
}  \label{RDS}

\subsection{Some Notations}
Let $\Xi$, $K$ and $H$ denote real separable Hilbert spaces of  inner products
$\langle\cdot,\cdot\rangle_\Xi$, $\langle\cdot,\cdot\rangle_K$ and $\langle\cdot,\cdot\rangle_H$, respectively.  The notation $|\cdot|$ stands for the norm in
various spaces, and a subscript is attached to indicate the underlying Hilbert space whenever  necessary.
$L(\Xi,H)$ is the space of all
bounded linear operators from $\Xi$ into $H$, and  the subspace of
all Hilbert-Schmidt operators, equipped with the Hilbert-Schmidt norm, is
denoted by $L_2(\Xi,H)$. Write $L(H):=L(H,H)$. Denote by ${\mathcal{S}}(H)$ the Banach space of all bounded and self-adjoint
operators in Hilbert space $H$,  endowed with the operator norm. The operator $A$,  with the domain being denoted by ${\mathcal {D}}(A)$, generates a strongly continuous semi-group   $\{e^{tA}, t\geq0\}$ of bounded linear operators in Hilbert space $H$. $A^*$ is  the adjoint operator of $A$, whose domain is denoted by ${\mathcal {D}}(A^*)$. Let
$T>0$ be a fixed number. We define a metric on $[0,T]\times H$ as follows:
$$
                        {d}((t,x),(s,y)):=|s-t|+|x-y|, \quad (t,x; s,y)\in ([0,T]\times H)^2.
$$
Then $([0,T]\times H, {d})$ is a complete metric space. 
We define a functional on $[0,T]\times H$ as follows:
\begin{eqnarray}\label{upsilon3}
                       \Upsilon((t,x),(s,y)):=|s-t|^2+|x_{t,t\vee s}^A-y_{s,t\vee s}^A|^4, \quad (t,x; s,y)\in ([0,T]\times H)^2,
\end{eqnarray}
where
$$
x_{t,r}^A:=e^{(r-t)A}x, \quad r\geq t, (t,x)\in [0,T]\times H.
$$
\begin{definition}\label{definitionc0409}
	Let $f:[t,T]\times H\rightarrow K$ be given for  some $t\in[0,T)$.
	\begin{description}
		\item[(i)]
		We say $f\in C^0([t,T]\times H,K)$ (resp., $USC^0([t,T]\times H,K)$, $LSC^0([t,T]\times H,K)$) if $f$ is continuous (resp., upper semi-continuous, lower semi-continuous)  on the metric space $([t,T]\times H, d)$.
		\item[(ii)]
		We say $f\in C_p^0([t,T]\times H,K)$ if $f\in C^0([t,T]\times H,K)$ and $f$ grows  in a polynomial way, i.e.,  there exist constants $L>0$ and $q\geq 0$ such that, for all $(s,x)\in [t,T]\times H$,
		$$
		|f(s,x)|\leq L(1+|x|^q).
		$$
\item[(iii)] We say $f\in C^{0+}([t,T]\times H,K)  (\mbox{resp.,} \  USC^{0+}([t,T]\times H,K), LSC^{0+}([t,T]\times H,K) )$ if 
  $f$ is continuous (resp., upper semi-continuous, lower semi-continuous) under $\Upsilon$, i.e.,
 for every fixed  $(s,x)\in [t,T]\times H$ and constant $C>0$, we have $\lim_{\Upsilon((l,y),(s,x))\rightarrow0,|y|\leq C}f(l,y)= f(s,x)$ (resp., $\limsup_{\Upsilon((l,y),(s,x))\rightarrow0,|y|\leq C}f(l,y)\leq f(s,x)$, $\liminf_{\Upsilon((l,y),(s,x))\rightarrow0,|y|\leq C}f(l,y)\geq f(s,x)$).
 \item[(iv)]
		We say $f\in C_p^{0+}([t,T]\times H,K)$ if $f\in C^{0+}([t,T]\times H,K)$ and $f$ grows  in a polynomial way.
	\end{description}
\end{definition}
For simplicity, we write $C^0([t,T]\times H):=C^0([t,T]\times H,\mathbb{R})$ and  similarly for $C^0_p([t,T]\times H), C^{0+}([t,T]\times H), C^{0+}_p([t,T]\times H), USC^{0+}([t,T]\times H), LSC^{0+}([t,T]\times H)$.
\begin{definition}\label{definitionc04091}
	Let $t\in[0,T)$ and $f:[t,T]\times H\rightarrow \mathbb{R}$ be given.
	\begin{description}
		\item[(i)]  We say $f\in C^{1,2}([t,T]\times H)\subset C^{0+}([t,T]\times H)$ if   $f_t$, $\nabla_{x}f$ and $\nabla^2_{x}f$ exist and are continuous in $(t,x)$  on the metric space $([t,T]\times H, d)$.
		\par
		\item[(ii)] We say
		$f\in C^{1,2}_p([t,T]\times H)$ if $f\in C^{1,2}([t,T]\times H)$ and $f$ together with  all its derivatives grow  in a polynomial way.
\par
		\item[(iii)] {We say
		$f\in C^{0,2}_p([t,T]\times H)$ if $f\in C^{0}_p([t,T]\times H)$, $\nabla_xf\in C^{0}_p([t,T]\times H,H)$  and $\nabla^2_{x}f\in C^{0}_p([t,T]\times H,\mathcal{S}(H))$.}
	\end{description}
\end{definition}
Notice that, $$
|x^A_{t,s}-y|\leq |x-y|+|x-x^A_{t,s}|, \ \mbox{if}\  (t,x), (s,y)\in [0,T]\times H \ \mbox{and}\ t\leq s; \ \mbox{and}
$$
$$
|x-y^A_{s,t}|\leq  |x^A_{s,t}-y^A_{s,t}|+|x^A_{s,t}-x|, \ \mbox{if}\  (t,x), (s,y)\in [0,T]\times H \ \mbox{and}\ t> s.
$$
we have, for every fixed $(t,x)\in [0,T]\times H$, $\Upsilon((s,y),(t,x))\rightarrow0$ as $d((s,y),(t,x))\rightarrow0$. Therefore,
\begin{lemma}\label{lemma-09021}
Let $t\in[0,T)$ and $f:[t,T]\times H\rightarrow K$ be given.
If  $f\in C^{0+}([t,T]\times H,K)$ (resp., $USC^{0+}([t,T]\times H,K)$, $LSC^{0+}([t,T]\times H,K)$), 
then $f\in C^0([t,T]\times H,K)$ (resp., $USC^0([t,T]\times H,K)$, $LSC^0([t,T]\times H,K)$).
\end{lemma}
By a cylindrical Wiener process defined on a  complete probability  space
$(\Omega,{\mathcal {F}},\mathbb{P})$, and with values in a Hilbert space $\Xi$,
we mean a family $\{W_t,t\geq0\}$ of linear mappings $\Xi\rightarrow
L^2(\Omega)$ such that for every $\xi, \eta \in \Xi$,
$\{W_t\xi,t\geq0\}$ is a real Wiener process and
${\mathbb{E}}[W_t\xi\cdot W_t\eta]=t \langle\xi,\eta\rangle_\Xi $.  The filtration $\mathbb{F}=\{{\mathcal {F}}_{t}\}_{0\leq t\leq T}$  is the natural one of $W$, augmented with the totality  ${\mathcal{N}}$ of all the $\mathbb{P}$-null sets of $\mathcal{F}$:
$$
{\mathcal{F}}_{t}=\sigma(W_s:s\in[0,t])\bigvee \mathcal
{N}.
$$
The filtration
$\mathbb{F}$  satisfies the usual conditions.
For every $[t_1,t_2]\subset[0,T]$, we also use the
notation:
$$
{\mathcal{F}}_{t_1}^{t_2}=\sigma(W_s-W_{t_1}:s\in[t_1,t_2])\bigvee \mathcal
{N}.
$$
We also write ${\mathcal{F}}^t:=\{{\mathcal{F}}^s_t, t\leq s\leq T\}$.
\par
By $\mathcal{P}$,   we denote the predictable $\sigma$-algebra generated by predictable processes,  and by $\mathcal{B}(\Theta)$,  we denote the Borel $\sigma$-algebra of a topological space $\Theta$.

Next, we introduce spaces of random variables or stochastic processes, taking values in a Hilbert space $K$.

For  $p\in[1,\infty)$ and $t\in [0,T]$, $L^p(\Omega,{\mathcal{F}}_{t},\mathbb{P};K)$ is the space of $K$-valued ${\mathcal{F}}_{t}$-measurable random variables $\xi$, equipped with the norm
$$|\xi|=\mathbb{E}[|\xi|^p]^{1/p}<\infty;$$
$L^{p}_{\mathcal{P}}(\Omega\times [0,T];K)$ is  the space of all predictable  processes $y\in L^p(\Omega\times [0,T];K)$, equipped  with the norm
$$|{y}|=\mathbb{E}\left[\int_{0}^{T}|{y}_t|^{p}dt\right]^{1/p}<\infty;$$
$L^{p}_{\mathcal{P}}(\Omega;L^{2}([0,T];K))$, is  the space of all predictable  processes $\{y_t,t\in [0,T]\}$, equipped with  the  norm
$$|{y}|=\mathbb{E}\left[\left(\int_{0}^{T}|{y}_t|^{2}dt\right)^{p/2}\right]^{1/p}<\infty;$$
for  $t\in (0,T]$, {$ L^{p}_{\mathcal{P}}(\Omega,C([0,t];K))$} is  the space of all $K$-valued continuous predictable processes
$\{y_s,s\in[0,t]\}$, equipped with the norm
$$|{y}|=\mathbb{E}\left[\sup_{s\in [0,T]}|y_s|^{p}\right]^{1/p}<\infty.$$
Elements of $L^{p}_{\mathcal{P}}(\Omega,C([0,t];K))$ are identified if they are  indistinguishable.

{ In the paper,}  whenever convenient, we also use the notation $f|_a$ or $f(\cdot)|_a$ or $f(x)|_{x=a}$ to denote the image of a map $f$ at the point $a$, and use the notation  $f|^b_a$, $f(\cdot)|^b_a$ or $f(x)|^{x=b}_{x=a}$ to denote the difference $f(b)-f(a)$ of the images of a map $f$ between two points $b$ and $a$.

\subsection{Borwein-Preiss variational principle}
In the literature, the Ekeland-Lebourg Theorem (see ~\cite[Theorem 3.25, page 188]{fab1}) in Hilbert spaces $H_{-2}$ is used to  {\it find out  a maximum point} of a perturbed auxiliary function  in the proof of comparison principle. Unfortunately, the Ekeland-Lebourg theorem does not apply
here since 
our viscosity sub-solution and super-solution are merely strong-continuous in the Hilbert space $H$ (see Remark \ref{remarkv0129120240129e} for details). Therefore,  a variational principle in a metric space is quite appealing  in our case.

In this {subsection},  we introduce  {the Borwein-Preiss variational principle (see Borwein \& Zhu  \cite[Theorem 2.5.2]{bor1}), which is} crucial to proving  the uniqueness and  stability  of viscosity solutions.

\begin{lemma}\label{theoremleft}
	Let the operator $A$ be the generator of a $C_0$ contraction
		{semi-group} and  $f\in USC^{0+}([0,T]\times H)$ bounded from above and satisfy
\begin{eqnarray}\label{09032025c}
\limsup_{|x|\rightarrow\infty}\sup_{t\in[0,T]}\left[\frac{f(t,x)}{|x|}\right]<0.
\end{eqnarray}
	Suppose that  $\{\delta_i\}_{i\geq0}$ is a sequence of positive number, and suppose that $\varepsilon>0$ and $(t_0,x_0)\in [0,T]\times H$ satisfy
	$$
	f\left(t_0,x_0\right)\geq \sup_{(t,x)\in [0,T]\times H}f(x)-\varepsilon.
	$$
	Then there exist $(\hat{t},\hat{x})\in [0,T]\times H$, $\{(t_i,x_i)\}_{i\geq1}\subset [0,T]\times H$ and constant $C>0$ such that,
   denoting $\mathbf{f}:=f-\sum^{\infty}_{i=0}\delta_i\Upsilon(\cdot, (t_i,x_i))\leq f$,
	\begin{description}
		\item[(i)]  $\Upsilon((\hat{t},\hat{x}),(t_i,x_i))\leq \frac{\varepsilon}{2^{i}\delta}, |x_i|\leq C$  for all $  i\geq 0$, and $t_i\uparrow \hat{t}$ as $i\rightarrow\infty$,
		\item[(ii)]  $\mathbf{f}(\hat{t},\hat{x})\geq f(t_0,x_0)$, and
		\item[(iii)]  $\mathbf{f}(s,x)<\mathbf{f}(\hat{t},\hat{x})$ for all $(s,x)\in [\hat{t},T]\times H\setminus \{(\hat{t},\hat{x})\}$.
		
	\end{description}
\end{lemma}
\begin{proof} The proof is similar to that of Borwein \& Zhu  \cite[Theorem 2.5.2 ]{bor1}, and is given below for  convenience of the reader.

Define sequences $\{(t_i,x_i)\}_{i\geq1}$ and $\{B_i\}_{i\geq1}$ inductively as follows. Set
\begin{eqnarray}\label{left1}
B_0:=\left\{(s,x)\in [t_0,T]\times H\bigg{|} \ f(s,x)-\delta_0\Upsilon((s,x),(t_0,x_0))\geq f(t_0,x_0)
                                         \right\}.
\end{eqnarray}
Since $(t_0,x_0)\in B_0$, $B_0$ is nonempty. Moreover it is closed because  $f\in USC^0([0,T]\times H,K)$ 
 by Lemma \ref{lemma-09021} and $\Upsilon(\cdot,(t_0,x_0))\in C^0([t_0,T]\times H,K)$.
 Indeed, for $(t,x;s,y)\in [t_0,T]\times H$,
\begin{eqnarray*}
\begin{aligned}
        &\big|\Upsilon((t,x),(t_0,x_0))-\Upsilon((s,y),(t_0,x_0))\big|\\
        =&\big||t-t_0|^2-|s-t_0|^2+|x-(x_0)^A_{t_0,t}|^4-|y-(x_0)^A_{t_0,s}|^4\big|\\
        \leq& 2T|s-t|^2+(|x-(x_0)^A_{t_0,t}|+|y-(x_0)^A_{t_0,s}|)^3[|x-y|+|(x_0)^A_{t_0,t}-(x_0)^A_{t_0,s}|]\\
        \rightarrow&0 \ \mbox{as}\ (s,y)\rightarrow(t,x)\ \mbox{in}\ ([t_0,T]\times H,d).
\end{aligned}
\end{eqnarray*}
 We also have that, for
all $(s,x)\in B_0$,
\begin{eqnarray}\label{left2}
\delta_0\Upsilon((s,x),(t_0,x_0))\leq f(s,x)-f(t_0,x_0)\leq \sup_{(l,y)\in [t,T]\times H}f(l,y)-f(t_0,x_0)\leq \varepsilon.
\end{eqnarray}
Take $(t_1,x_1)\in B_0$ such that
\begin{eqnarray}\label{left21111}
f(t_1,x_1)-\delta_0\Upsilon((t_1,x_1),(t_0,x_0))\geq \sup_{(s,y)\in B_0}[f(s,y)-\delta_0\Upsilon((s,y),(t_0,x_0))]-\frac{\delta_1\varepsilon}{2\delta_0}, \
\end{eqnarray}
and define in a similar way
\begin{eqnarray}\label{left3}
B_1:=\left\{\begin{matrix} (s,x)\\
\in B_0\cap [t_1,T]\times H\end{matrix} \,\, \Bigg{|} \,\,  \begin{matrix} f(s,x)-\sum_{k=0}^{1}\delta_k\Upsilon((s,x),(t_k,x_k))\\[3mm]
\quad \geq f(t_1,x_1)-\delta_0\Upsilon((t_1,x_1),(t_0,x_0))\end{matrix}\right\}.
\end{eqnarray}
In general, suppose that we have defined $(t_j,x_j)$ and $B_j$ for $j=1,2,\ldots, i-1$ such that
\begin{eqnarray}\label{left4}
\begin{aligned}
&\quad f(t_j,x_j)-\sum_{k=0}^{j-1}\delta_k\Upsilon((t_j,x_j),(t_k,x_k))\\
&\geq \sup_{(s,y)\in B_{j-1}}\bigg{[}f(s,y)-\sum_{k=0}^{j-1}\delta_k\Upsilon((s,y),(t_k,x_k))\bigg{]}-\frac{\delta_j\varepsilon}{2^j\delta_0},
\end{aligned}
\end{eqnarray}
and
\begin{eqnarray}\label{left5}
&B_j:=\left\{ \begin{matrix} (s,x)\\
\in B_{j-1}\cap [t_j,T]\times H\end{matrix} \,\, \Bigg{|}
\begin{matrix} f(s,x)-\sum_{k=0}^{j}\delta_k\Upsilon((s,x),(t_k,x_k))\\[3mm]
 \geq f(t_j,x_j)-\sum_{k=0}^{j-1}\delta_k\Upsilon((t_k,x_k),(t_j,x_j))
 \end{matrix}\right\}.
\end{eqnarray}
We choose $(t_i,x_i)\in B_{i-1}$ such that
\begin{eqnarray}\label{left6}
\begin{aligned}
&\quad f(t_i,x_i)-\sum_{k=0}^{i-1}\delta_k\Upsilon((t_i,x_i),(t_i,x_k))\\
&\geq \sup_{(s,y)\in B_{i-1}}\bigg{[}f(s,y)-\sum_{k=0}^{i-1}\delta_k\Upsilon((s,y),(t_k,x_k))\bigg{]}-\frac{\delta_i\varepsilon}{2^i\delta_0},
\end{aligned}
\end{eqnarray}
and we define
\begin{eqnarray}\label{left7}
 &B_i:=\left\{\begin{matrix} (s,x)\\
\in B_{i-1}\cap [t_i,T]\times H\end{matrix}\,  \Bigg{|}  \begin{matrix} f(s,x)-\sum_{k=0}^{i}\delta_k\Upsilon((s,x),(t_k,x_k))\\
 \geq f(t_i,x_i)-\sum_{k=0}^{i-1}\delta_k\Upsilon((t_i,x_i),(t_k,x_k))
 \end{matrix}\right\}.
\end{eqnarray}
We can see that  for every $i=1,2,\ldots$,  $B_i$ is a closed and nonempty set. By (\ref{09032025c}) and (\ref{left1}), there exists a constant $C>0$ such that \begin{eqnarray}\label{09042025}
|x_i|\leq C\quad \mbox{for all}\quad i\geq 0.
\end{eqnarray}
  It follows from (\ref{left6}) and (\ref{left7}) that, for all $(s,x)\in B_i$,
\begin{eqnarray*}
\begin{aligned}
	&\delta_i\Upsilon((s,x),(t_i,x_i))\\
\leq& \bigg{[}f(s,x)-\sum_{k=0}^{i-1}\delta_k\Upsilon((s,x),(t_k,x_k))\bigg{]}-\bigg{[}f(t_i,x_i)-\sum_{k=0}^{i-1}\delta_k\Upsilon((t_i,x_i),(t_k,x_k))\bigg{]}\\
	\leq&\sup_{(s,y)\in B_{i-1}}\bigg{[}f(s,y)-\sum_{k=0}^{i-1}\delta_k\Upsilon((s,y),(t_k,x_k))\bigg{]}-\bigg{[}f(t_i,x_i)-\sum_{k=0}^{i-1}\delta_k\Upsilon((t_i,x_i),(t_k,x_k))\bigg{]}\\
	\leq& \frac{\delta_i\varepsilon}{2^i\delta_0},
\end{aligned}
\end{eqnarray*}
which implies that
\begin{eqnarray}\label{left8}
\Upsilon((s,x),(t_i,x_i))\leq \frac{\varepsilon}{2^i\delta_0},\ \ \mbox{for all}\  (s,x)\in B_i.
\end{eqnarray}
 By the following Lemma \ref{lemma-0902},
 there  exists a unique
$(\hat{t},\hat{x})\in \bigcap_{i=0}^{\infty}B_i$ 
 and $t_i\uparrow \hat{t}$  as $i\rightarrow \infty$.
 Then $(\hat{t},\hat{x})$
satisfies (i) by (\ref{09042025}), (\ref{left2}) and  (\ref{left8}). For any $(s,x)\in [\hat{t},T]\times H$ and $(s,x)\neq (\hat{t},\hat{x})$, we have $(s,x)\notin\bigcap_{i=0}^{\infty}B_i$, and therefore, for some $j$,
\begin{eqnarray}\label{left9}
\begin{aligned}
f(s,x)-\sum_{k=0}^{\infty}\delta_k\Upsilon((s,x),(t_k,x_k))&\leq f(s,x)-\sum_{k=0}^{j}\delta_k\Upsilon((s,x),(t_k,x_k))\\
&<
f(t^j,x_j)-\sum_{k=0}^{j-1}\delta_k\Upsilon((t_j,x_j),(t_k,x_k)).
\end{aligned}
\end{eqnarray}
On the other hand, it follows from (\ref{left1}), (\ref{left7}) and $(\hat{t},\hat{x})\in \bigcap_{i=0}^{\infty}B_i$ that, for any $q\geq j$,
\begin{eqnarray}\label{left10}
\begin{aligned}
&f(t_0,x_0)\leq f(t_j,x_j)-\sum_{k=0}^{j-1}\delta_k\Upsilon((t_j,x_j),(t_k,x_k))\\
\leq& f(t_q,x_q)-\sum_{k=0}^{q-1}\delta_k\Upsilon((t_q,x_q),(t_k,x_k))
\leq f(\hat{t},{\hat{x}})-\sum_{k=0}^{q}\delta_k\Upsilon((\hat{t},\hat{x}),(t_k,x_k)).
\end{aligned}
\end{eqnarray}
Letting $q\rightarrow\infty$ in (\ref{left10}), we obtain
\begin{eqnarray}\label{left11}
\begin{aligned}
f(t_0,x_0)\leq & f(t_j,x_j)-\sum_{k=0}^{j-1}\delta_k\Upsilon((t_j,x_j),(t_k,x_k))\\
\leq& f(\hat{t},\hat{x})-\sum_{k=0}^{\infty}\delta_k\Upsilon((\hat{t},\hat{x}),(t_k,x_k)),
\end{aligned}
\end{eqnarray}
which verifies (ii). Combining  (\ref{left9}) and (\ref{left11}) yields (iii).
\end{proof}
\begin{lemma}\label{lemma-0902}
Let all the conditions in Lemma \ref{theoremleft}  hold true. Let $(t_i,x_i)$ and  $B_i$ be given in (\ref{left6}) and (\ref{left7}), respectively.
Then there exists  a unique $(\hat{t},\hat{x})\in \bigcap_{i=0}^{\infty}B_i$.
\end{lemma}
\begin{proof}
It is clear that there exists $\hat{t}\in [0,T]$ such that $t_i\uparrow \hat{t}$ as $i\rightarrow \infty$.
Notice that the operator $A$ is the generator of a $C_0$ contraction
		{semi-group}.
        By (\ref{left8}) and the definition of $\Upsilon$, for all $j>i$,
      \begin{eqnarray}\label{09022025}
             |(x_i)^A_{t_i,\hat{t}}-(x_j)^A_{t_j,\hat{t}}|^4\leq|e^{(\hat{t}-t_j)A}|^4|(x_i)_{t_i,t_j}-x_j|^4\leq 
             \frac{\varepsilon}{2^i\delta_0}.
   \end{eqnarray}
        This means $\{(x_j)^A_{t_j,\hat{t}}\}_{j\geq0}$ is a Cauchy sequence in $H$, then there exists a unique $\hat{x}\in H$ such that
        $|(x_j)^A_{t_j,\hat{t}}-\hat{x}|\rightarrow0$ as $j\rightarrow\infty$. Thus, $\lim_{j\rightarrow\infty}\Upsilon((t_j,x_j),(\hat{t},\hat{x}))=0$.

        By (\ref{left7}), for all $j>i$,
\begin{eqnarray}\label{09032025}
        f(t_j,x_j)-\sum_{k=0}^{i}\delta_k\Upsilon((t_j,x_j),(t_k,x_k))
 \geq f(t_i,x_i)-\sum_{k=0}^{i-1}\delta_k\Upsilon((t_i,x_i),(t_k,x_k)).
\end{eqnarray}
        Notice that, for every fixed $(t,x), (\bar{t},\bar{x})\in [0,T]\times H$ satisfying $t\geq \bar{t}$,
        $$
              |x-\bar{x}^A_{\bar{t},t}|\leq |x-y^A_{s,t}|+|y^A_{s,t}-\bar{x}^A_{\bar{t},t}|\leq  |x-y^A_{s,t}|+|y-\bar{x}^A_{\bar{t},s}|,\ \mbox{for all} \ (s,y)\in [\bar{t},t]\times H;
        $$
         $$
              |x-\bar{x}^A_{\bar{t},t}|\leq |x-x^A_{t,s}|+|x^A_{t,s}-y|+|y-\bar{x}^A_{\bar{t},s}|+|\bar{x}^A_{\bar{t},s}-\bar{x}^A_{\bar{t},t}|,\ \mbox{for all} \ (s,y)\in [t,T]\times H.
        $$
        We have,
        $$
                   |x-\bar{x}^A_{\bar{t},t}|\leq \liminf_{\Upsilon((t,x),(s,y))\rightarrow0,\ t,s\geq \bar{t}}|y-\bar{x}^A_{\bar{t},s}|.
        $$
Then, for every fixed $(\bar{t},\bar{x})\in [0,T]\times H$, $\Upsilon(\cdot,(\bar{t},\bar{x}))\in LSC^{0+}([\bar{t},T]\times H)$.
      Since $f\in USC^{0+}([0,T]\times H)$, letting $j\rightarrow\infty$ in (\ref{09032025}), by (\ref{09042025}) we obtian
 $$
        f(\hat{t},\hat{x})-\sum_{k=0}^{i}\delta_k\Upsilon(\hat{t},\hat{x}),(t_k,x_k))
 \geq f(t_i,x_i)-\sum_{k=0}^{i-1}\delta_k\Upsilon((t_i,x_i),(t_k,x_k)).
        $$
        Thus, $(\hat{t},\hat{x})\in B_i$ for $i\geq 0$. Therefore, $(\hat{t},\hat{x})\in \bigcap_{i=0}^{\infty}B_i$. Obviously, $\hat{t}$ is unique.  If there exists another $\hat{y}\in H$ such that $(\hat{t},\hat{y})\in \bigcap_{i=0}^{\infty}B_i$, by (\ref{left8}),
        $$
               |\hat{x}-\hat{y}|\leq |(x_i)^A_{t_i,\hat{t}}-\hat{x}|+|(x_i)^A_{t_i,\hat{t}}-\hat{y}|\leq 2\left(\frac{\varepsilon}{2^i\delta_0}\right)^{\frac{1}{4}}.
        $$
        Letting $i\rightarrow\infty$, we have $\hat{x}=\hat{y}$.
\end{proof}

\section{ Stochastic optimal control problems} 
\par
In this {section}, we consider the controlled state
equation (\ref{state1}) and cost equation (\ref{fbsde1}).
Let us introduce the admissible control. Let $t$ and $s$ be two deterministic times such that $0\leq t\leq s\leq T$.
\begin{definition}
	An admissible control process $\mathbf{u}=\{u_r,  r\in [t,s]\}$ on $[t,s]$  is an $\{{\mathcal{F}}_r\}_{t\leq r\leq s}$-progressively measurable process taking values in a Polish  space $(U,d_U)$. The set of all admissible controls on $[t,s]$ is denoted by ${\mathcal{U}}[t,s]$. We identify two processes $\mathbf{u}$ and $\tilde{\mathbf{u}}$ in ${\mathcal{U}}[t,s]$
	and write $\mathbf{u}\equiv\tilde{\mathbf{u}}$ on $[t,s]$, if $\mathbb{P}(u_{\cdot}=\tilde{u}_{\cdot} \ a.e. \ \mbox{in}\ [t,s])=1$.
\end{definition}
The precise notion of solution to equation (\ref{state1}) will be given next. We make the following assumption.
\begin{assumption}\label{hypstate}
	\begin{description}
		\item[(i)]
		The operator $A$ is the generator of
		a $C_0$  {semi-group}   of bounded linear operator $\{e^{tA}, t\geq0\}$ in Hilbert space
		$H$.
		\item[(i')]
		The operator $A$ is the generator of a $C_0$ contraction
		{semi-group}  of bounded linear operators  $\{e^{tA}, t\geq0\}$ in
		Hilbert space $H$.
		\par
\par
		\item[(ii)] $(b,\sigma) :[0,T]\times H\times U\to  H\times  L_2(\Xi,H)$ is continuous, {and
        $(b,\sigma)(\cdot,\cdot,\cdot,u)$ is continuous in $ [0,T]\times H$, uniformly in $u\in U$.} Moreover,
		there is a constant $L>0$ such that,  for all $(t, x, u) \in [0,T]\times H\times U$,
		\begin{eqnarray*}
		&&\displaystyle \left|b(t,x,u)\right|^2\vee\left|\sigma(t,x,u)\right|^2_{L_2(\Xi,H)}\leq
		L^2(1+|x|^2),\\
		&&\displaystyle \left|b(t,x,u)-b(t,y,u)\right|\vee\left|\sigma(t,x,u)-\sigma(t,y,u)\right|_{L_2(\Xi,H)}\leq
		L|x-y|.
		\end{eqnarray*}
\item[(iii)]
	$
	q: [0,T]\times H\times \mathbb{R}\times \Xi\times U\rightarrow \mathbb{R}$ and $ \phi:  H\rightarrow \mathbb{R}$ are continuous, {and
        ${q}$ is continuous in $(t,x)\in [0,T]\times H$, uniformly in $u\in U$. Moreover,}     there is a  constant $L>0$
	such that, for all $(t, x, y, z, x', y',z', u)
	\in [0,T]\times (H\times\mathbb{R}\times \Xi)^2\times U$,
	\begin{eqnarray*}
		&&| q(t,x,y,z,u)|\leq L(1+|x|+|y|+|z|),
		\\
		&&| q(t,x,y,z,u)-q(t,x',y',z',u)|\leq L(|x-x'|+|y-y'|+|z-z'|),
		\\
		&&|\phi(x)-\phi(x')|\leq L|x-x'|.
	\end{eqnarray*}
	\end{description}
\end{assumption}
\par
We say that $X$ is a mild solution of equation $(\ref{state1})$ with initial data $(t,\xi)\in [0,T)\times L^p(\Omega,\mathcal{F}_t,\mathbb{P};H)$ for some $p>2$ if it is a continuous, $\mathbb{F}$-predictable $H$-valued process such that  $\mathbb{P}$-a.s.,
\begin{eqnarray*}
	X_s=e^{(s-t)A}\xi+\int_{t}^{s}{e^{(s-l)A}}b(l,X_l,u_l)dl+\int_{t}^{s}{e^{(s-l)A}}\sigma(l,X_l,u_l)dW_l,  \ s\in [t,T].
\end{eqnarray*}
To emphasize the dependence on initial data and control, we denote the solution by $X^{t,\xi,u}$.
\par
The following lemma is standard, and is available, for example,  in \cite[Theorems 1.127, 1.130, 1.131]{fab1} and  \cite[Proposition 4.3]{fuh0}.

\begin{lemma}\label{lemmaexist0409}
Assume that Assumption \ref{hypstate} (i), (ii) and (iii)  be satisfied. Then for every $p>2$, $\mathbf{u}\in {\mathcal{U}}[0,T]$, and
	$\xi\in L^{p}(\Omega,\mathcal{F}_t,\mathbb{P}; H)$, SEE~(\ref{state1}) admits a
	unique mild solution $X^{t,\xi,\mathbf{u}}$, and BSDE (\ref{fbsde1}) has a unique
	pair of solutions $(Y^{t,\xi,\mathbf{u}}, Z^{t,\xi,\mathbf{u}})$. Furthermore, let  $X^{t,\eta,\mathbf{v}}$ and $(Y^{t,\eta,\mathbf{v}}, Z^{t,\eta,\mathbf{v}})$ be the solutions of SEE (\ref{state1}) and BSDE (\ref{fbsde1})
	corresponding to $(t,\eta,\mathbf{v})\in [0,T)\times L^p(\Omega,\mathcal{F}_t,\mathbb{P}; H)\times {\mathcal{U}}[0,T]$. Then, there is a positive constant  $C$ depending only on  $p$, $T$, $L$ and
	$M_1=:\sup_{s\in [0,T]}|e^{sA}|$, such that
	\begin{eqnarray}\label{state1est}
	\mathbb{E}\left[\sup_{s\in [t,T]}\left|X_s^{t,\xi,\mathbf{u}}\right|^p\right]
	+\mathbb{E}\bigg{[}\sup_{s\in [t,T]}\left|Y^{t,\xi,\mathbf{u}}_s\right|^p\bigg{]}
	+ {\mathbb{E}}\bigg{(}\int^{T}_{t}\left|Z^{t,\xi,\mathbf{u}}_s\right|^2ds\bigg{)}^{\frac{p}{2}}\leq C(1+\mathbb{E}[|\xi|^p]); \
	\end{eqnarray}
\begin{eqnarray}\label{2.6}
 \sup_{\mathbf{u}\in {\mathcal{U}}[0,T]}\mathbb{E}\left[\sup_{l\in[t,s]}\left|X^{t,\xi,\mathbf{u}}_l-\xi^A_{t,l}\right|^p\right]
	\leq
	C\left(1+\mathbb{E}\left[|\xi|^p\right]\right)|s-t|^{\frac{p}{2}}, \quad s\in[t,T];
	\end{eqnarray}
\begin{eqnarray}\label{fbjia4}
\begin{aligned}
	&\mathbb{E}\left[\sup_{s\in [t,T]}\left|X^{t,\xi,\mathbf{u}}_s-X^{t,\eta,\mathbf{v}}_s\right|^p\right]+\mathbb{E}\left[\sup_{s\in [t,T]}\left|Y^{t,\xi,\mathbf{u}}_s-Y^{t,\eta,\mathbf{v}}_s\right|^p\right]\\
&
\qquad \qquad+ {{\mathbb{E}}\bigg{(}\int^{T}_{t}\left|Z^{t,\xi,\mathbf{u}}_s-Z^{t,\eta,\mathbf{v}}_s\right|^2ds\bigg{)}^{\frac{p}{2}}}\\
	\leq & C \mathbb{E}[|\xi-\eta|^p]\\[3mm]
	&+C\mathbb{E}\!\left[\left|(b, \sigma)(X^{t,\eta,\mathbf{v}}_r,\cdot)\Bigm|^{u_r}_{v_r}\right|^p_{H\times L_2(\Xi, H)}
	\!\!\!\!\!\!+\left|q(r,X^{t,\eta,\mathbf{v}}_r,Y^{t,\eta,\mathbf{v}}_r,Z^{t,\eta,\mathbf{v}}_r,\cdot)\Bigm|^{u_r}_{v_r}\right|^p\right]\! dr.
\end{aligned}
	\end{eqnarray}
	\par
	Moreover,  let $A_\mu:=\mu A(\mu I-A)^{-1}$ be the Yosida approximation of $A$ and
	let $X^\mu$ be the solution of the following SEE
	\begin{eqnarray}\label{07162}
	\begin{aligned}
	X^\mu_s=e^{(s-t)A_\mu}\xi+\int^{s}_{t}e^{(s-l)A_\mu}b(l,X^{\mu}_l,u_l)dl+\int^{s}_{t}e^{(s-l)A_\mu}\sigma(l,X^{\mu}_l,u_l)dW_l, s\in [t,T].\
	\end{aligned}
	\end{eqnarray}
	Then, we have
	\begin{eqnarray}\label{0717}
	\lim_{\mu\rightarrow\infty}\mathbb{E}\left[\sup_{s\in [t,T]}\left|X^{t,\xi,\mathbf{u}}_s-X^\mu_s\right|^p\right]=0.
	\end{eqnarray}
\end{lemma}

For the particular case of a deterministic $\xi$, i.e. $\xi=x\in H$, denote by $X^{t,x,\mathbf{u}}$ the solution of equation (\ref{state1}) corresponding to $(t,x,\mathbf{u})\in [0,T]\times H\times {\mathcal{U}}[t,T]$. The following    lemma is needed to prove the existence and uniqueness of viscosity solutions.

\begin{lemma}\label{theoremito2}
	Let Assumption \ref{hypstate} be satisfied. For every $p\geq2$ and $(t,x, y)\in [0,T)\times H\times H$,  $\mathbb{P}$-a.s.,  for all $s\in [t,T]$, we have
	\begin{eqnarray}\label{jias510815jia11}
	\begin{aligned}
	\ \ \ \ &\quad
	|X^{t,x,\mathbf{u}}_s-y^A_{t,s}|^{p}=|\hat{X}_s|^{p}\\
	&\leq |\hat{X}_t|^{p}+p\int^{s}_{{t}}\biggl[|\hat{X}_l|^{p-2}\left\langle \hat{X}_l,\,  b(l,X^{t,x,\mathbf{u}}_l,u_l)\right\rangle_H\\
	&\quad+\frac{1}{2}\mbox{\rm Tr}\biggl((|\hat{X}_l|^{p-2}I+(p-2)|\hat{X}_l|^{p-4}\hat{X}_l\langle \hat{X}_l, \cdot\rangle_H) (\sigma\sigma^*)(l,X^{t,x,\mathbf{u}}_l,u_l)\biggr)\biggr]dl\\
	&\quad
	+p\int^{s}_{t}|\hat{X}_l|^{p-2}\langle \hat{X}_l, \, \sigma(l,X^{t,x,\mathbf{u}}_l,u_l)dW_l\rangle_H,
	\end{aligned}
	\end{eqnarray}
	where
	$\hat{X}_s:=X^{t,x,\mathbf{u}}_s-y^A_{t,s}$ for $s\in [t, T]$. Moreover,
   \begin{eqnarray}\label{fbjia41022}
	\begin{aligned}
	&\quad
	\mathbb{E}[|X^{t,x,\mathbf{u}}_s-X^{t,y,\mathbf{u}}_s|^2|{\mathcal{F}}_t]
+	\mathbb{E}[|Y^{t,x,\mathbf{u}}_s-Y^{t,y,\mathbf{u}}_s|^2|{\mathcal{F}}_t]\\
	 &\quad+\int^{T}_{s}\!\!\!\mathbb{E}[|Y^{t,x,\mathbf{u}}_l-Y^{t,y,\mathbf{u}}_l|^2+|Z^{t,x,\mathbf{u}}_l-Z^{t,y,\mathbf{u}}_l|^2|{\mathcal{F}}_t]dl
	\leq C|x-y|^2, \ {s\in[t,T]}
	\end{aligned}
	\end{eqnarray}
	and
	\begin{eqnarray}\label{fbjia51022}
	\mathbb{E}[|X^{t,x,\mathbf{u}}_s|^2|{\mathcal{F}}_t]+\mathbb{E}[|Y^{t,x,\mathbf{u}}_s|^2|{\mathcal{F}}_t]
	+\int^{T}_{t}\!\mathbb{E}[|Z^{t,x,\mathbf{u}}_s|^2|{\mathcal{F}}_t]ds\leq C(1+|x|^2), \ {s\in[t,T]} \
	\end{eqnarray}
	for a positive constant $C$ depending only on  $T$, $L$  and $M_1$.

\end{lemma}

\begin{proof} 
Define $\hat{X}$ by $\hat{X}_s:=X_s-y^A_{t,s}$ for $s\in [t,T]$. 
 Then, $\hat{X}$ solves the following SEE
\begin{eqnarray*}
	\hat{X}_s=e^{(s-t)A}\hat{X}_t+\int^{s}_{t}e^{(s-l)A}b(l,X_l, u_l)dl
	+\int^{s}_{t}e^{(s-l)A}\sigma(l,X_l,u_l)dW_l,\ s\in [t,T].
\end{eqnarray*}
By It\^o inequality  (see \cite[Proposition 1.166]{fab1}), we obtain (\ref{jias510815jia11}).

Letting $y=0$, $p=2$ and taking conditional expectation with respect to $\mathcal{F}_t$ in (\ref{jias510815jia11}),  by Assumption \ref{hypstate} (ii),
\begin{eqnarray*}
	\mathbb{E}[|X^{t,x,\mathbf{u}}_s|^{2}|\mathcal{F}_t]
\leq |x|^{2}+C\int^{s}_{{t}}\mathbb{E}[(1+|X^{t,x,\mathbf{u}}_l|^{2})|\mathcal{F}_t]dl.
	\end{eqnarray*}
Using  Gronwall's inequality, we get $ \mathbb{E}[|X^{t,x,\mathbf{u}}_s|^{2}|\mathcal{F}_t]$ satisfies (\ref{fbjia51022}). Similarly, we can show that
 $\mathbb{E}[|X^{t,x,\mathbf{u}}_s-X^{t,y,\mathbf{u}}_s|^2|{\mathcal{F}}_t]$ satisfies (\ref{fbjia41022}).

 Now we only prove (\ref{fbjia41022}), and (\ref{fbjia51022}) is proved in a similar way.  For simplicity, we write
\begin{eqnarray*}
	\begin{aligned}
		\tilde{X}_s:=&X^{t,x,\mathbf{u}}_s-X^{t,y,\mathbf{u}}_s,\ \tilde{Y}_s:=Y^{t,x,\mathbf{u}}_s-Y^{t,y,\mathbf{u}}_s, \
		\tilde{Z}_s:=Z^{t,x,\mathbf{u}}_s-Z^{t,y,\mathbf{u}}_s,\\
		\tilde{q}_s:=&\, q(s,\gamma,u_s)\Bigm|^{\gamma=(X^{t,x,\mathbf{u}}_s,Y^{t,x,\mathbf{u}}_s,Z^{t,x,\mathbf{u}}_s)}_{\gamma=(X^{t,y,\mathbf{u}}_s,Y^{t,y,\mathbf{u}}_s,Z^{t,y,\mathbf{u}}_s)}\, ,\ \quad s\in [t,T].
	\end{aligned}
\end{eqnarray*}
Applying  It\^o formula  to $e^{\beta s}|\tilde{Y}_s|^2$ on $[t,T]$, we obtain
\begin{eqnarray}\label{jias510815jia1022}
\begin{aligned}
&\quad e^{\beta s}|\tilde{Y}_s|^2+\int^{T}_{s}e^{\beta l}[\beta|\tilde{Y}_l|^2+|\tilde{Z}_l|^2]dl\\
&= e^{\beta T}|\phi(X_T^{t,x,\mathbf{u}})-\phi(X_T^{t,y,\mathbf{u}})|^2+2\int^{T}_{{s}}e^{\beta l}\langle\tilde{Y}_l,
\tilde{q}_l\rangle_Hdl-2\int^{s}_{t}e^{\beta l}\langle\tilde{Y}_l, \tilde{Z}_ldW_l\rangle_H.
\end{aligned}
\end{eqnarray}
Taking conditional expectation with respect to ${\mathcal{F}}_t$, 
we obtain,  {for $\beta>0$},
\begin{eqnarray*}
	\begin{aligned}
		&e^{\beta s}\mathbb{E}[|\tilde{Y}_s|^2|{\mathcal{F}}_t]+\int^{T}_{s}e^{\beta l}\mathbb{E}[\beta|\tilde{Y}_l|^2+|\tilde{Z}_l|^2|{\mathcal{F}}_t]dl\\
		\leq& e^{\beta T}L^2\mathbb{E}[|\tilde{X}_T|^2|{\mathcal{F}}_t]+\int^{T}_{{s}}e^{\beta l}\mathbb{E}\left[(2L^2+3L)|\tilde{Y}_l|^2+L|\tilde{X}_l|^2+\frac{1}{2}|\tilde{Z}_l|^2
		\bigg{|}{\mathcal{F}}_t\right]dl.
	\end{aligned}
\end{eqnarray*}
Let $\beta=2L^2+3L+1$,  we obtain (\ref{fbjia41022}).
\end{proof}

\begin{remark}\label{remarks}
	\begin{description}
		\item[(i)]      Since in general the functional $|x-\hat{y}|^p$ with fixed $\hat{y}\in H$ and $p\geq 2$ does not satisfies It\^o inequality (\ref{jias510815jia11}), then, it fails to be used directly  to construct  any auxiliary functional to prove the
		uniqueness and stability of viscosity solutions. However, by the preceding lemma,
		the  modified  functional  $g(t,x):=|x-\hat{y}^A_{\hat{t},t}|^p$ with fixed $(\hat{t},\hat{y})\in [0,T)\times H$ behaves very much like $|x-\hat{a}|^p$, and even gets rid of the effect of the unbounded operator $A$.  It shall be used to prove the uniqueness  of viscosity solutions when  the coefficients are merely strong-continuous  under $|\cdot|_H$.
\item[(ii)]  In light of (\ref{jias510815jia11}), it is convenient to choose $p$ as an even integer. Some estimates later will fail for $p=2$, so we choose $p=4$. Any larger even integer will serve for our purpose as well.
		\item[(iii)] As we stated in Item (i), $|x-\hat{y}^A_{\hat{t},t}|^4$ rather than $|x-\hat{y}|^4$ satisfies It\^o inequality (\ref{jias510815jia11}),
then the function ${\Upsilon}((t,x),(s,y))$
 rather than $|s-t|^2+|x-y|^4$
 can be applied 
   to Lemma \ref{theoremleft} to
		ensure the existence of  a maximizing point of a perturbed auxiliary functional in the proof of uniqueness.
  This explains why we use the function ${\Upsilon}$ in Lemma \ref{theoremleft}.
	\end{description}
\end{remark}
\noindent
The functionals $J(t,\xi,\mathbf{u}):=J(t,x,\mathbf{u})|_{x=\xi}$ and $Y^{t,\xi,\mathbf{u}}_t$, $(t,\xi)\in [0,T]\times L^p(\Omega,\mathcal{F}_t,\mathbb{P}; H)$, are related by the following
theorem.
\begin{theorem}\label{theoremj=y}
	\ \
	Under  Assumption \ref{hypstate} (i), (ii) and (iii),    for every $p>2$,  $(t,\xi, \eta)\in [0,T]\times (L^p(\Omega,\mathcal{F}_t,\mathbb{P}; H))^2$
and $\mathbf{u}\in {\mathcal{U}}[t,T]$, we have
	\begin{eqnarray}\label{j=y}
	J(t,\xi,\mathbf{u})=Y^{t,\xi,\mathbf{u}}_t,\ \mbox{and}
	\end{eqnarray}
	\begin{eqnarray}\label{0903jia1022}
	|Y^{t,\xi,\mathbf{u}}_t-Y^{t,\eta,\mathbf{u}}_t|\leq C|\xi-\eta|.
	\end{eqnarray}
\end{theorem}

From the uniqueness of the solution of (\ref{fbsde1}), it
follows that
$$
Y^{t,x,\mathbf{u}}_{t+\delta}=Y^{t+\delta,X^{t,x,\mathbf{u}}_{t+\delta},\mathbf{u}}_{t+\delta}=J\left(t+\delta,X^{t,x,\mathbf{u}}_{t+\delta},\mathbf{u}\right),\ \ \mathbb{P}\mbox{-a.s.}
$$
Formally,  under the assumptions  Assumption \ref{hypstate}  (i), (ii) and (iii), the  value functional $V(t,x)$
defined by (\ref{value1})
is  ${\mathcal{F}}_t$-measurable.
However, we have
\begin{theorem}\label{valuedet}  Let  Assumption \ref{hypstate} (i), (ii) and (iii) hold true.
	Then,  $V$ is a deterministic functional.
\end{theorem}
{The proof of the above two theorems is essentially the same as that in finite dimensional case, we postpone it to \ref{appendixc}}.

The following  property of the value functional $V$ which we present is an immediate consequence of {Lemma \ref{theoremito2}} and Theorem \ref{valuedet}.
\begin{lemma}\label{lemmavaluev}
	Let  Assumption \ref{hypstate} (i), (ii)  and (iii) be satisfied. Then,  
	for all 
	$0\leq t\leq T$, $x, y\in H$,
	\begin{eqnarray}\label{valuelip}
	|V(t,x)-V(t,y)|\leq C|x-y|, \quad
	|V(t,x)|\leq C(1+|x|).
	\end{eqnarray}
\end{lemma}

\par
We also have
\begin{lemma}\label{lemma3.6}
	For all $(t,\xi,\mathbf{u})\in[0,T]\times L^p(\Omega,\mathcal{F}_t,\mathbb{P};H)\times{\mathcal
		{U}}[t,T]$ for some $p>2$, we have
	\begin{eqnarray}\label{3.14}
	V(t,\xi)\geq Y^{t,\xi,\mathbf{u}}_t,\
	\ \mathbb{P}\mbox{-a.s.},
	\end{eqnarray}
	and for any $\varepsilon>0$ there exists an
	admissible control $\mathbf{u}\in{\mathcal
		{U}}[t,T]$ such that
	\begin{eqnarray}\label{3.15}
	V(t,\xi)\leq Y^{t,\xi,\mathbf{u}}_t+\varepsilon\
	\ \mathbb{P}\mbox{-a.s.}.
	\end{eqnarray}
\end{lemma}

\begin{proof}  From Theorem \ref{theoremj=y} and  the definition of $V(t,x)$,  we have, for any $\mathbf{u}\in{\mathcal
	{U}}[t,T]$,
$$
V(t,\xi)=V(t,y)|_{y=\xi}=\left[\mathop{\esssup}\limits_{\mathbf{v}\in{\mathcal
		{U}}[t,T]}J(t,y,\mathbf{v})\right]\bigg{|}_{y=\xi}\geq
J(t,y,\mathbf{u})|_{y=\xi}=Y^{t,\xi,\mathbf{u}}_t.
$$
We now prove (\ref{3.15}).
Let $f^k$ be the functions defined in the proof procedure of Theorem \ref{theoremj=y}.
Then $f^k(\omega)$ is ${\mathcal
	{F}}_{t}$-measurable and
converges to $\xi$ strongly and
uniformly.
  For {$k\varepsilon>3(1+C)$},  we have
$$
|f^k-\xi|_H\leq \frac{\varepsilon}{3(1+C)}\leq \frac{\varepsilon}{3}.
$$
Therefore, for every $\mathbf{u}\in {\mathcal {U}}[t,T]$, by (\ref{0903jia1022}) and (\ref{valuelip}),
$$
\left|Y^{t,f^k,\mathbf{u}}_t-Y^{t,\xi,\mathbf{u}}_t\right|\leq
\frac{\varepsilon}{3},
\ \
\left|V(t,f^k)-V(t,\xi)\right|\leq
\frac{\varepsilon}{3},\ \mathbb{P}\mbox{-a.s.}.
$$
Moreover, for every $h^n\in H$,  as in the proof  of Theorem \ref{valuedet},
we  choose an admissible control
$\mathbf{v}^n$ such that
$$
V(t,h^n)\leq
Y^{t,h^n, \mathbf{v}^n}_t+\frac{\varepsilon}{3}, \quad  \mathbb{P}\mbox{-a.s.}.
$$
Set
$\mathbf{u}:=\sum^{\infty}_{n=1}\mathbf{v}^n1_{\{\xi\in
	B_{n,k}\}}$.  Then,
\begin{eqnarray*}
\begin{aligned}
	Y^{t,\xi,\mathbf{u}}_t&\geq
	-|Y^{t,f^k,\mathbf{u}}_t-Y^{t,\xi,\mathbf{u}}_t|+Y^{t,f^k,\mathbf{u}}_t\geq-\frac{\varepsilon}{3}+\sum^{\infty}_{n=1}Y^{t,h^n,\mathbf{v}^{n}}_t1_{\{\xi\in
		B_{n,k}\}}\\
	&\geq-\frac{2\varepsilon}{3}+\sum^{\infty}_{n=1}V(t,h^n)1_{\{\xi\in
		B_{n,k}\}}=-\frac{2\varepsilon}{3}+V(t,f^k)
	\geq-\varepsilon+V(t,\xi),  \ \  \mathbb{P}\mbox{-a.s.}.
\end{aligned}
\end{eqnarray*}
Thus, we have (\ref{3.15}).
\end{proof}

\par
We now discuss the dynamic programming principle  (DPP) for the optimal control problem (\ref{state1}), (\ref{fbsde1}) and (\ref{value1}).
For this purpose, we define the family of backward semigroups associated with BSDE (\ref{fbsde1}), following the
idea of Peng \cite{peng11}.
\par
Given the initial path $(t,x)\in [0,T)\times H$, a positive number $\delta\leq T-t$, an admissible control $\mathbf{u}\in {\mathcal{U}}[t,t+\delta]$ and
a real-valued random variable $\zeta\in L^2(\Omega,{\mathcal{F}}_{t+\delta},\mathbb{P};\mathbb{R})$, we define
\begin{eqnarray}\label{gdpp}
G^{t,x,\mathbf{u}}_{s,t+\delta}[\zeta]:=\tilde{Y}^{t,x,\mathbf{u}}_s,\ \
\ \ \ \ s\in[t,t+\delta],
\end{eqnarray}
where $(\tilde{Y}^{t,x,\mathbf{u}}_s,\tilde{Z}^{t,x,\mathbf{u}}_s)_{t\leq s\leq
	t+\delta}$ is the solution of the following
BSDE:
\begin{eqnarray}\label{bsdegdpp}
\begin{cases}
d\tilde{Y}^{t,x,\mathbf{u}}_s =-q(s,X^{t,x,\mathbf{u}}_s,\tilde{Y}^{t,x,\mathbf{u}}_s,\tilde{Z}^{t,x,\mathbf{u}}_s,u_s)ds
 +\tilde{Z}^{t,x,\mathbf{u}}_sdW_s, \quad s\in [t, t+\delta]; \\
\ \tilde{Y}^{t,x,\mathbf{u}}(t+\delta)=\zeta
\end{cases}
\end{eqnarray}
with $X^{t,x,u}$ being the solution of SEE (\ref{state1}).
%
%
%

\begin{theorem}\label{theoremddp} 
	Let  Assumption \ref{hypstate} (i), (ii) and (iii)  be satisfied. Then, the value functional
	$V$ satisfies the following: for
	any $(t,x)\in [0,T)\times H$ and $0\leq t<t+\delta\leq T$,
	\begin{eqnarray}\label{ddpG}
	V(t,x)=\mathop{\esssup}\limits_{\mathbf{u}\in{\mathcal
			{U}}[t,t+\delta]}G^{t,x,\mathbf{u}}_{t,t+\delta}[V(t+\delta,X^{t,x,\mathbf{u}}_{t+\delta})].
	\end{eqnarray}
\end{theorem}
\par

\begin{proof}
By the definition of $V(t,x)$ we have
$$
V(t,x)=\mathop{\esssup}\limits_{\mathbf{u}\in{\mathcal
		{U}}[t,T]}G^{t,x,\mathbf{u}}_{t,T}[\phi(X^{t,x,\mathbf{u}}_T)]=\mathop{\esssup}\limits_{\mathbf{u}\in{\mathcal
		{U}}[t,T]}G^{t,x,\mathbf{u}}_{t,t+\delta}\left[Y^{t+\delta,X^{t,x,\mathbf{u}}_{t+\delta},\mathbf{u}}_{t+\delta}\right].
$$
From (\ref{3.14}) and the comparison theorem of BSDE (see  \cite[Theorem 2.7]{zhou1}), we have
$$
V(t,x)\leq \mathop{\esssup}\limits_{\mathbf{u}\in{\mathcal
		{U}}[t,T]}G^{t,x,\mathbf{u}}_{t,t+\delta}[V(t+\delta,X^{t,x,\mathbf{u}}_{t+\delta})].
$$
On the other hand, from (\ref{3.15}), for any
$\varepsilon>0$ and $\mathbf{u}\in {\mathcal {U}}[t,T]$,  there
is an admissible control $\overline{\mathbf{u}}\in{\mathcal
	{U}}[t+\delta,T]$ such that
$$
V(t+\delta,X^{t,x,\mathbf{u}}_{t+\delta})\leq Y^{t+\delta,X^{
		t,x,\mathbf{u}}_{t+\delta},\overline{\mathbf{u}}}_{t+\delta}+\varepsilon, \ \  \mathbb{P}\mbox{-a.s.}
$$
For any $\mathbf{u}\in {\mathcal {U}}[t,T]$ with $\overline{\mathbf{u}}\in{\mathcal {U}}[t+\delta,T]$ from above, define
$$
\mathbf{v}=\mathbf{u}1_{[t,t+\delta]}(\cdot)+\overline{\mathbf{u}}1_{(t+\delta,T]}(\cdot).
\in {\mathcal
	{U}}[t,T].
$$
Then, by Lemma \ref{lemmaexist0409}, 
 we have
\begin{eqnarray*}
	&&\left|G^{t,x,\mathbf{u}}_{t,t+\delta}\left[Y^{t+\delta,X^{t,x,\mathbf{u}}_{t+\delta},\overline{\mathbf{u}}}_{t+\delta}\right]
	-G^{t,x,\mathbf{u}}_{t,t+\delta}[V(t+\delta, X^{t,x,\mathbf{u}}_{t+\delta})]\right|\leq C\varepsilon,
\end{eqnarray*}
where the constant $C$ is independent of
admissible control processes.
Therefore,
\begin{eqnarray*}
\begin{aligned}
	V(t,x)&\geq
	 G^{t,x,\mathbf{v}}_{t,t+\delta}\left[Y^{t+\delta,X^{t,x,\mathbf{v}}_{t+\delta},\mathbf{v}}(t+\delta)\right]
   =G^{t,x,\mathbf{u}}_{t,t+\delta}\left[Y^{t+\delta,X^{t,x,\mathbf{u}}_{t+\delta},\overline{\mathbf{u}}}_{t+\delta}\right]\\
	&\geq G^{t,x,\mathbf{u}}_{t,t+\delta}[V(t+\delta,X^{t,x,\mathbf{u}}_{t+\delta})]-C\varepsilon.
\end{aligned}
\end{eqnarray*}%
For the arbitrariness of $\varepsilon$, we have (\ref{ddpG}).
\end{proof}

From  Theorem \ref{theoremddp}, we  have

\begin{theorem}\label{theorem3.9}
	Under  Assumption \ref{hypstate} (i), (ii) and (iii), the value functional $V\in C^{0+}([0,T]\times H)$ and
	there is a constant $C>0$ such that for every  $0\leq t\leq s\leq T, x,y\in H$,
	\begin{eqnarray}\label{hold}
	\ \ \ \ \      |V(t,x)-V(s,y^A_{t,s})|\leq
	C(1+|x|)(s-t)^{\frac{1}{2}}+C|x-y|.
	\end{eqnarray}
\end{theorem}

\begin{proof}
Let $(t,x,y)\in[0,T)\times H\times H$ and $s\in[t,T]$.
From Theorem \ref{theoremddp} it follows that for any
$\varepsilon>0$ there exists an admissible control
$\mathbf{u}^\varepsilon\in{\mathcal {U}}[t,s]$ such that
\begin{eqnarray}\label{3.19}
G^{t,x,\mathbf{u}^\varepsilon}_{t,s}[V(s,X^{t,x,\mathbf{u}^\varepsilon}_{s})]
+\varepsilon\geq V(t,x)\geq G^{t,x,\mathbf{u}^\varepsilon}_{t,s}[V(s,X^{t,x,\mathbf{u}^\varepsilon}_{s})].
\end{eqnarray}
Therefore,
\begin{eqnarray}\label{3.20}
|V(t,x)-V(s,y^A_{t,s})|\leq
|I_1|+|I_2|+\varepsilon,
\end{eqnarray}
where
\begin{eqnarray*}
&&I_1=\mathbb{E}G^{t,x,\mathbf{u}^\varepsilon}_{t,s}[V(s,X^{t,x,\mathbf{u}^\varepsilon}_{s})]-
\mathbb{E}G^{t,x,\mathbf{u}^\varepsilon}_{t,s}[V(s,y^A_{t,s})], \\
&&I_2=\mathbb{E}G^{t,x,\mathbf{u}^\varepsilon}_{t,s}[V(s,y^A_{t,s})]-
V(s,y^A_{t,s}).
\end{eqnarray*}
We have from Proposition 4.3 in \cite{fuh0} and  Lemmas \ref{lemmaexist0409} and \ref{lemmavaluev}  that, for some suitable constant
$C$ that is independent of the control
$\mathbf{u}^\varepsilon$ and may change from line to line,
\begin{eqnarray*}
	|I_1|&\leq& C\mathbb{E}\left|V\left(s,X^{t,x,\mathbf{u}^\varepsilon}_{s}\right)-V(s,y^A_{t,s})\right|\\
	&\leq&C\mathbb{E}\left|X^{t,x,\mathbf{u}^\varepsilon}_{s}-y^A_{t,s}\right|\leq C(1+|x|)(s-t)^{\frac{1}{2}}+C|x-y|.
\end{eqnarray*}
From the definition of
$G^{t,x,\mathbf{u}^\varepsilon}_{t,s}[\cdot]$, the
second term $I_2$ can be written as
\begin{eqnarray*}
\begin{aligned}
	I_2=& \mathbb{E}\bigg{[}V(s,y^A_{t,s})+\int^{s}_{t}
	q\left(l,X_l^{t,x,\mathbf{u}^{\varepsilon}},Y^{t,x,\mathbf{u}^{\varepsilon}}_l,Z^{t,x,\mathbf{u}^{\varepsilon}}_l,\mathbf{u}^{\varepsilon}_l\right)dl-
	\int^{s}_{t}Z^{t,x,\mathbf{u}^{\varepsilon}}_ldW_l\bigg{]}-V(s,y^A_{t,s})\\
	=&\mathbb{E}\int^{s}_{t}
	q\left(l,X_l^{t,x,\mathbf{u}^{\varepsilon}},Y^{t,x,\mathbf{u}^{\varepsilon}}_l,Z^{t,x,\mathbf{u}^{\varepsilon}}_l,\mathbf{u}^{\varepsilon}_l\right)dl,
\end{aligned}
\end{eqnarray*}
where $(Y^{t,x,\mathbf{u}^{\varepsilon}}_l,Z^{t,x,\mathbf{u}^{\varepsilon}}_l)_{t\leq l\leq
	s}$ is the solution of
(\ref{bsdegdpp}) with the terminal
condition $\zeta = V(s,y^A_{t,s})$
and the control $\mathbf{u}^\varepsilon$.
Using Schwartz inequality, we have
\begin{eqnarray*}
\begin{aligned}
	|I_2|&\leq (s-t)^{\frac{1}{2}}\bigg{[}\int^{s}_{t}
	 \mathbb{E}\left|q\left(l,X_l^{t,x,\mathbf{u}^{\varepsilon}},Y^{t,x,\mathbf{u}^{\varepsilon}}_l,Z^{t,x,\mathbf{u}^{\varepsilon}}_l,
\mathbf{u}^{\varepsilon}_l\right)\right|^2dl
	\bigg{]}^{\frac{1}{2}}\\
	&\leq C(s-t)^{\frac{1}{2}}\bigg{[}\int^{s}_{t}
	 \mathbb{E}\left(1+|X_l^{t,x,\mathbf{u}^{\varepsilon}}|^{2}+|Y^{t,x,\mathbf{u}^{\varepsilon}}_l|^2
+|Z^{t,x,\mathbf{u}^{\varepsilon}}_l|^2\right)dl
	\bigg{]}^{\frac{1}{2}}\\
	&\leq C(1+|x|)(s-t)^{\frac{1}{2}}.
\end{aligned}
\end{eqnarray*}
Hence, from (\ref{3.20}),
$$
|V(t,x)-V(s,y^A_{t,s})|\leq
C(1+|x|)(s-t)^{\frac{1}{2}}+C|x-y|
+\varepsilon,
$$
and letting $\varepsilon\downarrow
0$,  we have   (\ref{hold}).
\par
Finally, let $(t,x;s,y)\in([0,T]\times H)^2$ and $s\in[t,T]$, by (\ref{valuelip}) and (\ref{hold})
\begin{eqnarray*}
\begin{aligned}
	|V(t,x)-V(s,y)|&\leq  |V(t,x)-V(s,x^A_{t,s})|+ |V(s,x^A_{t,s})-V(s,y)|\\
	&\leq  C(1+|x|)(s-t)^{\frac{1}{2}}+C|x^A_{t,s}-y|.
\end{aligned}
\end{eqnarray*}
Thus, $V\in C^{0+}([0,T]\times H)$.
The proof is complete.
\end{proof}

\section{Viscosity solutions to  HJB equations: existence}\label{sect-exist}

\par
In this {section}, we consider the  second  order 
HJB equation (\ref{hjb1}). As usual, we begin with classical solutions.
\par
\begin{definition}\label{definitionccc}     (Classical solution)
	A functional $v\in C_p^{1,2}([0,T]\times H)$  with  $A^*\nabla_xv\in C_p^0([0,T]\times H,H)$    is called a classical solution to  HJB equation (\ref{hjb1}) if it satisfies
	HJB equation (\ref{hjb1})  {point-wisely}.
\end{definition}
\par
We shall prove that the value functional $V$ defined by (\ref{value1}) is a viscosity solution of HJB equation (\ref{hjb1}).
First, we give the following definition.
For $p>2$, $(\vartheta, \varpi)\in L^{p}_{\mathcal{P}}(\Omega\times [0,T];H\times L_2(\Xi,H))$,
and $(t,\xi)\in [0,T)\times L^p(\Omega,\mathcal{F}_t,\mathbb{P};H)$,  
then the following process
\begin{eqnarray}\label{formular1}
\begin{aligned}
X_s=e^{(s-t)A}\xi+\int^{s}_{t}e^{(s-l)A}\vartheta_ldl+\int^{s}_{t}e^{(s-l)A}\varpi_ldW_l,\quad s\in [\theta,T],
\end{aligned}
\end{eqnarray}
is well defined and $\mathbb{E}[\sup_{s\in [t,T]}|X_s|_H^{p}]<\infty$ (see~\cite[Proposition  7.3]{da}).
\begin{definition}\label{definitionc202402061a}
	Let $t\in[0,T)$ and $g:[0,T]\times H\rightarrow \mathbb{R}$ be given, and $A$ satisfy Assumption \ref{hypstate} (i').
		  We say $g\in C^{1,2}_{p,A-}([t,T]\times H)\subset C^{0,2}_p([t,T]\times H)\cap LSC^{0+}([t,T]\times H)$ if,
for the solution $X$ of (\ref{formular1}) with initial condition $(t,x)\in [0,T)\times H$,
{\begin{eqnarray}\label{statesop020240206a}
	\quad&\begin{aligned}
	g(s,X_s)\leq&\, g(t,X_t)
+\int^{s}_{t}\langle \nabla_xg(l, X_l), \, \varpi_ldW_l\rangle_H\\[3mm]
	&+\!\int_{t}^{s}\!\!\left[\langle \nabla_xg(l, X_l), \, \vartheta_l\rangle_H+\frac{1}{2}\mbox{\rm Tr}(\nabla^2_{x}g(l,X_l)(\varpi\varpi^*)_l)\right]\! dl.
	\end{aligned}
	\end{eqnarray}}
\end{definition}

Define for $t\in[0,T)$,
$$
\Phi_t:=\left\{\varphi\in C_p^{1,2}([t,T]\times H): A^*\nabla_x\varphi\in C_p^0([t,T]\times H,H)\right\}
$$
and
{\begin{eqnarray*}
	\begin{aligned}
		{\mathcal{G}}_t:=&\bigg{\{}g\in C^{0}_p([t,T]\times H):  \exists \  (h_i, g_i)\in {C^1([0,T];[0,+\infty))}\times  C^{1,2}_{p,A-}([t,T]\times H), \\
      &  \ \ \ \ i=1,\ldots, \ \mathbf{N}  \ \mbox{such that for all} \ (s,x)\in [t,T]\times H, \,  g(s,x)=\sum^{\mathbf{N}}_{i=1}h_i(s)g_i(s,x)
		\bigg{\}}.
	\end{aligned}
\end{eqnarray*}}
{
   For $g\in {\mathcal{G}}_t$, we write
\begin{eqnarray*}
\partial_{t}^og(s,x)&:=&\sum^{\mathbf{N}}_{i=1}h_i'(s)g_i(s,x).
\end{eqnarray*}
}
{
\begin{remark}\label{remarkv0129120240206ab}
\begin{description}
\item[(i)]
By the proof procedure of Lemma \ref{lemma-0902}, for every $(\hat t,\hat x)\in [0,T)\times H$,
the functional $f(s,x)=|x-\hat{x}^A_{\hat{t}, s}|\in LSC^{0+}([\hat{t},T]\times H)$.
	Then, by Lemma \ref{theoremito2}, for   $A$ satisfying Assumption  \ref{hypstate} (i'), the functional $ f^{2m}(s,x)\in C^{1,2}_{p,A-}([\hat t,T]\times H)$ for all integer $m\geq 1$. Therefore, for $h\in C^1([0,T];[0,+\infty))$, $\delta, \delta_i,\delta'_i\geq0, N>0,$ and $(\hat x,t_i,x^i)\in H\times [0,T]\times H,
		i=0,1,2,\ldots, $ such that $t_i\leq \hat t$,  and $ \sum_{i=0}^{\infty}(\delta_i+\delta'_i)\leq N$, the functional :  $g(s,x):=h(s)|x|^4+\delta|x-\hat{x}^A_{\hat{t},s}|^{2}+\sum_{i=0}^{\infty}[\delta_i|x-(x^i)^A_{\hat{t},s}|^4
		+\delta^{'}_i|s-t_i|^2
		]$,  $(s,x)\in [\hat{t},T]\times H$,  which arises in the proof of  the comparison theorem,  belongs to ${\mathcal{G}}_{\hat{t}}$.
\item[(ii)]  We note that $\Phi_t\cap{\mathcal{G}}_t\neq\emptyset$. For example,  for every $ h\in C^1([0,T];[0,+\infty))$, define
$
 g(s,x):=h(s), \ (t,x)\in [t,T]\times H$.
Then $g\in \Phi_t\cap{\mathcal{G}}_t$.
\end{description}
\end{remark}}
Now,  we define our notion of  viscosity solutions to the HJB equation (\ref{hjb1}).

\begin{definition}\label{definition4.1} A  functional
	$w\in USC^{0+}([0,T]\times H)$ (resp., $w\in LSC^{0+}([0,T]\times H)$) is called a
	viscosity sub-solution (resp., super-solution)
	to HJB equation (\ref{hjb1}) if the terminal condition,  $w(T,y)\leq   \phi(T,y)$(resp., $w(T,y)\geq   \phi(T,y)$),
	$y\in H$ is satisfied, and for any $(\varphi, g)\in  \Phi_t\times {\mathcal{G}}_t$ with $t\in [0,T)$, whenever the function $w-\varphi-g$  (resp.,  $w+\varphi+g$) satisfies for $(t,x)\in [0,T)\times H$,
	$$
	0=({w}-\varphi-g)(t,x)=\sup_{(s,y)\in [t,T]\times H}
	({w}- \varphi-g)(s,y)
	$$
	$$
	\left(\mbox{resp.,}\ \
	0=({w}+\varphi+g)(t,x)=\inf_{(s,y)\in [t,T]\times H}
	({w}+\varphi+g)(s,y)\right),
	$$
	we have
	\begin{eqnarray*}
		&&\displaystyle \varphi_{t}(t,x)+\partial_{t}^og(t,x)+\left\langle A^*\nabla_x\varphi(t,x),\, x\right\rangle_H\\[3mm]
		&&
		+{\mathbf{H}}(t,x, (\varphi+g)(t,x),\nabla_x(\varphi+g)(t,x),\nabla^2_{x}(\varphi+g)(t,x))\geq0
	\end{eqnarray*}
	\begin{eqnarray*}\begin{aligned}
			\biggl(\mbox{resp.,} &-\varphi_{t}(t,x)-\partial_{t}^og(t,x)-\left\langle A^*\nabla_x\varphi(t,x),\,  x\right\rangle_H\\[3mm]
			&+{\mathbf{H}}(t,x,
			-(\varphi+g)(t,x),-\nabla_x(\varphi+g)(t,x),-\nabla^2_{x}(\varphi+g)(t,x))
			\leq0\biggr).
		\end{aligned}
	\end{eqnarray*}
	A  functional   $w\in C^{0+}([0,T]\times H)$ is said to be a
	viscosity solution to HJB equation (\ref{hjb1}) if it is
	both a viscosity sub- and
	super-solution.
\end{definition}

\begin{remark}\label{remarkv12241}
Assuming the $B$-continuity  on the coefficients, Fabbri et al. \cite{fab1} studied  viscosity solution of  HJB equation (\ref{hjb1}). The additive  term $|\cdot|_H$ in a test function  will help   to ensure inequality (1.111) in~\cite[page 84]{fab1}, for  $\langle Ax, x\rangle_H\leq 0$ for all $x\in {\mathcal{D}}(A)$ when $A$ is the generator of a $C_0$ contraction semi-group.  While the inequality $\langle Ax, x-y\rangle_H\leq 0$ for all $x\in {\mathcal{D}}(A)$ fails to be true anymore for $y\ne0$,  and thus the introduction of the term $|\cdot-y|_H$ in a test function for fixed $y\in H$ will incur a trouble to ensure  inequality (1.111) in~\cite[page 84]{fab1}. This explains why the coefficients are assumed to satisfy  the $B$-continuity condition in~\cite[(3.21) and (3.22), page 184]{fab1}) to ensure
 the $B$-continuity of the value functional  and  the comparison theorem.
\end{remark}
\begin{remark}\label{remarkv012912024}
		As noted in Remark \ref{remarks}, the term $|x-\hat{y}^A_{\hat{t},t}|^{2m}_H$ with fixed $(\hat{t},\hat{y})\in [0,T)\times H$
can enter into the test functions for our viscosity solutions, for it satisfies
It\^o inequality (\ref{jias510815jia11}). The appearance of this term, increases some slight difficulty in the existence of viscosity solutions, but as we will see, significantly helps  us to establish the comparison principle of viscosity solutions.
Indeed,  in contrast to the $B$-continuity in~\cite[Definition 3.35, page 198]{fab1}, the viscosity solution $w$ is assumed here to be merely continuous,
 and therefore the strong continuity of the value functional in the state space $H$ is sufficient to ensure that the value functional is a viscosity solution,
  and the comparison theorem is also established without assuming the $B$-continuity  on the coefficients. In this way, 
 in our viscosity solution theory in infinite dimension, the  $B$-continuity assumption on the coefficients in~\cite[(3.21) and (3.22), page 184]{fab1} is not required at all.
\end{remark}
{
\begin{remark}\label{remarkv0129120240129e}
	Assuming the $B$-upper   and $B$-lower semi-continuity of $ {W}_1$ and $ {W}_2$, respectively, the Ekeland-Lebourg Theorem
(see ~\cite[Theorem 3.25, page 188]{fab1})  in $H_{-2}$ is used to ensure the existence of  a maximum point of the auxiliary function in the
proof of \cite[Theorem 3.50, page 206]{fab1}. Our functionals  $ {W}_1$ and $ {W}_2$ are merely strong-continuous in the Hilbert space $H$,
 and the Ekeland-Lebourg Theorem does not apply here. In fact, application in $H$ of Ekeland-Lebourg Theorem  (to yield an extremal point
 of the auxiliary function in the proof of comparison theorem) would incur a linear functional  $\langle p,\cdot\rangle_H$ (with some fixed $p\in H$),
   which does not necessarily belong to $ {\Phi}_{\hat t}$ or $ {\mathcal{G}}_{\hat t}$,  to be added to the test function. However,
the Borwein-Preiss variational principle applies,  since  its application (to get a maximum point of the auxiliary function) incurs the functional  $\sum^{\infty}_{i=0}\frac{1}{2^i}|x-(x^i)^A_{t_i,t}|^{4}$ with fixed  $\{(t_i,x^i)\}_{i\geq0}\in [0,\hat{t}]\times H$,  which belongs to $ {\mathcal{G}}_{\hat t}$,  to be added to the test functions.
\end{remark}
\begin{remark}\label{remarkv0129120240129f}
{As we stated in Remark \ref{remarkv0129120240206ab}}, the test functions in our definition of viscosity solutions {include}  the following three type functionals:  $|x|^{4}$, $\sum^{\infty}_{i=0}\frac{1}{2^i}[|x-(x^i)^A_{t_i,t}|^{4}+|t-t_i|^2]$  with fixed  $\{(t_i,x^i)\}_{i\geq0}\in [0,\hat{t}]\times H$ and $|x-\hat{x}^A_{\hat{t},t}|^{2}$ with fixed $(\hat{t},\hat{x})\in [0,T)\times H$, while the test functions in ~\cite[Definition 3.35, page 198]{fab1} consist of  only the norm functional: $|x|, x\in H$. The first functional is introduced to control the  quadratic growth condition (\ref{w2024}); as we see in Remark \ref{remarkv0129120240129e}, the second functional is introduced in our test function as a  consequence of  applying  the Borwein-Preiss variational principle; the third term is introduced in terms of the form $2|x-\hat{x}^A_{\hat{t},t}|^{2}+2|y-\hat{x}^A_{\hat{t},t}|^{2}$ to control $|x-y|^2$ in the proof of the comparison theorem.
\end{remark}}



\begin{theorem}\label{theoremvexist}
	Let Assumption \ref{hypstate}   be satisfied. Then,  the value
	functional $V$ defined by (\ref{value1}) is a
	viscosity solution to equation (\ref{hjb1}).
\end{theorem}

\begin{proof} 
First, let  $(\varphi, g)\in \Phi_{\hat{t}}\times  {\mathcal{G}}_{\hat{t}}$ with $\hat{t}\in [0,T)$
such that for  some $\hat{x}\in H$,
\begin{eqnarray}\label{0714}
0=(V-\varphi-g)(\hat{t},\hat{x})=\sup_{(s,y)\in [\hat{t},T]\times H}
(V- \varphi-g)(s,y).
\end{eqnarray}
Thus,  by the DPP of Theorem \ref{theoremddp}, we obtain that, for all ${\hat{t}}\leq {\hat{t}}+\delta\leq T$,
\begin{eqnarray}\label{4.9}
 0=V(\hat{t},\hat{x})-({{\varphi}}+g) (\hat{t},\hat{x})
=\mathop{\esssup}\limits_{\mathbf{u}\in {\mathcal{U}}[\hat{t},\hat{t}+\delta]} G^{\hat{t},\hat{x},\mathbf{u}}_{{\hat{t}},\hat{t}+\delta}\left[V\left(\hat{t}+\delta, X^{\hat{t},\hat{x},\mathbf{u}}_{{\hat{t}}+\delta}\right)\right]
-({{\varphi}}+g) (\hat{t},\hat{x}).
\end{eqnarray}
Then, for any $\varepsilon>0$ and $0<\delta\leq T-\hat{t}$,  we can  find a control  $
\mathbf{u}^{{\varepsilon},\delta}\in {\mathcal{U}}[\hat{t},\hat{t}+\delta]$ such
that
\begin{eqnarray}\label{4.10}
-{\varepsilon}\delta
\leq G^{\hat{t},\hat{x},\mathbf{u}^{{\varepsilon},\delta}}_{{\hat{t}},\hat{t}+\delta}\left[V\left(\hat{t}+\delta,X^{\hat{t},\hat{x},\mathbf{u}^{{\varepsilon},\delta}}_{{\hat{t}}+\delta}\right)\right]-({{\varphi}}+g) (\hat{t},\hat{x}).
\end{eqnarray}
We note that the process
$\left\{G^{\hat{t},\hat{x},\mathbf{u}^{{\varepsilon},\delta}}_{s,\hat{t}+\delta}\left[V\left(\hat{t}+\delta,X^{\hat{t},\hat{x},\mathbf{u}^{{\varepsilon},\delta}}_{{\hat{t}}+\delta}\right)\right],  \, s\in[\hat{t},\hat{t}+\delta]\right\}$
is the first component of  the solution $\left(Y^{\hat{t},\hat{x},\mathbf{u}^{{\varepsilon},\delta}}, \,  Z^{\hat{t},\hat{x},\mathbf{u}^{{\varepsilon},\delta}}\right)$ of the following
BSDE
\begin{eqnarray}\label{bsde4.10}
\begin{cases}
dY_s =
-q\left(s,X^{\hat{t},\hat{x},\mathbf{u}^{{\varepsilon},\delta}}_s,\, Y_s,\,  Z_s, \, u^{{\varepsilon},\delta}_s\right)\, ds
+Z_s\, dW_s, \quad   s\in[\hat{t},\hat{t}+\delta]; \\[2mm]
Y_{\hat{t}+\delta}=V\left(\hat{t}+\delta, X^{\hat{t},\hat{x},
	\mathbf{u}^{{\varepsilon},\delta}}_{{\hat{t}}+\delta}\right).
\end{cases}
\end{eqnarray}
Applying  It\^{o} formula (see \cite[Proposition 1.165]{fab1}) to ${\varphi}\left(s,X^{\hat{t},\hat{x},\mathbf{u}^{{\varepsilon},\delta}}_s\right)$ and
 inequality  (\ref{statesop020240206a}) 
to
$g\left(s,X^{\hat{t},\hat{x},\mathbf{u}^{{\varepsilon},\delta}}_s\right)$,   we get that
\begin{eqnarray}\label{bsde4.21}
\begin{aligned}
&\qquad \qquad \left({\varphi}+g\right)\left(s,X^{\hat{t},\hat{x},\mathbf{u}^{{\varepsilon},\delta}}_s\right)\leq Y^{1,\delta}_s\\
&\qquad=:\qquad  \left({\varphi}+g\right)\left(\hat{t},\hat{x}\right)+\int^{s}_{{\hat{t}}} \left({\mathcal{L}}\left({\varphi}+g\right)\right)\left(l,X^{\hat{t},\hat{x},\mathbf{u}^{{\varepsilon},\delta}}_l,
u^{{\varepsilon},\delta}_l\right)dl \\
&\qquad\qquad-\int^{s}_{{\hat{t}}}q\left(\gamma,\left({\varphi}+g\right)\left(\gamma\right), \nabla_x\left({\varphi}+g\right)\left(\gamma\right)
\sigma\left(\gamma,v\right), v\right)\Biggm|_{\scriptsize \begin{matrix}
v=u^{{\varepsilon},\delta}_l\\
\gamma=(l,X^{\hat{t},\hat{x},\mathbf{u}^{{\varepsilon},\delta}}_l)
\end{matrix}} dl\\
&\qquad\qquad
+\int^{s}_{{\hat{t}}}\left\langle \nabla_x\left({\varphi}+g\right)\left(l,X^{\hat{t},\hat{x},\mathbf{u}^{{\varepsilon},\delta}}_l\right),
\sigma\left(l,X^{\hat{t},\hat{x},\mathbf{u}^{{\varepsilon},\delta}}_l,u^{{\varepsilon},\delta}_l\right)dW_l\right\rangle_H,
\end{aligned}
\end{eqnarray}
where for  $\left(t,x,u\right)\in [0,T]\times H\times U$,
\begin{eqnarray*}
	&&\left({\mathcal{L}}\left({\varphi}+g\right)\right)\left(t,x,u\right)\\
	&:=&{{\varphi}_{t}\left(t,x\right)+{ \partial_t^o}g\left(t,x\right)}+\left\langle A^*\nabla_x {\varphi}\left(t,x\right),x\right\rangle_H
	+\left\langle\nabla_x \left({\varphi}+g\right)\left(t,x\right),b\left(t,x,u\right)\right\rangle_{H}\\
	&&+\frac{1}{2}\mbox{Tr}[\nabla_{x}^2\left({\varphi}+g\right)\left(t,x\right)\sigma\left(t,x,u\right)\sigma^*\left(t,x,u\right)]\\
	&&+
	 q\left(t,x,\left({\varphi}+g\right)\left(t,x\right),{\nabla_x\left({\varphi}+g\right)\left(t,x\right)}\sigma\left(t,x,u\right),u\right).
\end{eqnarray*}
It is clear that
\begin{eqnarray}\label{y_1}
Y^{1,\delta}_{\hat{t}}=({\varphi}+g)(\hat{t},\hat{x}).
\end{eqnarray}
Set for $s\in[\hat{t},\hat{t}+\delta]$,
\begin{eqnarray*}
	Y^{2,{\hat{t},\hat{x},\mathbf{u}^{{\varepsilon},\delta}}}_s&:=&
	Y^{1,\delta}_s-Y^{\hat{t},\hat{x},\mathbf{u}^{{\varepsilon},\delta}}_s\geq
	\left({\varphi}+g\right)\left(s,X_s^{\hat{t},\hat{x},\mathbf{u}^{{\varepsilon},\delta}}\right)-Y^{\hat{t},\hat{x},\mathbf{u}^{{\varepsilon},\delta}}_s,  \\
	Y^{3,\delta}_s&:=&
	Y^{1,\delta}_s-
	\left({\varphi}+g\right)\left(s,X_s^{\hat{t},\hat{x},\mathbf{u}^{{\varepsilon},\delta}}\right)\geq0,  \quad \hbox{\rm and}\\
	Z^{2,{\hat{t},\hat{x},\mathbf{u}^{{\varepsilon},\delta}}}_s  &:=&\nabla_x\left({\varphi}+g\right)\left(s,X^{\hat{t},\hat{x},\mathbf{u}^{{\varepsilon},\delta}}_s\right) \sigma\left(s,X^{\hat{t},\hat{x},\mathbf{u}^{{\varepsilon},\delta}}_s,u^{{\varepsilon},\delta}_s\right)-Z^{\hat{t},\hat{x},\mathbf{u}^{{\varepsilon},\delta}}_s.
\end{eqnarray*}
Comparing (\ref{bsde4.10}) and (\ref{bsde4.21}), we have, $\mathbb{P}$-a.s.,
\begin{eqnarray*}
	&&dY^{2,{\hat{t},\hat{x},\mathbf{u}^{{\varepsilon},\delta}}}_s\\
	 &=&\bigg{[}[\left({\mathcal{L}}\left({\varphi}+g\right)\right)\left(\gamma,v\right)
	 -q\left(\gamma,\left({\varphi}+g\right)\left(\gamma\right),
	\nabla_x\left({\varphi}+g\right)\left(\gamma\right)
	\sigma\left(\gamma,v\right),v\right)]\Biggm|_{\scriptsize\begin{matrix}\gamma=(s,X^{\hat{t},\hat{x},\mathbf{u}^{{\varepsilon},\delta}}_s)\\
		v=u^{{\varepsilon},\delta}_s\end{matrix}}\\
	&&+q\left(s,X^{\hat{t},\hat{x},\mathbf{u}^{{\varepsilon},\delta}}_s,Y^{\hat{t},\hat{x},\mathbf{u}^{{\varepsilon},\delta}}_s,
	Z^{\hat{t},\hat{x},\mathbf{u}^{{\varepsilon},\delta}}_s,u^{{\varepsilon},\delta}_s\right)\bigg{]}ds
	+Z^{2,{\hat{t},\hat{x},\mathbf{u}^{{\varepsilon},\delta}}}_sdW_s\\
	 &=&\bigg{[}\left({\mathcal{L}}\left({\varphi}+g\right)\right)\left(s,X^{\hat{t},\hat{x},\mathbf{u}^{{\varepsilon},\delta}}_s,u^{{\varepsilon},\delta}_s\right)
	+A_{1,\delta}(s)Y^{3,\delta}_s
	-A_{1,\delta}(s)Y^{2,{\hat{t},\hat{x},\mathbf{u}^{{\varepsilon},\delta}}}_s\\
	&&-\left\langle A_{2,\delta}(s), Z^{2,{\hat{t},\hat{x},\mathbf{u}^{{\varepsilon},\delta}}}_s\right\rangle_{\Xi}\bigg{]}ds+Z^{2,{\hat{t},\hat{x},\mathbf{u}^{{\varepsilon},\delta}}}_sdW_s,
\end{eqnarray*}
where $|A_{1,\delta}|\vee|A_{2,\delta}|\leq L$.
Applying It\^o  formula  (see~\cite[Proposition 2.2]{el}), we have
\begin{eqnarray}\label{4.14}
\begin{aligned}
&\qquad\qquad\qquad Y^{2,{\hat{t},\hat{x},\mathbf{u}^{{\varepsilon},\delta}}}_{\hat{t}}
=\mathbb{E}\bigg{[}Y^{2,{\hat{t},\hat{x},\mathbf{u}^{{\varepsilon},\delta}}}_{t+\delta}\Gamma^{\hat{t},\delta}_{\hat{t}+\delta}\\
&\qquad\qquad-
\int^{\hat{t}+\delta}_{{\hat{t}}}\!\!\!\!\! \Gamma^{\hat{t},\delta}_l
\left[\left({\mathcal{L}}\left({\varphi}+g\right)\right)
\left(l,X^{\hat{t},\hat{x},\mathbf{u}^{{\varepsilon},\delta}}_l,u^{{\varepsilon},\delta}_l\right)+A_{1,\delta}\left(l\right)Y^{3,\delta}_l\right]dl\bigg{|}
{\mathcal{F}}_{\hat{t}}\bigg{]},
\end{aligned}
\end{eqnarray}
where $\Gamma^{\hat{t},\delta}$ solves the linear SDE
$$
d\Gamma^{\hat{t},\delta}_s=\Gamma^{\hat{t},\delta}_s(A_{1,\delta}(s)ds+A_{2,\delta}(s)dW_s),\ s\in [{\hat{t}},{\hat{t}}+\delta];\ \ \ \Gamma^{\hat{t},\delta}_{\hat{t}}=1.
$$
Obviously, $\Gamma^{\hat{t},\delta}\geq 0$. Combining (\ref{4.10}), (\ref{y_1}) and (\ref{4.14}), we have
\begin{eqnarray}\label{4.15}
\begin{aligned}
-\varepsilon&\leq \frac{1}{\delta}[Y^{\hat{t},\hat{x},\mathbf{u}^{{\varepsilon},\delta}}_{\hat{t}}-({\varphi}+g)(\hat{t},\hat{x})]
=\frac{1}{\delta}(Y^{\hat{t},\hat{x},\mathbf{u}^{{\varepsilon},\delta}}_{\hat{t}}-Y^{1,\delta}_{\hat{t}})
=-\frac{1}{\delta}Y^{2,{\hat{t},\hat{x},\mathbf{u}^{{\varepsilon},\delta}}}_{\hat{t}}\\
&= \frac{1}{\delta}\mathbb{E}\bigg{[}-Y^{2,{\hat{t},\hat{x},\mathbf{u}^{{\varepsilon},\delta}}}_{\hat{t}+\delta}\Gamma^{\hat{t},\delta}_{\hat{t}+\delta}\\
&\quad
+\int^{\hat{t}+\delta}_{{\hat{t}}}\!\!\!\!\! \Gamma^{\hat{t},\delta}_l\left[({\mathcal{L}}({\varphi}+g))\left(l,X^{\hat{t},\hat{x},\mathbf{u}^{{\varepsilon},\delta}}_l,u^{{\varepsilon},\delta}_l\right)
+A_{1,\delta}(l)Y^{3,\delta}_l\right]dl\bigg{]}\\[3mm]
&:=I_1+I_2+I_3+I_4+I_5
\end{aligned}
\end{eqnarray}
with the five terms
\begin{eqnarray*}
	 I_1&:=&-\frac{1}{\delta}\mathbb{E}\bigg{[}Y^{2,{\hat{t},\hat{x},\mathbf{u}^{{\varepsilon},\delta}}}_{\hat{t}+\delta}\Gamma^{\hat{t},\delta}_{\hat{t}+\delta}\bigg{]}, \\
	 I_2&:=&\frac{1}{\delta}\mathbb{E}\bigg{[}\int^{\hat{t}+\delta}_{{\hat{t}}}\!\!\!({\mathcal{L}}({\varphi}+g))(\hat{t},\hat{x},u^{{\varepsilon},\delta}_l)\, dl\bigg{]},\nonumber\\
	 I_3&:=&\frac{1}{\delta}\mathbb{E}\bigg{[}\int^{\hat{t}+\delta}_{{\hat{t}}}\!\!\left[({\mathcal{L}}({\varphi}+g))
\left(l,X^{\hat{t},\hat{x},u^{{\varepsilon},\delta}}_l,u^{{\varepsilon},\delta}_l\right)-
	({\mathcal{L}}({\varphi}+g))\left(\hat{t},\hat{x},u^{{\varepsilon},\delta}_l\right)\right]dl\bigg{]},\nonumber\\
	 I_4&:=&\frac{1}{\delta}\mathbb{E}\bigg{[}\int^{\hat{t}+\delta}_{{\hat{t}}}\!\!\!\!\!(\Gamma^{\hat{t},\delta}_l-1)
({\mathcal{L}}({\varphi}+g))\left(l,X^{\hat{t},\hat{x},\mathbf{u}^{{\varepsilon},\delta}}_l,
	u^{{\varepsilon},\delta}_l\right)dl\bigg{]}, \quad \hbox{\rm and }\nonumber\\
	 I_5&:=&\frac{1}{\delta}\mathbb{E}\bigg{[}\int^{\hat{t}+\delta}_{{\hat{t}}}\!\!\!\Gamma^{\hat{t},\delta}_lA_{1,\delta}(l)Y^{3,\delta}_l\, dl\bigg{]}. \nonumber
\end{eqnarray*}
Since the coefficients of the operator ${\mathcal{L}}$  have a  linear growth,
combining the regularity of $(\varphi, g)\in \Phi_{\hat{t}}\times {\mathcal{G}}_{\hat{t}}$, there exist a integer
$\bar{p}\geq1$ and a constant $C>0$ independent of $u\in U$ such that, for all $(t,x,{u})\in [0,T]\times H\times U$,
\begin{eqnarray}\label{4.4444}|({\varphi}+g)(t,x)|\vee  |
({\mathcal{L}}({\varphi}+g))(t,x,u)|
\leq  C(1+|x|_H)^{\bar{p}}.
\end{eqnarray}
In view of Lemma \ref{lemmaexist0409}, we also have
\begin{eqnarray*}
	\sup_{l\in[\hat{t},\hat{t}+\delta]}\mathbb{E}|\Gamma^{\hat{t},\delta}_l-1|^2\leq C\delta.
\end{eqnarray*}
Thus, in view of (\ref{0714}) and (\ref{bsde4.21}), we have
\begin{eqnarray}\label{4.1611}
I_1&=& -\frac{1}{\delta}\mathbb{E}\left[\left(Y^{1,\delta}_{\hat{t}+\delta}-Y^{\hat{t},\hat{x},\mathbf{u}^{{\varepsilon},\delta}}_{\hat{t}+\delta}\right)
\Gamma^{\hat{t},\delta}_{\hat{t}+\delta}\right]\nonumber\\
&\leq&-\frac{1}{\delta}\mathbb{E}\left[\left(({\varphi}+g)\left(\hat{t}+\delta,X_{\hat{t}+\delta}^{\hat{t},\hat{x},\mathbf{u}^{{\varepsilon},\delta}}\right)
-Y^{\hat{t},\hat{x},\mathbf{u}^{{\varepsilon},\delta}}_{\hat{t}+\delta}\right)
\Gamma^{\hat{t},\delta}_{\hat{t}+\delta}\right]\nonumber\\
&=&\frac{1}{\delta}\mathbb{E}\left[\left(V\left(\hat{t}+\delta,X^{\hat{t},\hat{x},
	 u^{{\varepsilon},\delta}}_{{\hat{t}}+\delta}\right)-({\varphi}+g)\left(\hat{t}+\delta,X_{\hat{t}+\delta}^{\hat{t},\hat{x},\mathbf{u}^{{\varepsilon},\delta}}\right)\right)
\Gamma^{\hat{t},\delta}_{\hat{t}+\delta}\right]\leq 0
\end{eqnarray}
and
\begin{eqnarray}\label{4.16}
I_2&\leq&\frac{1}{\delta}\bigg{[}\int^{\hat{t}+\delta}_{{\hat{t}}}\!\!\!\sup_{u\in U}({\mathcal{L}}({\varphi}+g))({\hat{t},\hat{x},{u}})\, dl\bigg{]}
= {\varphi}_{t}(\hat{t},\hat{x})+{ \partial_t^o}g(\hat{t},\hat{x})+\left\langle A^*\nabla_x{\varphi}(\hat{t},\hat{x}),\hat{x}\right\rangle_H\nonumber\\[3mm]
&&\quad +\, {\mathbf{H}}(\hat{t},\hat{x},\, ({\varphi}+g)(\hat{t},\hat{x}),\, \nabla_x({\varphi}+g)(\hat{t},\hat{x}),\,
\nabla^2_{x}({\varphi}+g)(\hat{t},\hat{x})).
\end{eqnarray}
Now we prove that other three terms  $I_3$, $I_4$ and $I_5$ are of higher order.  In view of  (\ref{2.6}),
$$
\lim_{\delta\rightarrow0}\mathbb{E}\left[\sup_{\hat{t}\leq l\leq \hat{t}+\delta}
|X^{\hat{t},\hat{x},\mathbf{u}^{{\varepsilon},\delta}}_{l}-\hat{x}|^{\bar{p}}\right]\leq
C\lim_{\delta\rightarrow0}\mathbb{E}\left[\sup_{\hat{t}\leq l\leq \hat{t}+\delta}\left[
|X^{\hat{t},\hat{x},\mathbf{u}^{{\varepsilon},\delta}}_{l}-\hat{x}^A_{\hat{t},l}|+|\hat{x}^A_{\hat{t},l}-\hat{x}|\right]^{\bar{p}}\right]=0.
$$
Then,  in view of (\ref{4.4444}), using the dominated convergence theorem, we have
$$
\lim_{\delta\rightarrow0}\, \!\!\sup_{\hat{t}\leq l\leq\hat{t}+ \delta}\mathbb{E}\left|({\mathcal{L}}({\varphi}+g))\left(l,X^{{\hat{t},\hat{x},\mathbf{u}^{{\varepsilon},\delta}}}_l,u^{{\varepsilon},\delta}_l\right)-
({\mathcal{L}}({\varphi}+g))(\hat{t},\hat{x},u^{{\varepsilon},\delta}_l)\right|=0\quad \mbox{and}
$$
\begin{eqnarray*}
	&&\lim_{\delta\rightarrow0} \sup_{\hat{t}\leq l\leq \hat{t}+\delta}\mathbb{E}\left|\Gamma^{\hat{t},\delta}_lA_{1,\delta}(l)Y^{3,\delta}_l\right|\\
	&\leq& L\lim_{\delta\rightarrow0} \sup_{\hat{t}\leq l\leq \hat{t}+\delta} \mathbb{E}\left[\left|\Gamma^{\hat{t},\delta}_l\right|\left(|Y^{1,\delta}_l-(\varphi+g)(\hat{t},\hat{x})|
	+\left|\varphi+g\right|(\gamma)\biggm|^{\gamma=(l,X^{\hat{t},\hat{x},\mathbf{u}^{{\varepsilon},\delta}}_l)}_{\gamma=(\hat{t},\hat{x})}\right)\right] \\
	&=&\, 0.
\end{eqnarray*}
Furthermore,  we have
\begin{eqnarray}\label{4.18}
\qquad \lim_{\delta\rightarrow0}|I_3|
\leq \lim_{\delta\rightarrow0}\sup_{\hat{t}\leq l\leq \hat{t}+\delta}\mathbb{E}\bigg{|}&{\mathcal{L}}({\varphi}+g)(\gamma,v)\biggm|^{(\gamma,v)=\left(l,X^{{\hat{t},\hat{x},\mathbf{u}^{{\varepsilon},\delta}}}_l,\,\,  u^{{\varepsilon},\delta}_l\right)}_{(\gamma,v)=(\hat{t},\hat{x},\,\, u^{{\varepsilon},\delta}_l)}\bigg{|}=0
\end{eqnarray}
and
\begin{eqnarray}\label{4.180714}
\lim_{\delta\rightarrow0}|I_5| \leq\lim_{\delta\rightarrow0}\sup_{\hat{t}\leq l\leq \hat{t}+\delta}\mathbb{E}\left|\Gamma^{\hat{t},\delta}_lA_{1,\delta}(l)Y^{3,\delta}_l\right|=0.
\end{eqnarray}
Moreover, for some finite constant $C>0$,
\begin{eqnarray}\label{4.19}
\begin{aligned}
&\qquad |I_4|
\leq\frac{1}{\delta}\int^{\hat{t}+\delta}_{{\hat{t}}}\mathbb{E}|\Gamma^{\hat{t},\delta}_l-1|
\left|({\mathcal{L}}({\varphi}+g))(l,X^{\hat{t},\hat{x},\mathbf{u}^{{\varepsilon},\delta}}_l,
u^{{\varepsilon},\delta}_l)\right|
dl\\
&\qquad \leq\frac{1}{\delta}\!\!\int^{\hat{t}+\delta}_{{\hat{t}}}\!\!\! \left(\mathbb{E}\left(\Gamma^{\hat{t},\delta}_l-1\right)^2\right)^{\frac{1}{2}}\!\!\! \left(\mathbb{E}\left[\left({\mathcal{L}}\left({\varphi}+g\right)\right)\left(l,X^{\hat{t},\hat{x},
	\mathbf{u}^{{\varepsilon},\delta}}_l,
u^{{\varepsilon},\delta}_l\right)\right]^2\right)^{\frac{1}{2}}\!\! dl
\\
&\qquad\leq C(1+|\hat{x}|_H)^{\bar{p}}\delta^\frac{1}{2}.
\end{aligned}
\end{eqnarray}
Substituting  (\ref{4.1611}), (\ref{4.16}) and (\ref{4.19}) into (\ref{4.15}), we have
\begin{eqnarray}\label{4.2000000}
\begin{aligned}
-\varepsilon&\leq \varphi_{t}(\hat{t},\hat{x})+ \partial_t^og(\hat{t},\hat{x}) +\left\langle A^*\nabla_x\varphi(\hat{t},\hat{x}),\,  \hat{x}\right\rangle_H\\
&\quad+\mathbf{H}(\hat{t},\hat{x},\, (\varphi+g)(\hat{t},\hat{x}),\, \nabla_x(\varphi+g)(\hat{t},\hat{x}),
\, \nabla^2_{x}(\varphi+g)(\hat{t},\hat{x}))\\
&\quad+|I_3|+|I_5|
+C(1+|\hat{x}|_H)^{\bar{p}}\,\delta^\frac{1}{2}.
\end{aligned}
\end{eqnarray}
Sending $\delta$ to $0$,  in view of  (\ref{4.18}) and (\ref{4.180714}),  we have
\begin{eqnarray*}
	\begin{aligned}
		-\varepsilon
		&\leq\varphi_{t}(\hat{t},\hat{x})+\partial_t^og(\hat{t},\hat{x})+\left\langle A^*\nabla_x{\varphi}\left(\hat{t},\hat{x}\right), \hat{x}\right\rangle_H\\
		 &\quad+{\mathbf{H}}\left(\hat{t},\hat{x},\, \left({\varphi}+g\right)\left(\hat{t},\hat{x}\right),\, \nabla_x\left({\varphi}+g\right)\left(\hat{t},\hat{x}\right),\,
		\nabla^2_{x}\left({\varphi}+g\right)\left(\hat{t},\hat{x}\right)\right).
	\end{aligned}
\end{eqnarray*}
Since $\varepsilon$ is arbitrary, the value functional
$V$ is  a viscosity sub-solution to (\ref{hjb1}).
\par
In a symmetric (even easier) way, we show that  $V$ is also a viscosity super-solution to equation (\ref{hjb1}).
The proof is complete.
\end{proof}


{Now, let us give the consistency and stability results for the viscosity
solutions.  }
\begin{theorem}\label{theorem1223}
	Let Assumption \ref{hypstate} hold true, $v\in C_p^{1,2}([0,T]\times H)$ and $A^*\nabla_xv\in C_p^0([0,T]\times H,H)$. Then
	$v$ is a classical solution of  equation (\ref{hjb1}) if and only if it is a viscosity solution.
\end{theorem}

\begin{theorem}\label{theoremstability}
	Let $b,\sigma,q,\phi$ satisfy Assumption \ref{hypstate}, and $v\in C^{0+}([0,T]\times H)$. Assume that
	\item[(i)]      for any $\varepsilon>0$, there exist $b^\varepsilon, \sigma^\varepsilon, q^\varepsilon, \phi^\varepsilon$ and $v^\varepsilon\in C^{0+}([0,T]\times H)$ such that  $b^\varepsilon, \sigma^\varepsilon, q^\varepsilon, \phi^\varepsilon$ satisfy  Assumption \ref{hypstate} and $v^\varepsilon$ is a viscosity  {sub-solution} (resp.,  {super-solution}) of HJB equation (\ref{hjb1}) with generators $b^\varepsilon, \sigma^\varepsilon, q^\varepsilon, \phi^\varepsilon$;
	\item[(ii)] as $\varepsilon\rightarrow0$, $(b^\varepsilon, \sigma^\varepsilon, q^\varepsilon, \phi^\varepsilon,v^\varepsilon)$ converge to
	$(b, \sigma, q, \phi, v)$  uniformly in the following sense: 
	\begin{eqnarray}\label{sss}
&&	\lim_{\varepsilon\rightarrow 0}\sup_{\scriptsize\begin{matrix} (t,x,y,z,u)\\
		\in [0,T]\times H \times \mathbb{R}\times H\times U\end{matrix}}[(|b^\varepsilon-b|+|\sigma^\varepsilon-\sigma|)(t,x,u)
	\nonumber\\
	&&~~~~~+|q^\varepsilon-q|(t,x,{y},z\sigma(t,x,u),u)+|\phi^\varepsilon-\phi|(x)+|v^\varepsilon-v|(t,x)]=0.
	\end{eqnarray}
	Then,  $v$ is a viscosity  {sub-solution} (resp., super-solution) of HJB equation (\ref{hjb1}) with generators $(b,\sigma,q,\phi)$.
\end{theorem}

{The proof of Theorems \ref{theorem1223} and \ref{theoremstability} is rather standard.  For the sake
of the completeness of the article and the convenience of readers, we postpone their proof to \ref{appendixd}.}

\section{Viscosity solutions to  PHJB equations: Crandall-Ishii lemma}
\par
In this {section},  we extend Crandall-Ishii lemma to functionals  belonging to  $ USC^{0+}([0,T]\times H)$. 
It is the cornerstone of the theory of viscosity solutions, and is the key  to prove the comparison theorem  that will be given in the next {section}. \par
Let$\{e_i\}_{i\geq1}$ be an orthonormal basis of $H$ such that $e_i\in {\mathcal{D}}(A^*)$ for all $i\geq 1$. For every $N\geq 1$, we denote by $H_N$ the vector space generated by vectors $e_1,\ldots,e_N$, and by $P_N$  the orthogonal projection onto $H_N$. Define $Q_N:=I-P_N$.  We have
the orthogonal decomposition $H=H_N+H^\bot_N$ with $H^\bot_N:=Q_NH$. We denote by $x_N,y_N, \ldots$ points in $H_N$ and by $x^-_N,y^-_N,\ldots$ points in $H^\bot_N$, and write $x=(x_N,x^-_N)$ for $x\in H$.
\begin{definition}\label{definition0513}  For every $N\geq 1$, let $(\hat{t},\hat{x}_N)\in (0,T)\times H_N$ and $f\in USC^0([0,T]\times H_N)$
 bounded from above. We say $f\in \Phi_N(\hat{t},\hat{x}_N)$ if there is a constant $r>0$ such that,
	for every $L>0$, there is  a constant {$\tilde{C}_0\geq0$} depending only on $L$ such that, for every  function $\varphi\in C^{1,2}([0,T]\times H_N)$ such that the function
	$$f(s,y_N)-\varphi(s,y_N), \quad [0,T]\times H_N$$
	 attains a  maximum  at a point $(t,x_N)\in (0,T)\times H_N$, 
	if   
	\begin{eqnarray*}
		|t-\hat{t}|+|x_N-\hat{x}_N|<r,\quad
		(|f|+|\nabla_x\varphi|
		+|\nabla^2_x\varphi|)(t,x_N)\leq L,
	\end{eqnarray*}
	then
	\begin{eqnarray}\label{0528110}
	{\varphi}_{t}(t,x_N) \geq -\tilde{C}_0.
	\end{eqnarray}
\end{definition}
Write  $e^{lA^*}:=(e^{lA})^*$. Then,  $e^{lA^*}$ is a $C_0$ {semi-group} on $H$ generated by $A^*$.
Notice that, for all $0\leq t\leq s\leq T$ and $x,y\in H$, $i\geq 1$,
$$
\langle x, e_i\rangle-\langle y, e_i\rangle=\langle x-x^A_{t,s}, e_i\rangle+\langle x^A_{t,s}-y, e_i\rangle
=\langle x, e_i-e^{(s-t)A^*}e_i\rangle+\langle x^A_{t,s}-y, e_i\rangle;
$$
and for all $0\leq s< t\leq T$ and $x,y\in H$,
$$
\langle x, e_i\rangle-\langle y, e_i\rangle=\langle x-y^A_{s,t}, e_i\rangle+\langle y^A_{s,t}-y, e_i\rangle
=\langle x-y^A_{s,t}, e_i\rangle+\langle y, e^{(s-t)A^*}e_i-e_i\rangle.
$$
We have,  for $\varphi:[0,T]\times H_N\rightarrow \mathbb{R}$, $\varphi\in C^{0+}([0,T]\times H_N)$ if and only
if  $\varphi\in C^{0}([0,T]\times H_N)$. Therefore, we can write $\varphi\in C^{1,2}([0,T]\times H_N)$ without any confusion.
%
\begin{definition}\label{definition06072024}
	Let $\hat{t}\in [0,T)$ be fixed and  $w\in USC^0([t,T]\times H)$
 bounded from above.
	For every $N\geq 1$, define,  for $(t,x_N)\in [0,T]\times H_N$,
	\begin{eqnarray*}
		&&\widetilde{w}^{\hat{t},N}(t,x_N):=\sup_{y_N=x_N}
		w(t,y), \ \   t\in [\hat{t},T];\\
		&&\widetilde{w}^{\hat{t},N}(t,x_N):=\widetilde{w}^{\hat{t},N}(\hat{t},x_N)-(\hat{t}-t)^{\frac{1}{2}}, \ \   t\in[0,\hat{t}\,).
	\end{eqnarray*}
	Let $\widetilde{w}^{\hat{t},N,*}$ be the upper
	semi-continuous envelope of $\widetilde{w}^{\hat{t},N}$ (see    
	\cite[Definition D.10]{fab1}), i.e.,
	$$
	\widetilde{w}^{\hat{t},N,*}(t,x_N)=\limsup_{\scriptsize\begin{matrix} (s,y_N)\in [0,T]\times H_N\\ (s,y_N)\rightarrow(t,x_N)\end{matrix}}\widetilde{w}^{\hat{t},N}(s,y_N).
	$$
\end{definition}
In what follows, by a $modulus \ of \ continuity$, we mean a continuous function $\rho_1:[0,\infty)\rightarrow[0,\infty)$, with $\rho_1(0)=0$; by a $local\ modulus\ of $ $ continuity$, we mean  a continuous function $\rho_1:[0,\infty)\times[0,\infty) \rightarrow[0,\infty)$ such that the function $\rho_1(\cdot,r)$ is a modulus of continuity for each $r\geq0$ and $\rho_1$ is non-decreasing in the  second variable.
\begin{theorem}\label{theorem0513} (Crandall-Ishii lemma)\ \    Let $N\geq1$. Let $w_1,w_2\in USC^{0+}([0,T]\times H)$
bounded from above and such that
	\begin{eqnarray}\label{051312024}
	\limsup_{|x|\rightarrow\infty}\sup_{t\in[0,T]}\left[\frac{w_1(t,x)}{|x|}\right]<0;\ \ \  \limsup_{|x|\rightarrow\infty}\sup_{t\in[0,T]}\left[\frac{w_2(t,x)}{|x|}\right]<0.
	\end{eqnarray}
	Let $\varphi\in C^2( H_N\times H_N)$ be such that
	$$
	w_1(t,x)+w_2(t,y)-\varphi(x_N,y_N)
	$$
	has a strict
	maximum over $[\hat{t},T]\times H\times H$ at a point $(\hat{t},\hat x, \hat y)$ with $\hat{t}\in (0,T)$ and there exists a  strictly monotone increasing function ${\tilde{\rho}}:[0,\infty)\rightarrow[0,\infty)$ with $\tilde{\rho}(0)=0$ such that, for all $(t,x,y)\in [\hat{t},T]\times H\times H$,
	\begin{eqnarray}\label{101012024}
	\begin{aligned}
	&\quad w_1(\hat{t}, \hat x)+w_2(\hat{t}, \hat y)-\varphi(\hat{x}_N,\hat{y}_N)\\
	&\geq w_1(t,x)+w_2(t,y)-\varphi(x_N,y_N)+\tilde{\rho}\left(|t-\hat{t}|^2+\left|x-\hat{x}^A_{\hat t,t}\right|^4+\left|y-\hat{y}^A_{\hat t, t}\right|^4\right).
	\end{aligned}
	\end{eqnarray} 
	Assume, moreover, $\widetilde{w}_{1}^{\hat{t},N,*}\in \Phi_N(\hat{t},\hat{x}_N)$ and $\widetilde{w}_{2}^{\hat{t},N,*}\in \Phi_N(\hat{t},\hat{y}_N)$, and there exists a local modulus of continuity  $\rho_1$  such that, for all  $\hat{t}\leq t\leq s\leq T, \ x\in H$ and $x_{t,s}^{A,N}:=(e^{(s-t)A}x)_N^{-}+x_N$,
	\begin{eqnarray}\label{0608a2024}
	\begin{aligned}
	w_i(t,x)-w_i(s,x_{t,s}^{A,N})\leq \rho_1(|s-t|,|x|),\quad i=1,2.
	\end{aligned}
	\end{eqnarray}
	Then for every $\kappa>0$, there exist
	sequences  $(t_{k},x^{k}; s_{k},y^{k})\in \left([\hat{t},T]\times H\right)^2$ and
	sequences of functions $(\varphi_{1,k}, \psi_{1,k}, \varphi_{2,k}, \psi_{2,k})\in \Phi_{\hat{t}}\times \Phi_{\hat{t}}\times  {\mathcal{G}}_{t_k}\times {\mathcal{G}}_{s_k}$ bounded from below   
	such that 
	$$
	(w_{1}-\varphi_{1,k}-\varphi_{2,k})(t,x)
	$$
	has a strict  maximum $0$ at  $(t_k,x^{k})$ over $[{t_k},T]\times H$,
	$$
	(w_{2}-\psi_{1,k}-\psi_{2,k})(t,y)
	$$
	has a strict  maximum $0$ at $({s_k},y^k)$ over $[{s_k},T]\times H$, and
	\begin{eqnarray}\label{0608v2024}
	&&\left(t_{k}, x^{k};\,  (w_1, (\varphi_{1,k})_{t},\nabla_x\varphi_{1,k},\nabla^2_{x}\varphi_{1,k}, \partial_{t}^o\varphi_{2,k},\nabla_x\varphi_{2,k},\nabla^2_{x}\varphi_{2,k})(t_{k},{x}^{k})\right)\nonumber\\
	&&\underrightarrow{k\rightarrow\infty}\, \left({\hat{t}},\hat{x};\,  w_1(\hat{t},\hat{x}), b_1, \nabla_{x_1}\varphi(\hat x_N,\hat y_N), X_N, 0,\mathbf{0},\mathbf{0}\right),
	\end{eqnarray}
	\begin{eqnarray}\label{0608vw2024}
	&&\left(s_{k}, y^{k}; \, (w_2,(\psi_{1,k})_{t},\nabla_x\psi_{1,k},\nabla^2_{x}\psi_{1,k}, \partial_{t}^o\psi_{2,k},\nabla_x\psi_{2,k},\nabla^2_{x}\psi_{2,k})(s_{k},y^{k})\right)\nonumber\\
	&&\underrightarrow{k\rightarrow\infty}\, \left({\hat{t}},\hat{y};\,  w_2(\hat{t}, \hat{y}),  b_2, \nabla_{x_2}\varphi(\hat{x}_N,\hat{y}_N), Y_N, 0,\mathbf{0},\mathbf{0}\right),
	\end{eqnarray}
	where $b_{1}+b_{2}=0$ and $X_N, Y_N\in \mathcal{S}(H_N)$ satisfying  the following inequality
	\begin{eqnarray}\label{II0615}
	-\left(\frac{1}{\kappa}+|B|\right)I\leq
	\left(\begin{matrix}
	X_N&0\\
	0&Y_N \end{matrix}\right)\leq B+\kappa B^2
	\end{eqnarray}
for $B:=\nabla^2_{x}\varphi(\hat{x}_N,\hat{y}_N). $
Here,  $\nabla^2_{x}\varphi$  is  the Hessian of $\varphi$
	with respect to  the  variable $x=(x_1,x_2)\in H_N\times H_N$, and $\nabla_{x_1}\varphi$ and $\nabla_{x_2}\varphi$ are  gradients of $\varphi$
	with respect to  the first and second variables, respectively.  
\end{theorem}

\begin{proof}
From Lemma \ref{lemma4.30615} below, we know that the function
\begin{eqnarray}\label{05211}
\begin{aligned}
&\widetilde{w}_{1}^{\hat{t},N,*}(t,x_N)+ \widetilde{w}_{2}^{\hat{t},N,*}(t,y_N)-\varphi(x_N,y_N), \quad
(t,x_N,y_N) \in [0,T]\times H_N\times H_N \\
& \mbox{ is maximized at} \ (\hat{t},\hat{x}_N,\hat{y}_N).
\end{aligned}
\end{eqnarray}
Moreover, we have
$\widetilde{w}_{1}^{\hat{t},N,*}(\hat{t},\hat{x}_N)={w}_{1}(\hat{t},\hat{x})$, $\widetilde{w}_{2}^{\hat{t},N,*}(\hat{t},\hat{y}_N)={w}_{2}(\hat{t},{y})$.
Then, by $(\widetilde{w}_{1}^{\hat{t},N,*},$ $ \widetilde{w}_{2}^{\hat{t},N,*})\in \Phi_N(\hat{t},\hat{x}_N) \times \Phi_N(\hat{t},\hat{y}_N)$ and Remark \ref{remarkv0528}, Theorem 8.3 in \cite{cran2} and
Lemma 5.4 of Chapter 4 in \cite{yong} yield the existence of  sequences of   functions
$\tilde{\varphi}_k,\tilde{\psi}_k\in C^{1,2}([0,T]\times  H_N)$  bounded from below such that both functions
$$\widetilde{w}_{1}^{\hat{t},N,*}(t, x_N)-\tilde{\varphi}_k(t,x_N), \quad (t,x_N)\in [0,T]\times H_N,\ \mbox{and}$$
$$\widetilde{w}_{2}^{\hat{t},N,*}(s,y_N)-\tilde{\psi}_k(s,y_N), \quad (s,y_N)\in [0,T]\times H_N$$
 attain a strict  maximum $0$ at some points  $(t_k,x^{k}_N)\in (0,T)\times  H_N$  and $(s_k,y^{k}_N)\in (0,T)\times  H_N$, respectively, and moreover,
\begin{eqnarray}\label{0607a}
&&\lim_{k\to \infty}\left(t_{k}, x^{k}_N, \widetilde{w}_{1}^{\hat{t},N,*}(t_{k}, x^{k}_N), (\tilde{\varphi}_k)_t(t_{k}, x^{k}_N),\nabla_x\tilde{\varphi}_k(t_{k}, x^{k}_N),\nabla^2_x\tilde{\varphi}_k(t_{k}, x^{k}_N)\right)\nonumber\\
&&\qquad=\left({\hat{t}},\hat{x}_N, w_1(\hat{t},\hat{x}),b_1, \nabla_{x_1}\varphi(\hat{x}_N,\hat{y}_N), X_N\right)
\end{eqnarray}
and
\begin{eqnarray}\label{0607b}
&&\lim_{k\to \infty}\left(s_{k}, y^{k}_N, \widetilde{w}_{2}^{\hat{t},N,*}(s_{k}, y^{k}_N),(\tilde{\psi}_k)_t(s_{k}, y^{k}_N),\nabla_x\tilde{\psi}_k(s_{k}, y^{k}_N),\nabla^2_x\tilde{\psi}_k(s_{k}, y^{k}_N)\right)\nonumber\\
&&\qquad =\left({\hat{t}},\hat{y}, w_2(\hat{t},\hat{y}),b_2, \nabla_{x_2}\varphi(\hat{x}_N,\hat{y}_N), Y_N\right),
\end{eqnarray}
with $b_{1}+b_{2}=0$ and the inequality~(\ref{II0615}) being satisfied.

We assert that   $(t_{k}, s_{k})\in [\hat{t},T)\times [\hat{t},T)$ for all $k\ge 1$. Otherwise, for example, there exists a subsequence of $\{t_{k}\}_{k\geq1}$ still denoted by itself such that $t_{k}<\hat{t}$ for all $k\geq1$.
Since the function
$$\widetilde{w}_{1}^{\hat{t},N,*}(t, x_N)-\tilde{\varphi}_k(t,x_N), \quad (t,x_N)\in [0,T]\times  H_N $$
attains the maximum at $(t_{k},x^{k}_N)$, we obtain that
$$
(\tilde{\varphi}_k)_t(t_{k},x^{k}_N)={\frac{1}{2}}(\hat{t}-t_{k})^{-\frac{1}{2}}\rightarrow\infty,\ \mbox{as}\ k\rightarrow\infty,
$$
which  contradicts  that $(\tilde{\varphi}_k)_t(t_{k},x^{k}_N)\rightarrow b_{1}\in \mathbb{R}$. 
\par
Now we consider the functionals,
for $(t,x), (s,y)\in [\hat{t},T]\times H$,
\begin{eqnarray}\label{4.11116}
\quad\quad\quad       \Gamma^1_{k}(t,x):=w_{1}(t,x)-\tilde{\varphi}_k(t,x_N),\  \Gamma^2_{k}(s,y):= w_{2}(s,y)
-\tilde{\psi}_k(s,y_N).
\end{eqnarray}
Clearly,
 $\Gamma^1_{k}, \Gamma^2_{k} \in USC^{0+}([\hat{t},T]\times H)$   bounded from above and satisfy (\ref{051312024}).
Take
$\delta_i:=\frac{1}{2^i}$ for all $i\geq0$. For every $k$ and $j>0$,
applying  Lemma \ref{theoremleft} to $\Gamma^1_{k}$ and  $\Gamma^2_{k}$, respectively, we have that,
for every  $(\check{t}_{0},\check{x}^{0}; \check{s}_{0},\check{y}^{0})\in ([\hat{t},T]\times  H)^2$ satisfying
\begin{eqnarray}\label{20210509a}
\begin{aligned}
&
\Gamma^1_{k}(\check{t}_{0},\check{x}^{0})\geq \sup_{(t,x)\in [\hat{t},T]\times H}\Gamma^1_{k}(t,x)-\frac{1}{j},\\
&\Gamma^2_{k}(\check{s}_{0},\check{y}^{0})\geq \sup_{(s,y)\in [\hat{t},T]\times H}\Gamma^2_{k}(s,y)-\frac{1}{j},
\end{aligned}
\end{eqnarray}
there exist $(t_{k,j},{x}^{k,j}; s_{k,j},{y}^{k,j})\in ([\hat{t},T]\times H)^2,\,  j\ge 1$, two sequences $\{(\check{t}_{i},\check{x}^{i})\}_{i\geq1}\subset
[\check{t}_{0},T]\times H$, $\{(\check{s}_{i},\check{y}^{i})\}_{i\geq1}\subset
[\check{s}_{0},T]\times H$ and constant $C>0$ such that
\begin{description}
	\item[(i)] 
	$\Upsilon((\check{t}_i,\check{x}^{i}),(t_{k,j},{x}^{k,j}))
	\vee\Upsilon((\check{s}_i,\check{y}^{i}),(s_{k,j},{y}^{k,j}))\leq \frac{1}{2^ij}$, $|\check{x}^{i}|\vee|\check{y}^{i}|\leq C$ for all $i\geq0$,
	and $\displaystyle \check{t}_{i}\uparrow t_{k,j}$, $\check{s}_{i}\uparrow s_{k,j}$ as $i\rightarrow\infty$,
	\item[(ii)]  $\displaystyle \mathbf{\Gamma}^1_k(t_{k,j},{x}^{k,j}):=\Gamma^1_k(t_{k,j},{x}^{k,j})
	-\sum_{i=0}^{\infty}\frac{1}{2^i}\Upsilon((\check{t}_i,\check{x}^{i}),(t_{k,j},{x}^{k,j}))
	\geq \Gamma^1_{k}(\check{t}_0,\check{x}^{0})$,\\
	\ \ \ $\displaystyle\mathbf{ \Gamma}^2_k(s_{k,j},{y}^{k,j}):=\Gamma^2_k(s_{k,j},{y}^{k,j})
	-\sum_{i=0}^{\infty}\frac{1}{2^i}\Upsilon((\check{s}_i,\check{y}^{i}),(s_{k,j},{y}^{k,j}))\geq \Gamma^2_{k}(\check{s}_0,\check{y}^{0})$,
	and
	\item[(iii)]   $\mathbf{\Gamma}^1_{k}(t,x)<\mathbf{\Gamma}^1_{k}(t_{k,j},{x}^{k,j})$ for all $(t,x)\in [t_{k,j},T]\times H\setminus \{(t_{k,j},{x}^{k,j})\}$,
	and $\mathbf{\Gamma}^2_{k}(s,y)<\mathbf{\Gamma}^2_{k}(s_{k,j},{y}^{k,j})$ for all $(s,y)\in [s_{k,j},T]\times H\setminus \{(s_{k,j},{y}^{k,j})\}$,
\end{description}
By (\ref{0608a2024}) and (\ref{20210509a}), we can assume $\check{t}_{0},\check{s}_{0}>\hat{t}$, then we have
\begin{eqnarray}\label{2512141}
t_{k,j},s_{k,j}>\hat{t}.
\end{eqnarray}
By Lemma \ref{lemma4.40615} below, we have
\begin{eqnarray}\label{4.226}
\lim_{j\to \infty}\left(\begin{aligned}
&
t_{k,j},  & ({x}^{k,j})_N\\
& s_{k,j},  & ({y}^{k,j})_N
\end{aligned}\right)
=\left(\begin{aligned}
&t_k, &x_N^k\\
&s_k, &y_N^k
\end{aligned}\right),
\end{eqnarray}
\begin{eqnarray}\label{05231}
\lim_{j\to \infty}\left(\begin{aligned}
&
\widetilde{w}_{1}^{\hat{t},N,*}(t_{k,j}, ({x}^{k,j})_N)\\
&\widetilde{w}_{2}^{\hat{t},N,*}(s_{k,j},  ({y}^{k,j})_N)
\end{aligned}\right)
=\left(\begin{aligned}
&
\widetilde{w}_{1}^{\hat{t},N,*}(t_k,x_N^k)\\
&\widetilde{w}_{2}^{\hat{t},N,*}(s_k,y_N^k)
\end{aligned}\right),\quad \mbox{and}
\end{eqnarray}
\begin{eqnarray}\label{05232}
\lim_{j\to \infty} \left(\begin{aligned}
&w_1(t_{k,j},{x}^{k,j})\\
 &w_2(s_{k,j},{y}^{k,j})
\end{aligned}\right)
=\left(\begin{aligned}
&\widetilde{w}_{1}^{\hat{t},N,*}(t_k,x_N^k)\\
&\widetilde{w}_{2}^{\hat{t},N,*}(s_k,y_N^k)
\end{aligned}\right).
\end{eqnarray}
In view of  (\ref{0607a}) and (\ref{0607b}),  we  have a subsequence $j_k$ such that
\begin{eqnarray}\label{08181}
\begin{aligned}
	&\left(t_{k,j_k}, ({x}^{k,j_k})_N, w_1({t_{k,j_k}},{x}^{k,j_k}),((\tilde{\varphi}_k)_t,\nabla_x\tilde{\varphi}_k,\nabla^2_x\tilde{\varphi}_k)(t_{k,j_k}, ({x}^{k,j_k})_N)\right)\\
	&\underrightarrow{k\rightarrow\infty}\left({\hat{t}},\hat{x}_N, w_1(\hat{t},\hat{x}),(b_1, \nabla_{x_1}\varphi(\hat{x}_N,\hat{y}_N), X)\right),
\end{aligned}
\end{eqnarray}
\begin{eqnarray}\label{08182}
\begin{aligned}
	&\left(s_{k,j_k}, ({y}^{k,j_k})_N, w_2({s_{k,j_k}},{y}^{k,j_k}),((\tilde{\psi}_k)_t,\nabla_x\tilde{\psi}_k,\nabla^2_x\tilde{\psi}_k)(s_{k,j_k}, ({y}^{k,j_k})_N)\right)\\
	&\underrightarrow{k\rightarrow\infty}\left({\hat{t}},\hat{y}_N, w_2(\hat{t},\hat{y}),(b_2, \nabla_{x_2}\varphi(\hat{x}_N,\hat{y}_N), Y)\right).
\end{aligned}
\end{eqnarray}
It remains to show that $ {x}^{k,j_k}\rightarrow\hat{x}$ and ${y}^{k,j_k}\rightarrow\hat{y}$ as $k\rightarrow\infty$.  Noting that $({x}^{k,j_k})_N$ $\rightarrow \hat{x}_N$ and $({y}^{k,j_k})_N\rightarrow \hat{y}_N$  as $k\rightarrow\infty$,
by the continuity of $\varphi$, there exists a constant {$\tilde{M}_1>0$} such that
$$|\varphi(({x}^{k,j_k})_N,({y}^{k,j_k})_N)|\leq \tilde{M}_1.$$
Then by (\ref{051312024}) and
\begin{eqnarray}\label{1004a}
\begin{aligned}
&w_1({t_{k,j_k}},{x}^{k,j_k})+w_2({s_{k,j_k}},{y}^{k,j_k})-\varphi(({x}^{k,j_k})_N,
({y}^{k,j_k})_N)\\
\rightarrow&
w_1(\hat{t},\hat{x})+w_2(\hat{t},\hat{y})-\varphi(\hat{x}_N,\hat{y}_N),
\end{aligned}
\end{eqnarray}
there exists a constant $M_2>0$ such that
$$
|{x}^{k,j_k}|_H \vee     |{y}^{k,j_k}|_H\leq M_2.
$$
By (\ref{2512141}) and (\ref{08182}), we can find a subsequence of  $s_{k,j_k}$  still denoted by ifself such that $s_{k,j_k}\leq t_{k,j_k}$. Then, by (\ref{0608a2024}),
\begin{eqnarray*}
\begin{aligned}
	& w_1(t_{k,j_k},{x}^{k,j_k})+w_2(s_{k,j_k},y^{k,j_k})-\varphi(({x}^{k,j_k})_N,
	({y}^{k,j_k})_N)\\
	\leq& w_1(t_{k,j_k},{x}^{k,j_k})+w_2(t_{k,j_k},({y}^{k,j_k})_{s_{k,j_k},t_{k,j_k}}^{A,N})-\varphi(({x}^{k,j_k})_N,
	({y}^{k,j_k})_N)\\
	&+\rho_1(|s_{k,j_k}-t_{k,j_k}|,M_2).
\end{aligned}
\end{eqnarray*}
Letting $k\rightarrow\infty$, by (\ref{08181}) and (\ref{08182}),
\begin{eqnarray*}
	&& \liminf_{k\rightarrow\infty} \left[w_1(t_{k,j_k},{x}^{k,j_k})+w_2(t_{k,j_k},{({y}^{k,j_k})_{s_{k,j_k},t_{k,j_k}}^{A,N}})-\varphi(({x}^{k,j_k})_N,
	({y}^{k,j_k})_N)\right]\\
	&&\geq w_1(\hat{t},\hat{x})+w_2(\hat{t},\hat{y})-\varphi(\hat{x}_N,\hat{y}_N).
\end{eqnarray*}
Thus, by (\ref{101012024}) 
we get that $|{x}^{k,j_k}-\hat{x}^A_{\hat{t},t_{k,j_k}}|\rightarrow0$ as $k\rightarrow\infty$.
Notice that $|{x}^{k,j_k}-\hat{x}|\leq |\hat{x}-\hat{x}^A_{\hat{t},t_{k,j_k}}|+|{x}^{k,j_k}-\hat{x}^A_{\hat{t},t_{k,j_k}}|
$,
we have $x^{k,j_k}\rightarrow\hat{x}$ as $k\rightarrow\infty$. By (\ref{2512141}) and (\ref{08181}), we can also find a subsequence of $t_{k,j_k}$  still denoted by itself such that $s_{k,j_k}\geq t_{k,j_k}$. From the similar procedure above, we have
 $y^{k,j_k}\rightarrow\hat{y}$ as $k\rightarrow\infty$.

We notice that, by  
the property (i) of $(t_{k,j}, x^{k,j}$, $s_{k,j}, y^{k,j})$, there exists a generic constant $C>0$ such that
\begin{eqnarray*}
	2\sum_{i=0}^{\infty}\frac{1}{2^i}\left[(s_{k,j_k}-\check{s}_{i})+(t_{k,j_k}-\check{t}_{i})\right]
	\leq Cj_k^{-\frac{1}{2}};
\end{eqnarray*}
\begin{eqnarray*}
	\left|\nabla_x\left[\sum_{i=0}^{\infty}\frac{1}{2^i}
	\left|{x}^{k,j_k}-(\check{x}^{i})^A_{\check{t}_{i},t_{k,j_k}}\right|^4\right]\right|
	+\left|\nabla_x\left[\sum_{i=0}^{\infty}\frac{1}{2^i}
	\left|y^{k,j_k}-(\check{y}^{i})^A_{\check{s}_{i},s_{k,j_k}}\right|^4\right]\right|
	\leq Cj_k^{-\frac{3}{4}};
\end{eqnarray*}
and
\begin{eqnarray*}
	\left|\nabla^2_{x}\left[\sum_{i=0}^{\infty}\frac{1}{2^i}
	\left|{x}^{k,j_k}-(\check{x}^{i})^A_{\check{t}_{i},t_{k,j_k}}\right|^4\right]\right|
	+\left|\nabla^2_{x}\left[\sum_{i=0}^{\infty}\frac{1}{2^i}
	\left|y^{k,j_k}-\check{y}^{i}_{\check{s}_{i},s_{k,j_k}}\right|^4\right]\right|
	\leq Cj_k^{-\frac{1}{2}}.
\end{eqnarray*}
Therefore the theorem  holds with
\begin{eqnarray*}
\begin{aligned}
	\varphi_{1,k}(t,x):=&\tilde{\varphi}_k(t,x_N)
	-\tilde{\varphi}_k(t_{k,j_k},({x}^{k,j})_N)
	+w_1(t_{k,j_k},{x}^{k,j_k}), \\
   \varphi_{2,k}(t,x):=&\sum_{i=0}^{\infty}\frac{1}{2^i}\Upsilon((\check{t}_{i},\check{x}^{i}),(t,x))
	-\sum_{i=0}^{\infty}\frac{1}{2^i}\Upsilon(({\check{t}_{i}}^{i},\check{x}),(t_{k,j_k},{x}^{k,j_k})),\\
	\psi_{1,k}(s,y):=&\tilde{\psi}_k(s,y_N)
	-\tilde{\psi}_k(s_{k,j_k},(s_{k,j},{y}^{k,j})_N)+w_2(s_{k,j_k},y^{k,j_k}),\\
	\psi_{2,k}(s,y):=&\sum_{i=0}^{\infty}\frac{1}{2^i}\Upsilon(({\check{s}_{i}},\check{\eta}^{i}),(s,y))
	-\sum_{i=0}^{\infty}\frac{1}{2^i}\Upsilon(({\check{s}_{i}}^{i},\check{\eta}),(s_{k,j_k},y^{k,j_k}))\quad \mbox{and}\\
t_{k}:=&t_{k,j_k},\quad {x}^{k}:={x}^{k,j_k},\quad s_{k}:=s_{k,j_k}, \quad {y}^{k}:=y^{k,j_k}.
\end{aligned}
\end{eqnarray*}
The proof is now complete.
 \end{proof}

\begin{remark}\label{remarkv0528}
	As mentioned in Remark 6.1 in Chapter V of \cite{fle1},  Condition (\ref{0528110}) is stated with reverse inequality in Theorem 8.3 of \cite{cran2}. However, we  immediately obtain  results (\ref{0608v2024})-(\ref{II0615}) from Theorem
	8.3 of \cite{cran2} by considering the functions $u_1(t,x) :=\widetilde{w}_{1}^{\hat{t},N,*}(T-t,x)$ and
	$u_2(t,x) := \widetilde{w}_{2}^{\hat{t},N,*}(T-t,x)$.
\end{remark}

\par
To complete the  proof of Theorem \ref{theorem0513}, it remains to state and prove the following  two lemmas.
\begin{lemma}\label{lemma4.30615}\ \ Let all the conditions in Theorem  \ref{theorem0513} be satisfied. Recall that $(\hat{t},\hat{x},\hat{y})$ is given  in Theorem  \ref{theorem0513},  $\widetilde{w}_{1}^{\hat{t},N,*}$ and $ \widetilde{w}_{2}^{\hat{t},N,*}$ are defined in Definition \ref{definition06072024} with respect to $w_1$ and $w_2$ given in Theorem  \ref{theorem0513}, respectively.
	Then the function:
	$$\widetilde{w}_{1}^{\hat{t},N,*}(t,x_N)+ \widetilde{w}_{2}^{\hat{t},N,*}(t,y_N)-\varphi(x_N,y_N),   \quad (t,x_N,y_N)\in [0,T]\times  H_N\times  H_N $$
attains the maximum  at $(\hat{t},\hat{x}_N,\hat{y}_N)$.
	Moreover, we have
	\begin{eqnarray}\label{2020020206}
	\widetilde{w}_{1}^{\hat{t},N,*}(\hat{t},\hat{x}_N)={w}_{1}(\hat{t},\hat{x}), \quad \widetilde{w}_{2}^{\hat{t},N,*}(\hat{t},\hat{y}_N)={w}_{2}(\hat{t},\hat{y}).
	\end{eqnarray}
\end{lemma}

\begin{proof}
For every $\hat{t}\leq t\leq s\leq T$ and $x_N\in   H_N$,
from the definition of $\widetilde{w}^{\hat{t},N}_1$ 
it follows that
\begin{eqnarray}\label{220831}
\begin{aligned}
\widetilde{w}^{\hat{t},N}_1(t,x_N)-\widetilde{w}^{\hat{t},N}_1(s,x_N)
=\sup_{y_N=x_N}
w_{1}(t,y)
\,\,-\sup_{y_N=x_N}
w_{1}(s,y).
\end{aligned}
\end{eqnarray}
By (\ref{051312024}), there exist constants {$M_3>0$} and $\varepsilon>0$ such that
\begin{eqnarray}\label{07061}
w_1(t,z)\leq -\varepsilon|z|_H,\ \ \mbox{if}\ |z|_H \geq M_3, \ (t,z)\in [\hat{t},T]\times H.
\end{eqnarray}
By (\ref{0608a2024}),
\begin{eqnarray*}
	w_1(\hat{t},x_N)-\sup_{l\in[\hat{t},T]}\rho_1(|l-\hat{t}|,|x_N|_H)\leq w_1(t,x_N+(e^{(t-\hat{t})A}x_N)^{-}_N)\leq \sup_{y_N=x_N}w_{1}(t,y).
\end{eqnarray*}
Notice that $w_1(\hat{t},x_N)-\sup_{l\in[\hat{t},T]}\rho_1(|l-\hat{t}|,|x_N|_H)$ depends only on $x_N$, there exists a constant $C^1_{x_N}>0$ depending only on   $x_N$ such that, for all $(t,z)\in [\hat{t},T]\times H$ satisfying $|z|_H\geq C^1_{x_N}$,
\begin{eqnarray}\label{07062}
-\varepsilon|z|_H< w_1(\hat{t},x_N)-\sup_{l\in[\hat{t},T]}\rho_1(|l-\hat{t}|,|x_N|_H)\leq \sup_{y_N=x_N}w_{1}(t,y).
\end{eqnarray}
Taking $C_{x_N}=M_3\vee C^1_{x_N}$, by (\ref{07061}) and (\ref{07062}),
\begin{eqnarray}\label{220831a}
w_1(t,z)<\sup_{y_N=x_N}w_{1}(t,y),\ \ \mbox{if} \ |z|_H\geq C_{x_N}, (t,z)\in [\hat{t},T]\times H.
\end{eqnarray}
Combining (\ref{220831}) and (\ref{220831a}),  from (\ref{0608a2024}) we have
\begin{eqnarray}\label{202105083}
\begin{aligned}
&\quad \widetilde{w}^{\hat{t},N}_1(t,x_N)-\widetilde{w}^{\hat{t},N}_1(s,x_N)\\
&=\sup_{y_N=x_N,
	|y|_H\leq C_{x_N}}\!\!\!\!\!
w_{1}(t,y)
\,\, -\sup_{y_N=x_N} w_{1}(s,y)
 \\
\qquad &\leq \sup_{
	y_N=x_N,
	|y|_H\leq C_{x_N}}\!\!\!\!\!
{[}w_{1}(t,y)-w_{1}(s,x_N+(e^{(s-t)A}y)^{-}_N))
{]}\\
&\leq  \sup_{y_N=x_N,
	|y|_H\leq C_{x_N}}    \!\!\! \!\!\rho_1(|s-t|,\, |y|_H)
\, \leq \, \rho_1(|s-t|,\, C_{x_N}).
\end{aligned}
\end{eqnarray}
Clearly, if $0\leq t\leq s\leq \hat{t}$, we have
\begin{eqnarray}\label{202105084}
\widetilde{w}^{\hat{t},N}_1(t,x_N)-\widetilde{w}^{\hat{t},N}_1(s,x_N)=-(\hat{t}-t)^{\frac{1}{2}}+(\hat{t}-s)^{\frac{1}{2}}\leq 0,
\end{eqnarray}
and, if $0\leq t\leq  \hat{t}\leq s\leq T$ , we have
\begin{eqnarray}\label{202105085}
\ \ \                \widetilde{w}^{\hat{t},N}_1(t,x_N)-\widetilde{w}^{\hat{t},N}_1(s,x_N)\leq\widetilde{w}^{\hat{t},N}_1(\hat{t},x_N)-\widetilde{w}^{\hat{t},N}_1(s,x_N) \leq \rho_1(|s-\hat{t}|,C_{x_N}).
\end{eqnarray}
On the other hand,
for every $(t,x_N,y_N)\in [0,T]\times  H_N\times  H_N$,
by the definitions of $\widetilde{w}_{j}^{\hat{t},N,*}$, $j=1,2$, there exist  sequences  $(l_i,x^N_i), (\tau_i,y^N_i)\in [0,T]\times  H_N$ such that
$(l_{i},x^N_i)\rightarrow (t,x_N)$ and $(\tau_{i},y^N_i)\rightarrow (t,y_N)$
as $i\rightarrow\infty$ and
\begin{eqnarray}\label{202105081}
\widetilde{w}_{1}^{\hat{t},N,*}(t,x_N)=\lim_{i\rightarrow\infty}\widetilde{w}^{\hat{t},N}_{1}(l_i,x^N_i), \ \ \ \widetilde{w}_{2}^{\hat{t},N,*}(t,y_N)=\lim_{i\rightarrow\infty}\widetilde{w}^{\hat{t},N}_2(\tau_i,y^N_i).
\end{eqnarray}
Without loss of generality, we may assume  $l_i\leq \tau_i$ for all $i>0$.
By (\ref{202105083})-(\ref{202105085}), we have
\begin{eqnarray}\label{202105082jiaa}
\begin{aligned}
\widetilde{w}_{1}^{\hat{t},N,*}(t,x_N)=\lim_{i\rightarrow\infty}\widetilde{w}^{\hat{t},N}_{1}(l_i,x^N_i)\leq \liminf_{i\rightarrow\infty}[\widetilde{w}^{\hat{t},N}_{1}(\tau_i,x^N_i)+\rho_1(|\tau_i-l_i|,C_{x^N_i})].
\end{aligned}
\end{eqnarray}
Here $C_{x^N_i}>0$ is the constant  that makes the following formula true:
$$
w_1(t,z)<\sup_{
	y_N=x^N_i} w_{1}(t,y),
\quad \mbox{if} \ |z|_H\geq C_{x^N_i}\  \text{and} \ (t,x)\in [\hat{t},T]\times H.
$$
We claim that we can assume that there exists a constant $M_4>0$ such that $C_{x^N_i}\leq M_4$ for all $i\geq 1$. Indeed, if not, for every $n$, there exists  $i_n$ such that
\begin{eqnarray}
\qquad\quad
\widetilde{w}^{\hat{t},N}_{1}(l_{i_n},x^N_{i_n})=\begin{cases}\displaystyle \sup_{
	y_N=x^N_{i_n},
	|y|_H> n}
{[}w_{1}(l_{i_n},y){]}, \qquad\qquad \quad \ l_{i_n}\geq \hat{t};\\[5mm]
\displaystyle \sup_{y_N=x^N_{i_n},
	|y|_H> n}
{[}w_{1}(\hat{t},y){]}-({\hat{t}}-l_{i_n})^{\frac{1}{2}}, \quad \ l_{i_n}< \hat{t}.
\end{cases}
\end{eqnarray}
Letting $n\rightarrow\infty$, by (\ref{051312024}), we get that
$$
\widetilde{w}^{\hat{t},N}_{1}(l_{i_n},x^N_{i_n})\rightarrow-\infty \ \mbox{as}\ n\rightarrow\infty,
$$
which contradicts the convergence that $ \widetilde{w}_{1}^{\hat{t},N,*}(t,x_N)=\lim_{i\rightarrow\infty}\widetilde{w}^{\hat{t},N}_{1}(l_i,x^N_i)$. Then, by  (\ref{202105082jiaa}),
\begin{eqnarray}\label{20210704a}
\begin{aligned}
\widetilde{w}_{1}^{\hat{t},N,*}(t,x_N)
\leq \, \liminf_{i\rightarrow\infty}\, [\widetilde{w}^{\hat{t},N}_{1}(\tau_i,x^N_i)+\rho_1(|\tau_i-l_i|,M_4)]
 = \, \liminf_{i\rightarrow\infty}\, \widetilde{w}^{\hat{t},N}_{1}(\tau_i,x^N_i).
\end{aligned}
\end{eqnarray}
Therefore, by (\ref{202105081}), (\ref{20210704a}) and the definitions of $\widetilde{w}^{\hat{t},N}_{1}$ and $\widetilde{w}^{\hat{t},N}_{2}$,
\begin{eqnarray}\label{20210508b6}
\begin{aligned}
&\quad\widetilde{w}_{1}^{\hat{t},N,*}(t,x_N)+\widetilde{w}_{2}^{\hat{t},N,*}(t,y_N)-\varphi(x_N,y_N)\\
&\leq\liminf_{i\rightarrow\infty}\, [\widetilde{w}^{\hat{t},N}_{1}(\tau_i,x^N_i)+\widetilde{w}^{\hat{t},N}_{2}(\tau_i,y^N_i)-\varphi(x^N_i,y^N_i)]\\
&\leq\sup_{\scriptsize\begin{matrix}(l,x^N_0,y^N_0)\\
	\in [0,T]\times  H_N\times H_N\end{matrix}}\!\!\!\!\! [\widetilde{w}^{\hat{t},N}_{1}(l,x^N_0)+\widetilde{w}^{\hat{t},N}_{2}(l,y^N_0)-\varphi(x^N_0,y^N_0)]\\
&=\sup_{\scriptsize\begin{matrix}(l,x^N_0,y^N_0)\\
	\in [\hat{t},T]\times  H_N\times H_N\end{matrix}}\!\!\!\!\!
[\widetilde{w}^{\hat{t},N}_{1}(l,x^N_0)+\widetilde{w}^{\hat{t},N}_{2}(l,y^N_0)-\varphi(x^N_0,y^N_0)].
\end{aligned}
\end{eqnarray}
We also have, for $(l,x^N_0,y^N_0)\in [\hat{t},T]\times  H_N\times  H_N$,
\begin{eqnarray}\label{0608aa}
\begin{aligned}
\ \ \ \ \ \ &\quad
\widetilde{w}^{\hat{t},N}_{1}(l,x^N_0)+  \widetilde{w}^{\hat{t},N}_{2}(l,y^N_0)-\varphi(x^N_0,y^N_0)  \\
&=\sup_{
	z_N=x^N_0,
	y_N=y^N_0}\!\!\!
\left[w_{1}(l,z)
+w_{2}(l,y)
-\varphi(z_N,y_N)\right]\\
&\leq w_{1}(\hat{t},\hat{x})+w_{2}(\hat{t},\hat{y})
-\varphi(\hat{x}_N,\hat{y}_N),
\end{aligned}
\end{eqnarray}
where the  inequality becomes equality if  $l={\hat{t}}$ 
and $x^N_0=\hat{x}_N,y^N_0=\hat{y}_N$.
Combining  (\ref{20210508b6}) and (\ref{0608aa}), we obtain that
\begin{eqnarray}\label{0608a1}
\begin{aligned}
\widetilde{w}_{1}^{\hat{t},N,*}(t,x_N)+\widetilde{w}_{2}^{\hat{t},N,*}(t,y_N)-\varphi(x_N,y_N)
\leq w_{1}(\hat{t},\hat{x})+w_{2}(\hat{t},\hat{y})
-\varphi(\hat{x}_N,\hat{y}_N).
\end{aligned}
\end{eqnarray}
By the definitions of $\widetilde{w}_1^{\hat{t},N,*}$ and $\widetilde{w}_2^{\hat{t},N,*}$, we have $\widetilde{w}_1^{\hat{t},N,*}(t,x_N)\geq \widetilde{w}^{\hat{t},N}_1(t,x_N), \widetilde{w}_{2}^{\hat{t},N,*}(t,y_N)$ $ \geq \widetilde{w}^{\hat{t},N}_{2}(t,y_N)$.  Then by also (\ref{0608aa}) and (\ref{0608a1}), for every $(t,x_N,y_N)\in [0,T]\times  H_N\times H_N$,
\begin{eqnarray}\label{0608abc}
\begin{aligned}
&\quad\widetilde{w}_{1}^{\hat{t},N,*}(t,x_N)+\widetilde{w}_{2}^{\hat{t},N,*}(t,y_N)-\varphi(x_N,y_N)\\
&\leq w_{1}(\hat{t},\hat{x})+w_{2}(\hat{t},\hat{y})
-\varphi(\hat{x}_N,\hat{y}_N)\\
&=\widetilde{w}^{\hat{t},N}_{1}(\hat{t},\hat{x}_N)+  \widetilde{w}^{\hat{t},N}_{2}(\hat{t},\hat{y}_N)
-\varphi(\hat{x}_N,\hat{y}_N)\\
&\leq\widetilde{w}_{1}^{\hat{t},N,*}(\hat{t},\hat{x}_N)+  \widetilde{w}_{2}^{\hat{t},N,*}(\hat{t},\hat{y}_N)
-\varphi(\hat{x}_N,\hat{y}_N).
\end{aligned}
\end{eqnarray}
Thus
we obtain that (\ref{2020020206}) holds true, and  $ \widetilde{w}_{1}^{\hat{t},N,*}(t,x_N)+\widetilde{w}_{2}^{\hat{t},N,*}(t,y_N)-\varphi(x_N, y_N)$ has a maximum at $(\hat{t}, \hat{x}_N,
\hat{y}_N)$
on $[0,T]\times  H_N\times  H_N$.
The proof is now complete.
\end{proof}

\begin{lemma}\label{lemma4.40615}\ \
	Let all the conditions in Theorem  \ref{theorem0513} hold. Recall that $\Gamma^1_{k}$ and $\Gamma^2_{k}$ are defined in  (\ref{4.11116}).
	Then the maximum points $(t_{k,j},x^{k,j})$
	of $\Gamma^1_{k}(t,x)
	-\sum_{i=0}^{\infty}
	\frac{1}{2^i}\Upsilon((\check{t}_{i},\check{x}^i),(t,x))$ and the maximum points
	$ (s_{k,j},y^{k,j})$
	of $\Gamma^2_{k}(s,y)
	-\sum_{i=0}^{\infty}
	\Upsilon((\check{s}_{i},\check{y}^{i}),(s,y))$
	satisfy  conditions (\ref{4.226}), (\ref{05231}) and (\ref{05232}).
\end{lemma}

\begin{proof}
Recall that $\widetilde{w}_{i}^{\hat{t},N,*}\geq \widetilde{w}^{\hat{t},N}_{i}$, $i=1,2$, 
by  the definitions of $\widetilde{w}^{\hat{t},N}_i$, $i=1,2$, 
 we get that
\begin{eqnarray*}
\begin{aligned}
	& \widetilde{w}_{1}^{\hat{t},N,*}(t_{k,j},(x^{k,j})_N)
	-\tilde{\varphi}_k(t_{k,j},(x^{k,j})_N)\\
\geq& w_1(t_{k,j},x^{k,j})-\tilde{\varphi}_k(t_{k,j},x^{k,j})_N)
	=\Gamma^1_{k}(t_{k,j},x^{k,j}),\\
	&\widetilde{w}_{2}^{\hat{t},N,*}(s_{k,j},(y^{k,j})_N)
	-\tilde{\psi}_k(s_{k,j},(y^{k,j})_N)\\
	\geq& w_2(s_{k,j},y^{k,j})-\tilde{\psi}_k(s_{k,j},(y^{k,j})_N)
	=\Gamma^2_{k}(s_{k,j},y^{k,j}).
\end{aligned}
\end{eqnarray*}
We notice that, from (\ref{20210509a}) and the property (ii) of $(t_{k,j},x^{k,j}; s_{k,j},y^{k,j})$,
\begin{eqnarray*}
	&&\Gamma^1_{k}(t_{k,j},x^{k,j})\geq\Gamma^1_{k}({\check{t}_{0}},\check{x}^{0})
	\geq \sup_{(t,x)\in [\hat{t},T]\times H}\Gamma^1_{k}(t,x)-\frac{1}{j}, \\
	&& \Gamma^2_{k}(s_{k,j},y^{k,j})\geq\Gamma^2_{k}({\check{s}_{0}},\check{y}^{0})
	\geq \sup_{(s,y)\in [\hat{t},T]\times H}\Gamma^2_{k}(s,y)-\frac{1}{j},
\end{eqnarray*}
and by the  definitions of $\widetilde{w}_{1}^{\hat{t},N,*}$ and $\widetilde{w}_{2}^{\hat{t},N,*}$,
\begin{eqnarray*}
	&&\sup_{(t,x)\in [\hat{t},T]\times H}\Gamma^1_{k}(t,x)\geq\widetilde{w}_{1}^{\hat{t},N,*}(t_k,{x}_N^k)-\tilde{\varphi}_k(t_k,{x}_N^k),\\
	&&\sup_{(s,y)\in [\hat{t},T]\times H}\Gamma^2_{k}(s,y)\geq\widetilde{w}_{2}^{\hat{t},N,*}(s_k,{y}_N^k)-\tilde{\psi}_k(s_k,{y}_N^k).
\end{eqnarray*}
Therefore,
\begin{eqnarray}\label{0525b6}
&&\widetilde{w}_{1}^{\hat{t},N,*}(t_{k,j},(x^{k,j})_N)-\tilde{\varphi}_k(t_{k,j},(x^{k,j})_N)\nonumber\\
&\geq&\Gamma^1_{k}(t_{k,j},x^{k,j})\geq
\widetilde{w}_{1}^{\hat{t},N,*}(t_k,{x}_N^k)-\tilde{\varphi}_k(t_k,{x}_N^k)-\frac{1}{j},
\end{eqnarray}
\begin{eqnarray}\label{0525b0920}
&&\widetilde{w}_{2}^{\hat{t},N,*}(s_{k,j},(y^{k,j})_N)-\tilde{\psi}_k(s_{k,j},(y^{k,j})_N)\nonumber\\
&\geq&\Gamma^2_{k}(s_{k,j},y^{k,j})\geq
\widetilde{w}_{2}^{\hat{t},N,*}(s_k,{y}_N^k)-\tilde{\psi}_k(s_k,{y}_N^k)-\frac{1}{j}.
\end{eqnarray}
By (\ref{051312024}) and  that $\tilde{\varphi}_k$ and $\tilde{\psi}_k$ are bounded from below, 
there exists a constant  $M_5>0$  that is sufficiently  large   that
$|x^{k,j}|\vee
|y^{k,j}|<M_5$. In particular, $|(x^{k,j})_N|\vee|(y^{k,j})_N|$ $<M_5$.
We note that $M_5$ is independent of $j$.
Then letting $j\rightarrow\infty$ in (\ref{0525b6}) and (\ref{0525b0920}),
we obtain (\ref{4.226}).  Indeed, if not, we may assume that  there exist $(\grave{t},\grave{x}_N;\grave{s},\grave{y}_N)\in ([0, T]\times  H_N)^2$ and  a subsequence of $(t_{k,j},(x^{k,j})_N;s_{k,j},(y^{k,j})_N)$ still denoted by itself  such that
$$
\lim_{j\to \infty}(t_{k,j},(x^{k,j})_N;s_{k,j},(y^{k,j})_N)= (\grave{t},\grave{x}_N;\grave{s},\grave{y}_N)\neq (t_k,x_N^k;s_k,y_N^k).
$$
Letting $j\rightarrow\infty$ in (\ref{0525b6}) and (\ref{0525b0920}), by the upper semi-continuity of $\widetilde{w}_{1}^{\hat{t},N,*}+\widetilde{w}_{2}^{\hat{t},N,*}-\tilde{\varphi}_k-\tilde{\psi}_k$, we have
\begin{eqnarray*}
	 &&\widetilde{w}_{1}^{\hat{t},N,*}(\grave{t},\grave{x}_N)+\widetilde{w}_{2}^{\hat{t},N,*}(\grave{s},\grave{y}_N)
     -\tilde{\varphi}_k(\grave{t},\grave{x}_N)-\tilde{\psi}_k(\grave{s},\grave{y}_N)\\
	&\geq&\widetilde{w}_{1}^{\hat{t},N,*}(t_k,{x}_N^k)+\widetilde{w}_{2}^{\hat{t},N,*}(s_k,{y}_N^k) -\tilde{\varphi}_k(t_k,{x}_N^k)-\tilde{\psi}_k(s_k,{y}_N^k),
\end{eqnarray*}       which  contradicts that
$(t_k,x_N^{k},s_k,y_N^{k})$  is the  strict  maximum point of the function: $\widetilde{w}_{1}^{\hat{t},N,*}(t,x_N)+\widetilde{w}_{2}^{\hat{t},N,*}(s,y_N)-\tilde{\varphi}_k(t,x_N)-\tilde{\psi}_k(s,y_N), \, (t,x_N;s,y_N)\in ([0,T]\times H_N)^2$.\par
By (\ref{4.226}), the  upper semi-continuity of $\widetilde{w}_{1}^{\hat{t},N,*}$ and $\widetilde{w}_{2}^{\hat{t},N,*}$ and  the continuity of $\tilde{\varphi}_k$ and $\tilde{\psi}_k$, letting $j\rightarrow\infty$ in (\ref{0525b6})  and (\ref{0525b0920}), we obtain
\begin{eqnarray*}
\begin{aligned}
	&\widetilde{w}_{1}^{\hat{t},N,*}(t_k,x_N^k)
	\geq \limsup_{j\rightarrow\infty}\widetilde{w}_{1}^{\hat{t},N,*}(t_{k,j},({x}^{k,j})_N)\\
	\geq& \liminf_{j\rightarrow\infty}\widetilde{w}_{1}^{\hat{t},N,*}(t_{k,j},({x}^{k,j})_N)
	\geq\widetilde{w}_{1}^{\hat{t},N,*}(t_k,x_N^k),\\
	&\widetilde{w}_{2}^{\hat{t},N,*}(s_k,y_N^k)
	\geq \limsup_{j\rightarrow\infty}\widetilde{w}_{2}^{\hat{t},N,*}(s_{k,j},({y}^{k,j})_N)\\
	\geq& \liminf_{j\rightarrow\infty}\widetilde{w}_{2}^{\hat{t},N,*}(s_{k,j},({y}^{k,j})_N)
	\geq\widetilde{w}_{2}^{\hat{t},N,*}(s_k,y_N^k).
\end{aligned}
\end{eqnarray*}
Thus, we get (\ref{05231}) holds true.
Letting $j\rightarrow\infty$ in (\ref{0525b6}) and (\ref{0525b0920}), by (\ref{05231}) and the definitions of $\Gamma^1_{k}$ and $\Gamma^2_{k}$,
$$
\widetilde{w}_{1}^{\hat{t},N,*}(t_k,x_N^k)=\lim_{j\rightarrow\infty}w_1(t_{k,j}, x^{k,j}), \ \ \
\widetilde{w}_{2}^{\hat{t},N,*}(s_k,y_N^k)=\lim_{j\rightarrow\infty}w_2(s_{k,j},y^{k,j}).
$$
Thus, we obtain (\ref{05232}).
The proof is now complete.
\end{proof}

\section{Viscosity solutions to  HJB equations: uniqueness}
\par
This section is devoted to the  proof of uniqueness of  viscosity
solutions to equation (\ref{hjb1}). This result, together with
the results from  Section \ref{sect-exist}, will be used to characterize
the value functional defined by (\ref{value1}).
We require  the following assumption on $\sigma$.
\begin{assumption}\label{hypstate5666}
	For every $(t,x)\in [0,T)\times H$,
	\begin{eqnarray}\label{g5}
	\lim_{N\to \infty} \sup_{u\in U}\left|Q_N{\sigma}(t,x,u)\right|_{L_2(\Xi, H)}^2=0.
	\end{eqnarray}
\end{assumption}
\par
{By \cite[Proposition 11.2.13]{zhang}, without loss of generality we may assume that} there exists a constant $L>0$ such that,
for all $(t,x, p,l)\in [0,T]\times{H}\times H\times {\mathcal{S}}(H)$ and $r_1,r_2\in \mathbb{R}$ with $r_1<r_2$,
\begin{eqnarray}\label{5.1}
{\mathbf{H}}(t,x,r_1,p,l)-{\mathbf{H}}(t,x,r_2,p,l)\geq L(r_2-r_1).
\end{eqnarray}
We  now state the main result of this {section}.
\begin{theorem}\label{theoremhjbm2024}
	Let Assumption \ref{hypstate}  be satisfied.
	Let $  W_1\in C^{0+}([0,T]\times H)$ $(\mbox{resp}.,    W_2\in C^{0+}([0,T]\times H))$ be  a viscosity sub-solution (resp., super-solution) to equation (\ref{hjb1}) and  let  there exist a  constant $L>0$  and a local modulus of continuity  $\rho_2$
	such that, for any  $0\leq t\leq  s\leq T$ and
	$x,y\in{H}$,
	\begin{eqnarray}\label{w2024}
	|   W_1(t,x)|\vee |   W_2(t,x)|\leq L (1+|x|^2),
	\end{eqnarray}
	\begin{eqnarray}\label{w12024}
	\begin{aligned}
	&\left|   W_1(s,x^A_{t,s})-   W_1(t,y)\right|\vee\left|   W_2(s,x^A_{t,s})-   W_2(t,y)\right|\\
	\leq&\, \rho_2(|s-t|,\, |x|\vee|y|)+L(1+|x|+|y|)|x-y|.
	\end{aligned}
	\end{eqnarray}
	Then  $   W_1\leq    W_2$.
\end{theorem}
Theorems    \ref{theoremvexist} and \ref{theoremhjbm2024} yield that the viscosity solution to   HJB equation  (\ref{hjb1})
is  the value functional  $  V$.

\begin{theorem}\label{theorem522024}
	Let  Assumption \ref{hypstate}   be satisfied.  Then the value
	functional $  V$ 
is the unique viscosity
	solution to HJB equation~(\ref{hjb1}) which satisfies (\ref{w2024}) and (\ref{w12024}).
\end{theorem}

\begin{proof}[Proof of Theorem \ref{theoremhjbm2024}]
It is sufficient to prove that the functional $W_1-\frac{\varrho}{t+1}\leq W_2$ for $\varrho>0$. Since $W_1$ is a viscosity sub-solution of HJB equation (\ref{hjb1}),  the functional $\widetilde{w}^\varrho:=W_1-\frac{\varrho}{t+1}$  with $\varrho>0$, is a viscosity sub-solution
to the following HJB  equation
\begin{eqnarray}\label{1002a2024}
\begin{cases}
 (  \widetilde{w}^\varrho)_{t}(t,x)+\left\langle A^*\nabla_x  \widetilde{w}^\varrho(t,x),\,  x\right\rangle_H\\[3mm]
\quad+ {\mathbf{H}}\left(t,x,\, (  \widetilde{w}^\varrho, \nabla_x   \widetilde{w}^\varrho,\nabla^2_{x}   \widetilde{w}^\varrho)(t,x)\right)
=c:= \frac{\varrho}{(T+1)^2},  \qquad  (t, x)\in [0,T)\times H; \\[3mm]
  \widetilde{w}^\varrho(T,x)= {\phi}(x), \  x\in H.\\
\end{cases}
\end{eqnarray}
Therefore,  it is sufficient to prove
$W_1\leq W_2$ under the stronger  assumption that $W_1$ is a viscosity sub-solution of the last HJB equation.

The rest of the proof splits in the following four steps.
\par
$Step\  1.$ Definitions of auxiliary functionals.
\par
It is sufficient to prove that $   W_1\leq   W_2$ in $[T-\bar{a},T)\times
H$ with
$$\bar{a}:=\frac{1}{2(20L+8)L}\wedge{T},$$
for the desired comparison is easily extended  to the whole time interval $[0,T]$ via a backward iteration over the intervals
$[T-i\bar{a},T-(i-1)\bar{a})$.  Otherwise, there is  $(\tilde{t},\tilde{x})\in (T-\bar{a},T)\times
H$  such that
$\tilde{m}:=  W_1(\tilde{t},\tilde{x})-  W_2(\tilde{t},\tilde{x})>0$.
\par
Let  $\varepsilon >0$ be  a small number such that
$$
   W_1(\tilde{t},\tilde{x})-   W_2(\tilde{t},\tilde{x})-2\varepsilon \frac{\nu T-\tilde{t}}{\nu
	T}|\tilde{x}|^4
>\frac{\tilde{m}}{2}\ \mbox{and}
$$
\begin{eqnarray}\label{5.32024}
\frac{\varepsilon}{\nu T}\leq\frac{c}{4}
\end{eqnarray}
with
$$
\nu:=1+\frac{1}{2T(20L+8)L}.
$$
Next,  we define for any  $(t,x,y)\in [0,T]\times H\times H$,
\begin{eqnarray*}
	  \Psi(t,x,y):=  W_1(t,x)-  W_2(t,y)-\beta|x-y|^2
	-\varepsilon\frac{\nu T-t}{\nu
		T}\, \left(|x|^4+|y|^4\right),
\end{eqnarray*}
and  for all $(t,x,y; s,x',y')\in ([0,T]\times H\times H)^2$,
\begin{eqnarray*}
\Upsilon^{1}((t,x,y),\, (s,x',y')):= 2|s-t|^2+\left|x^A_{t,t\vee s}-(x')^A_{s,t\vee s}\right|^4 +\left|y^A_{t,t\vee s}-(y')^A_{s,t\vee s}\right|^4.
\end{eqnarray*}
In view of  (\ref{w2024}), the function $  \Psi\in USC^{0+}([{T-\bar{a}},T]\times H\times H)$  bounded from above and satisfies
$$
\limsup_{|x|+|y|\rightarrow\infty}\sup_{t\in[{T-\bar{a}},T]}\left[\frac{\Psi(t,x,y)}{|x|}\right]<0.
$$
Take
$\delta_i:=\frac{1}{2^i}$ for all $i\geq0$. From the same proof procedure of  Lemma \ref{theoremleft} it follows that,
for every  $(t_0,x^0,y^0)\in [{\tilde{t}},T)\times H\times H$ satisfying
$$
  \Psi({t_0},x^0,y^0)\geq \sup_{(s,x,y)\in  [\tilde{t},T]\times H\times H}  \Psi(t,x,y)-\frac{1}{\beta},\
\    \mbox{and} \ \   \Psi({t_0},x^0, {y^0})\geq   \Psi(\tilde{t},\tilde{x},\tilde{y}) >\frac{\tilde{m}}{2},
$$
there exist $(\hat{t},\hat{x},\hat{y})\in [\tilde{t},T]\times H\times H$, a sequence $\{(t_i,x^i,y^i)\}_{i\geq1}\subset
[\tilde{t},T]\times H\times H$ and constant $C>0$ such that
\begin{description}
	\item[(i)] $\Upsilon^{1}((t_i,x^i,y^i),(\hat{t},\hat{x},\hat{y}))
\leq \frac{1}{2^i\beta}$, $|x^i|\vee|y^i|\leq C$ for all $i\geq0$,
and $t_i\uparrow \hat{t}$ as $i\rightarrow\infty$,
	\item[(ii)]  $  \Psi_1(\hat{t},\hat{x},\hat{y})
\geq   \Psi({t_0},x^0,{y^0})$, and
	\item[(iii)]    for all $(s, x,y)\in [\hat{t},T]\times H\times H\setminus \{(\hat{t},\hat{x},\hat{y})\}$,
	{   \begin{eqnarray}\label{iii42024}
		  \Psi_1(s,x,y)
		<  \Psi_1(\hat{t},\hat{x},\hat{y}),
		\end{eqnarray}}
\end{description}
where we define, for $(t,x,y)\in  [\tilde t, T]\times H\times H$,
{$$
	  \Psi_1(t,x,y):=    \Psi(t,x,y)
	-\sum_{i=0}^{\infty}
	\frac{1}{2^i}\Upsilon^{1}((t_i,x^i,y^i),(t,x,y))
.
	$$}
Note that the point
$({\hat{t}},\hat{x},\hat{y})$ depends on  $\beta$ and
$\varepsilon$.
\par
\par
$Step\ 2.$
There exists ${{M}_0}>0$  {independent of $\beta$}
such that
\begin{eqnarray}\label{5.10jiajiaaaa}
|\hat{x}|\vee|\hat{y}|<M_0, \ \mbox{and}
\end{eqnarray}
\begin{eqnarray}\label{5.10}
\lim_{\beta\to \infty} \beta|\hat{x}-\hat{y}|^2
=0.
\end{eqnarray}
Moreover, there exists  $N_0>0$ {independent of $\beta$}
such that  $\hat{t}\in [\tilde{t},T)$ for all $\beta\geq N_0$.

The proof follows. First,   noting $\nu$ is independent of  $\beta$, by the definition of  ${\Psi}$,
there exists a constant  ${M}_0>0$ {independent of $\beta$}  that is sufficiently  large   that
$
\Psi(t,x, y)<0
$ for all $t\in [T-\bar{a},T]$ and $|x|\vee|y|\geq M_0$. Thus, we have $|\hat{x}|\vee|\hat{y}|\vee
|x^{0}|\vee|y^{0}|<M_0$. {We note that $M_0$ depends on $\varepsilon$.}
\par
Second, by (\ref{iii42024}), we have
{\begin{eqnarray}\label{5.56789}
	2\Psi_1(\hat{t},\hat{x},\hat{y})
	\geq  \Psi_1(\hat{t},\hat{x},\hat{x})
	+\Psi_1(\hat{t},\hat{y},\hat{y}).
	\end{eqnarray}}
This implies that
\begin{eqnarray}\label{5.6}
\begin{aligned}
2\beta|\hat{x}-\hat{y}|^2
&\leq|W_1(\hat{t},\hat{x})-W_1(\hat{t},\hat{y})|
+|W_2(\hat{t},\hat{x})-W_2(\hat{t},\hat{y})|\\
&\quad+
\sum_{i=0}^{\infty}\frac{1}{2^i}[\Upsilon((t_i,y^i),(\hat{t},\hat{x}))+\Upsilon((t_i,x^i),(\hat{t},\hat{y}))].
\end{aligned}
\end{eqnarray}
On the other hand, 
by the property (i) of $(\hat{t}, \hat{x},\hat{y})$,
\begin{eqnarray}\label{4.7jiajia130}
\begin{aligned}
&\quad
\sum_{i=0}^{\infty}\frac{1}{2^i}[\Upsilon((t_i,y^i), (\hat{t},\hat{x}))+\Upsilon((t_i,x^i),(\hat{t},\hat{y}))]\\
&\leq2^3\sum_{i=0}^{\infty}\frac{1}{2^i}[\Upsilon((t_i,y^i),(\hat{t},\hat{y}))
+\Upsilon((t_i,x^i),(\hat{t},\hat{x}))+2|\hat{x}-\hat{y}|^4]
\leq\frac{2^4}{\beta}+{2^5}|\hat{x}-\hat{y}|^4.
\end{aligned}
\end{eqnarray}
Combining (\ref{5.6}) and (\ref{4.7jiajia130}),  from  (\ref{w2024}) and (\ref{5.10jiajiaaaa})  we have
\begin{eqnarray}\label{5.jia6}
\begin{aligned}
&\quad(2{\beta}-2^7M^2_0)|\hat{x}-\hat{y}|^2\leq 2{\beta}|\hat{x}-\hat{y}|^2-2^5|\hat{x}-\hat{y}|^4\\
&\leq |W_1(\hat{t},\hat{x})-W_1(\hat{t},\hat{y})|
+|W_2(\hat{t},\hat{x})-W_2(\hat{t},\hat{y})|+\frac{2^4}{\beta}\\
&\leq 2L(2+|\hat{x}|^2+|\hat{y}|^2)+\frac{2^4}{\beta}
\leq 4L(1+M_0^2)+\frac{2^4}{\beta}.
\end{aligned}
\end{eqnarray}
Letting $\beta\rightarrow\infty$, we have
\begin{eqnarray}\label{5.66666123}
	|\hat{x}-\hat{y}|^2
	\leq \frac{1}{2{\beta}-2^7M^2_0}\left[4L(1+M^2_0)+\frac{2^4}{\beta}\right]\rightarrow0\ 
	\mbox{as} \ \beta\rightarrow\infty.
\end{eqnarray}
From   (\ref{w12024}), (\ref{5.6}), (\ref{4.7jiajia130}) and (\ref{5.66666123}), we conclude
that
\begin{eqnarray}\label{5.10112345}
\begin{aligned}
\beta|\hat{x}-\hat{y}|^2
&\leq L(1+2M_0)|\hat{x}-\hat{y}|
+\frac{2^3}{\beta}+{2^{4}}|\hat{x}-\hat{y}|^4
\rightarrow0 \ 
\mbox{as} \ \beta\rightarrow\infty.
\end{aligned}
\end{eqnarray}
Finally, by (\ref{5.66666123}), there is  a sufficiently large number $N_0>0$ such that
$$
L|\hat{x}-\hat{y}|
\leq
\frac{\tilde{m}}{4}
$$
for all $\beta\geq N_0$.
Then,  $\hat{t}\in [\tilde{t},T)$ for all $\beta\geq N_0$. Indeed, if  $\hat{t}=T$,  we  deduce the following contradiction:
\begin{eqnarray*}
	\frac{\tilde{m}}{2}\leq\Psi(\hat{t},\hat{x},\hat{y})\leq \phi(\hat{t},\hat{x})-\phi(\hat{t},\hat{y})\leq
	L|\hat{x}-\hat{y}|
	\leq
	\frac{\tilde{m}}{4}.
\end{eqnarray*}

{ $Step\ 3.$   Crandall-Ishii lemma.}
\par

For every $N\geq1$, we define,
for $(t,x,y)\in [0,T]\times H\times H$,
\begin{eqnarray}\label{060912024}
\begin{aligned}
w_{1}(t,x)&=  W_1(t,x)-\varepsilon\frac{\nu T-t}{\nu
	T}\, |x|^4
-\varepsilon \Upsilon((t,x),(\hat{t},\hat{x}))\\
&\quad
-2\beta\left|x^-_N-(\hat{z}^A_{\hat{t},t})^-_N\right|^2-\sum_{i=0}^{\infty}
\frac{1}{2^i}\Upsilon((t_i,x^i),(t,x)),
\end{aligned}
\end{eqnarray}
\begin{eqnarray}\label{060922024}
\begin{aligned}
w_{2}(t,y)&=-  W_2(t,y)-\varepsilon\frac{\nu T-t}{\nu
	T}\, |y|^4
-\varepsilon \Upsilon((t,y),(\hat{t},\hat{y}))\\
&\quad
-2\beta\left|y_N^--(\hat{z}^A_{\hat{t},t})^-_N\right|^2-\sum_{i=0}^{\infty}
\frac{1}{2^i}\Upsilon((t_i,y^i),(t,y)),
\end{aligned}
\end{eqnarray}
where $\hat{z}=\frac{\hat{x}+\hat{y}}{2}$.   We  note that $w_1,w_2$ depend on $\hat{z}$ and $N$, and thus on $\beta$,
$\varepsilon$ and  $N$.  Define $\varphi\in C^2(H_N\times H_N)$ by
\begin{eqnarray}\label{070632024}
\varphi(x_N,y_N)=\beta|x_N-y_N|^2,\ \ (x_N,y_N)\in H_N\times H_N.
\end{eqnarray}
By Lemma \ref{0611a} below, $w_1$ and $w_2$ satisfy the conditions of  Theorem \ref{theorem0513} with $\varphi$ defined by (\ref{070632024}).
Then by Theorem \ref{theorem0513}, there exist
sequences  $(l_{k},\check{x}^{k}), (s_{k},\check{y}^{k})\in [\hat{t},T]\times H$ and
sequences of functionals 
$(\varphi_{1,k}, \psi_{1,k}, \varphi_{2,k}, \psi_{2,k})$ $\in  \Phi_{\hat{t}}\times   \Phi_{\hat{t}}\times  {\mathcal{G}}_{l_k} \times  {\mathcal{G}}_{s_k}$ bounded from below such that 
\begin{eqnarray}\label{0609a2024}
w_{1}(t,x)-\varphi_{1,k}(t,x)-\varphi_{2,k}(t,x)
\end{eqnarray}
has a strict  maximum $0$ at  $(l_k,\check{x}^{k})$ over $[l_k, T]\times H$,
\begin{eqnarray}\label{0609b2024}
w_{2}(t,y)-\psi_{1,k}(t,y)-\psi_{2,k}(t,y)
\end{eqnarray}
has a strict  maximum $0$ at $(s_k,\check{y}^{k})$ over $ [s_k, T]\times H$, and
\begin{eqnarray}\label{0608v12024}
\begin{aligned}
  &\left(l_{k}, \check{x}^{k};\,  (w_1,(\varphi_{1,k})_{t},\nabla_x\varphi_{1,k},\nabla^2_{x}\varphi_{1,k}, \partial_{t}^o\varphi_{2,k},\nabla_x\varphi_{2,k},\nabla^2_{x}\varphi_{2,k})( l_{k},\check{x}^{k})\right)\\
&\underrightarrow{k\rightarrow\infty}\, \left({\hat{t}},\hat{x};\,  w_1(\hat{t},\hat{x}), b_1, 2\beta (\hat{x}_N-\hat{y}_N), X_N, 0,\mathbf{0},\mathbf{0}\right),
\end{aligned}
\end{eqnarray}
\begin{eqnarray}\label{0608vw12024}
\begin{aligned}
&\left(s_{k}, \check{y}^{k}; \, (w_2, (\psi_{1,k})_{t}, \nabla_x\psi_{1,k}, \nabla^2_{x}\psi_{1,k}, \partial_{t}^o\psi_{2,k}, \nabla_x\psi_{2,k}, \nabla^2_{x}\psi_{2,k})( s_{k},\check{y}^{k})\right)\\
&  \underrightarrow{k\rightarrow\infty}\, \left({\hat{t}},\hat{y};\,  w_2(\hat{t},\hat{y}), b_2, 2\beta(\hat{y}_N-\hat{x}_N), Y_N, 0,\mathbf{0},\mathbf{0}\right),
\end{aligned}
\end{eqnarray}
where $b_{1}+b_{2}=0$ and $X_N,Y_N\in \mathcal{S}(H_N)$ satisfy the following inequality:
\begin{eqnarray}\label{II10}
-6\beta\left(\begin{array}{cc}
I&0\\
0&I
\end{array}\right)\leq \left(\begin{array}{cc}
X_N&0\\
0&Y_N
\end{array}\right)\leq  6\beta \left(\begin{array}{cc}
I&-I\\
-I&I
\end{array}\right).
\end{eqnarray}
We note that  
sequence  $(\check{x}^{k},\check{y}^{k},l_{k},s_{k},\varphi_{1,k},\psi_{1,k},\varphi_{2,k},\psi_{2,k})$ and $b_{1},b_{2},X_N,Y_N$  depend on  $\beta$,
$\varepsilon$ and $N$. We also note that  (\ref{II10}) follows from (\ref{II0615}) choosing $\kappa=\frac{1}{2}\beta^{-1}$. In fact, by (\ref{070632024}),
$$
B= \nabla_x^2\varphi(\hat{x}_N,\hat{y}_N)
=2\beta\left(\begin{array}{cc}
I&-I\\
-I&I
\end{array}\right),
$$
and thus, if $\kappa=\frac{1}{2}\beta^{-1}$,
$$B+\kappa B^2=(1+4\kappa \beta)B=3B,\quad \mbox{and}
\quad -\left(\frac{1}{\kappa}+|B|\right)=-\left(2\beta+4\beta\right)=-6\beta.$$
Then from (\ref{II0615}) it follows that  (\ref{II10}) holds true.  For every $(t,x; s,y)\in ([T-\bar{a},T]\times H)^2$, define
\begin{eqnarray*}
\begin{aligned}
	\chi^{k,N,1}(t,x)  :=&\varphi_{1,k}(t,x)-2\beta\left|x_N-(\hat{z}^A_{\hat{t},t})_N\right|^2,\\[3mm]
	\chi^{k,N,2}(t,x):=& \varepsilon\frac{\nu T-t}{\nu
		T}|x|^4
	+\varepsilon \Upsilon((t,x),(\hat{t},\hat{x}))
	+\sum_{i=0}^{\infty}
	\frac{1}{2^i}\Upsilon((t_i,x^i),(t,x))\\
	&+\varphi_{2,k}(t,x)
	+2\beta\left|x-\hat{z}^A_{\hat{t},t}\right|^2, \\[3mm]
	\chi^{k,N}(t,x)  :=&\chi^{k,N,1}(t,x)+\chi^{k,N,2}(t,x); \quad \hbox{\rm and }
\end{aligned}
\end{eqnarray*}
\begin{eqnarray*}
\begin{aligned}
	\hbar^{k,N,1}(s,y)   :=&\psi_{1,k}(s,y)-2\beta\left|y_N-(\hat{z}^A_{\hat{t},t})_N\right|^2,\\[3mm]
	\hbar^{k,N,2}(s,y) :=& \varepsilon\frac{\nu T-s}{\nu
		T}|y|^4
	+\varepsilon \Upsilon((s,y),(\hat{t},{\hat{y}}))+\sum_{i=0}^{\infty}
	\frac{1}{2^i}\Upsilon((t_i,\eta^i),(s,y))\\
	&
	+\psi_{2,k}(s,y)
	+2\beta\left|y-\hat{z}^A_{\hat{t},t}\right|^2, \\[3mm]
	\hbar^{k,N}(s,y)   :=&\hbar^{k,N,1}(s,y)+\hbar^{k,N,2}(s,y).
\end{aligned}
\end{eqnarray*}
Clearly, we have  $\chi^{k,N,2}\in {\mathcal{G}}_{{l_k}}$ and $\hbar^{k,N,2}\in  {\mathcal{G}}_{{s_k}}$, and in view of  Lemma \ref{lemma03302}, we have  $\chi^{k,N,1},\hbar^{k,N,1}\in  \Phi_{\hat{t}}$.
Moreover, by  (\ref{0609a2024}), (\ref{0609b2024}) and definitions of $w_1$ and $w_2$, we have
$$
\left(  W_1-\chi^{k,N,1}-\chi^{k,N,2}\right)(l_k,\check{x}^k)=\sup_{(t,x)\in [{{l_k}},T]\times H}
\left(  W_1-\chi^{k,N,1}-\chi^{k,N,2}\right)(t,x)\quad \mbox{and}
$$
$$
\left(  W_2+\hbar^{k,N,1}+\hbar^{k,N,2}\right)(s_k,\check{y}^k)=\inf_{(s,y)\in [{s_k},T]\times H}
\left(  W_2+\hbar^{k,N,1}+\hbar^{k,N,2}\right)(s,y).
$$

Since $l_{k},s_{k}\rightarrow {\hat{t}}$ as $k\rightarrow\infty$ and ${\hat{t}}<T$
for $\beta>N_0$, we see that for every fixed $\beta>N_0$,   there is a constant $ K_\beta>0$ such that
$$
|l_{k}|\vee|s_{k}|<T 
\quad \mbox{for all}    \ \ k\geq K_\beta.
$$
Now, for every $\beta>N_0$ and  $k>K_\beta$, 
we have from the definition of viscosity solutions that
\begin{eqnarray}\label{vis12024}
\begin{aligned}
&\partial_{t}^o\chi^{k,N}(l_k,\check{x}^k)
+\left\langle A^*\nabla_{x}(\varphi_{1,k})(l_k,\check{x}^k),\, \check{x}^k\right\rangle_H
-4\beta\left\langle A^*\left((\check{x}^k)_N-(\hat{z}^A_{l_k-\hat{t}})_N\right),\,
\check{x}^k
\right\rangle_H \ \ \ \\
&+ {\mathbf{H}}{(}l_k,\check{x}^k,  (  W_1, \nabla_x\chi^{k,N},
\nabla^2_{x}\chi^{k,N})(l_k,\check{x}^k)
{)}\geq c
\end{aligned}
\end{eqnarray}
and
\begin{eqnarray}\label{vis22024}
&&
-\partial_{t}^o\hbar^{k,N}(s_k,\check{y}^k)-\left\langle A^*\nabla_{x}(\psi_{1,k})(s_k,\check{y}^k),\, \check{y}^k\right\rangle_H
+4\beta\left\langle A^*\left((\check{y}^k)_N-(\hat{z}^A_{s_k-\hat{t}})_N\right),\,
\check{y}^k
\right\rangle_H\nonumber\\
&&+ {\mathbf{H}}{(}s_k,\check{y}^k, (  W_2,
-\nabla_x\hbar^{k,N}, -\nabla^2_{x}\hbar^{k,N})(s_k,\check{y}^k){)}\leq0,
\end{eqnarray}
where, for every $(t,x)\in [{l_k},T]\times{H}$ and  $ (s,y)\in [{s_k},T]\times{H}$, we have from Lemma \ref{lemma03302},
\begin{eqnarray*}
\begin{aligned}
	\partial_{t}^o\chi^{k,N}(t,x):=&(\chi^{k,N,1})_{t}(t,x)+\partial_{t}^o\chi^{k,N,2}(t,x)\\
	=&
	(\varphi_{1,k})_{t}(t,x)+4\beta\left\langle A^*\left(x_N-(\hat{z}^A_{\hat{t},t})_N\right),\,
	\hat{z}^A_{\hat{t},t}\right\rangle_H\\
	&-\frac{\varepsilon}{\nu T}|x|^4
	+2\varepsilon({t}-{\hat{t}})+2\sum_{i=0}^{\infty}\frac{1}{2^i}(t-t_{i})+\partial_{t}^o\varphi_{2,k}(t,x),
\end{aligned}
\end{eqnarray*}
\begin{eqnarray*}
\begin{aligned}
	\nabla_x\chi^{k,N}(t,x):=&\nabla_x\chi^{k,N,1}(t,x)+\nabla_x\chi^{k,N,2}(t,x)\\
=&\nabla_{x}(\varphi_{1,k})(t,x)
	+4\beta(x^-_N-(\hat{z}^A_{\hat{t},t})^-_N)
	+\varepsilon\frac{\nu T-{t}}{\nu T}\nabla_x|x|^4\\
& +\varepsilon\nabla_x|x-\hat{x}^A_{\hat{t},t}|^4
	+\sum_{i=0}^{\infty}\frac{1}{2^i}\nabla_x|x-(x^i)^A_{t_i,t}|^4
	+\nabla_{x}(\varphi_{2,k})(t,x),
\end{aligned}
\end{eqnarray*}
\begin{eqnarray*}
\begin{aligned}
	\nabla^2_{x}\chi^{k,N}(t,x):=&\nabla^2_{x}\chi^{k,N,1}(t,x)+\nabla^2_{x}\chi^{k,N,2}(t,x)\\
	=&\nabla^2_{x}(\varphi_{1,k})(t,x)+4\beta Q_N+\varepsilon\frac{\nu T-{t}}{\nu T}\nabla^2_x|x|^4 \\
	&+\varepsilon\nabla^2_x|x-\hat{x}^A_{\hat{t},t}|^4
+\sum_{i=0}^{\infty}\frac{1}{2^i}\nabla^2_x|x-(x^i)^A_{t_i,t}|^4
+\nabla^2_{x}(\varphi_{2,k})(t,x),
     \end{aligned}
\end{eqnarray*}
\begin{eqnarray*}
\begin{aligned}
	\partial_{t}^o\hbar^{k,N}(s,y):=&(\hbar^{k,N,1}))_{t}(s,y)+\partial_{t}^o\hbar^{k,N,2}(s,y)\\
	=&
	(\psi_{1,k})_{t}(s,y)+4\beta\left\langle A^*\left(y_N-(\hat{z}^A_{\hat{t},s})_N\right),\,
	\hat{z}^A_{\hat{t},s}\right\rangle_H\\
	&-\frac{\varepsilon}{\nu T}|y|^4+2\varepsilon({s}-{\hat{t}})
	+2\sum_{i=0}^{\infty}\frac{1}{2^i}(s-t_{i})
	+\partial_{t}^o\psi_{2,k}(s,y),
\end{aligned}
\end{eqnarray*}
\begin{eqnarray*}
\begin{aligned}
	\nabla_x\hbar^{k,N}(s,y):=&\nabla_x\hbar^{k,N,1}(s,y)+\nabla_x\hbar^{k,N,2}(s,y)\\
	=& \nabla_{x}\psi_{1,k}(s,y)+4\beta(y^-_N-(\hat{z}^A_{\hat{t},s})^-_N)
	+\varepsilon\frac{\nu T-{s}}{\nu T}
	\nabla_x|y|^4\\
&+\varepsilon\nabla_x|y-\hat{y}^A_{\hat{t},s}|^4
+\sum_{i=0}^{\infty}\frac{1}{2^i}\nabla_x|y-(y^i)^A_{t_i,s}|^4
	+\nabla_{x}\psi_{2,k}(s,y),\quad \mbox{and}
\end{aligned}
\end{eqnarray*}
\begin{eqnarray*}
\begin{aligned}
	\nabla^2_{x}\hbar^{k,N}(s,y):=&\nabla^2_{x}\hbar^{k,N,1}(s,y)+\nabla^2_{x}\hbar^{k,N,2}(s,y)\\
	=&\nabla^2_{x}\psi_{1,k}(s,y)+4\beta Q_N+\varepsilon\frac{\nu T-{s}}{\nu T}
	\nabla^2_x|y|^4\\
	&+\varepsilon\nabla^2_x|y-\hat{y}^A_{\hat{t},s}|^4
+\sum_{i=0}^{\infty}\frac{1}{2^i}\nabla^2_x|y-(y^i)^A_{t_i,s}|^4
+\nabla^2_{x}\psi_{2,k}(s,y).
\end{aligned}
\end{eqnarray*}
\par
{ $Step\ 4.$  Calculation and  completion of the proof.}
\par
Letting 
$k\rightarrow\infty$ in (\ref{vis12024}) and (\ref{vis22024}), and using (\ref{0608v12024}) and (\ref{0608vw12024}),  we have
\begin{eqnarray}\label{031032024}
\begin{aligned}
& b_1+4\beta\langle A^*(\hat{x}_N-\hat{z}_N),\,
\hat{z}\rangle_H-\frac{\varepsilon}{\nu T}|\hat{x}|^4+2\sum_{i=0}^{\infty}\frac{1}{2^i}(\hat{t}-t_i)\\
&+2\beta\left\langle A^*\left(\hat{x}_N-\hat{y}_N\right),\,\hat{x}\right\rangle_H
-4\beta\left\langle A^*\left(\hat{x}_N-\hat{z}_N\right),\,
\hat{x}
\right\rangle_H\\
&+ {\mathbf{H}}(\hat{t},\hat{x}, (  W_1,
\nabla_x\chi^{N},\nabla^2_{x}\chi^{N})(\hat t,\hat{x}))
\geq c\quad \mbox{and}
\end{aligned}
\end{eqnarray}
\begin{eqnarray}\label{031042024}
\begin{aligned}
& -b_2-4\beta\left\langle A^*\left(\hat{y}_N-\hat{z}_N\right),\,
\hat{z}\right\rangle_H+ \frac{\varepsilon}{\nu T}|\hat{y}|^4-2\sum_{i=0}^{\infty}\frac{1}{2^i}(\hat{t}-t_i)\\
&+2\beta \left\langle A^*\left(\hat{x}_N-\hat{y}_N\right),\, \hat{y}\right\rangle _H
+4\beta\left\langle A^*\left(\hat{y}_N-\hat{z}_N\right),\,
\hat{y}
\right\rangle_H\\
&+ {\mathbf{H}}(\hat t, \hat{y}, (  W_2,-\nabla_x\hbar^{N},-\nabla^2_{x}\hbar^{N})(\hat t,\hat{y}))
\leq0,
\end{aligned}
\end{eqnarray}
where
\begin{eqnarray*}
\begin{aligned}
\nabla_x\chi^{N}(\hat{t}, \hat{x})
&:=
2\beta(\hat{x}-\hat{y})
 +\varepsilon\frac{\nu T-{\hat{t}}}{\nu T}\nabla_x|\hat{x}|^4
+\sum_{i=0}^{\infty}\frac{1}{2^i}\nabla_x|\hat{x}-(x^i)^A_{t_i,\hat{t}}|^4,\\
\nabla^2_{x}\chi^{N}(\hat{t}, \hat{x})
&:= X_N+4\beta Q_N
 +\varepsilon\frac{\nu T-{\hat{t}}}{\nu T}\nabla^2_{x}|\hat{x}|^4
+\sum_{i=0}^{\infty}\frac{1}{2^i}\nabla^2_{x}|\hat{x}-(x^i)^A_{t_i,\hat{t}}|^4,\\
\nabla_x\hbar^{N}(\hat{t},\hat{y})
&:=
2\beta(\hat{y}-\hat{x})
 +\varepsilon\frac{\nu T-{\hat{t}}}{\nu T}\nabla_x|\hat{y}|^4+\sum_{i=0}^{\infty}\frac{1}{2^i}
\nabla_x|\hat{y}-(y^i)^A_{t_i,\hat{t}}|^4\ \mbox{and}\\
\nabla^2_{x}\hbar^{N}(\hat{t}, \hat{y})
&:= Y_N+4\beta Q_N
+\varepsilon\frac{\nu T-{\hat{t}}}{\nu T}\nabla^2_{x}|\hat{y}|^4
+\sum_{i=0}^{\infty}\frac{1}{2^i}
\nabla^2_{x}|\hat{y}-(y^i)^A_{t_i,\hat{t}}|^4.
\end{aligned}
\end{eqnarray*}
Since $b_1+b_2=0$ and $\hat{z}=\frac{\hat{x}+\hat{y}}{2}$, subtracting
both sides of   (\ref{031032024}) from those  of (\ref{031042024}), we observe the following cancellations:
\begin{eqnarray*}
 4\beta\langle A^*(\hat{x}_N-\hat{z}_N),\,
\hat{z}\rangle_H-\left[-4\beta\left\langle A^*\left(\hat{y}_N-\hat{z}_N\right),\,
\hat{z}\right\rangle_H\right]=4\beta\langle A^*(\hat{x}_N+\hat{y}_N-2\hat{z}_N),\, \hat{z}\rangle_H=0,
\end{eqnarray*}
\begin{eqnarray*}
 2\beta\left\langle A^*\left(\hat{x}_N-\hat{y}_N\right),\,\hat{x}\right\rangle_H-2\beta \left\langle A^*\left(\hat{x}_N-\hat{y}_N\right),\, \hat{y}\right\rangle _H
=2\beta\left\langle A^*\left(\hat{x}_N-\hat{y}_N\right),\,\hat{x}-\hat{y}\right\rangle_H,
\end{eqnarray*}
 \begin{eqnarray*}
 \begin{aligned}
&-4\beta\left\langle A^*\left(\hat{x}_N-\hat{z}_N\right),\,
\hat{x}
\right\rangle_H-
4\beta\left\langle A^*\left(\hat{y}_N-\hat{z}_N\right),\,
\hat{y}
\right\rangle_H\\
=&-2\beta\langle A^*(\hat{x}_N-\hat{y}_N),\, \hat{x}\rangle_H-
2\beta\left\langle A^*\left(\hat{y}_N-\hat{x}_N\right),\,
\hat{y}
\right\rangle_H=-2\beta\left\langle A^*\left(\hat{x}_N-\hat{y}_N\right),\,\hat{x}-\hat{y}\right\rangle_H;
\end{aligned}
\end{eqnarray*}
%
{\it all the terms which involve the unbounded operator $A$  mutually cancels out},  and  thus we have
\begin{eqnarray}\label{vis1122024}
\begin{aligned}
&\quad c+ \frac{\varepsilon}{\nu T}(|\hat{x}|^4+|\hat{y}|^4)-4\sum_{i=0}^{\infty}\frac{1}{2^i}(\hat{t}-t_i)\\
&\leq {\mathbf{H}}(\hat{t},\hat{x}, (  W_1,\nabla_x\chi^{N},\nabla^2_{x}\chi^{N})(\hat{t},\hat{x}))- {\mathbf{H}}(\hat{t},\hat{y}, (  W_2,-\nabla_x\hbar^{N},-\nabla^2_{x}\hbar^{N})(\hat{t},\hat{y})).
\end{aligned}
\end{eqnarray}
On the  other hand, from  (\ref{5.1}) and via a simple calculation, we have
\begin{eqnarray}\label{v4}
\begin{aligned}
&{\mathbf{H}}(\hat{t},\hat{x}, (  W_1,\nabla_x\chi^{N},\nabla^2_{x}\chi^{N})(\hat{t},\hat{x}))- {\mathbf{H}}(\hat{t},\hat{y}, (  W_2,-\nabla_x\hbar^{N},-\nabla^2_{x}\hbar^{N})(\hat{t},\hat{y}))\\
\leq & {\mathbf{H}}(\hat{t},\hat{x},   W_2(\hat{t},\hat{y}),(\nabla_x\chi^{N},\nabla^2_{x}\chi^{N})(\hat{t},\hat{x}))- {\mathbf{H}}(\hat{t},\hat{y}, (  W_2,-\nabla_x\hbar^{N},-\nabla^2_{x}\hbar^{N})(\hat{t},\hat{y}))\\
\leq &\sup_{u\in U}(J_{1}+J_{2}+J_{3}).
\end{aligned}
\end{eqnarray}
Here,
from Assumption \ref{hypstate} (ii)   and (\ref{II10}),  we have
\begin{eqnarray}\label{j22024}
\begin{aligned}
J_{1}:=&\frac{1}{2}\mbox{Tr}\left[\nabla^2_{x}\chi^{N}(\hat{t}, \hat{x})\, (  \sigma\,   \sigma^*)(\hat{t}, \hat{x},u)\right]-\frac{1}{2}\mbox{Tr}\left[-\nabla^2_{x}\hbar^{N}(\hat{t}, \hat{y})\, (  \sigma\,   \sigma^*)(\hat{t}, \hat{y},u)\right]\\
\leq&\, 3\beta\left|  \sigma(\hat{t}, \hat{x},u)-  \sigma(\hat{t}, \hat{y},u)\right|_{L_2(\Xi,H)}^2\\
&+2\beta\left (\left|Q_N  \sigma(\hat{t}, \hat{x},u)\right|_{L_2(\Xi,H)}^2+\left |Q_N  \sigma(\hat{t}, \hat{y},u)\right|_{L_2(\Xi,H)}^2\right)\\
&+6\varepsilon\frac{\nu T-{\hat{t}}}{\nu    T}\left(|\hat{x}|^2\, \left|  \sigma(\hat{t}, \hat{x},u)\right|_{L_2(\Xi,H)}^2+|\hat{y}|^2\, \left|  \sigma(\hat{t}, \hat{y},u)\right|_{L_2(\Xi,H)}^2\right)
\\
&
+6\left|  \sigma(\hat{t}, \hat{x},u)\right|_{L_2(\Xi,H)}^2\sum_{i=0}^{\infty}\frac{1}{2^i}\left|(x^i)^A_{t_i,\hat{t}}- \hat{x}\right|^2\\
&
+6\left|  \sigma(\hat{t}, \hat{y},u)\right|_{L_2(\Xi,H)}^2\sum_{i=0}^{\infty}\frac{1}{2^i}\left|(y^i)^A_{t_i,\hat{t}}- \hat{y}\right|^2
\\
\leq&
\, 3\beta{L^2}\, \left|\hat{x}-\hat{y}\right|^2+2\beta\left(\left|Q_N  \sigma(\hat{t}, \hat{x},u)\right|_{L_2(\Xi,H)}^2+\left|Q_N  \sigma(\hat{t}, \hat{y},u)\right|_{L_2(\Xi,H)}^2\right)\\
&+6L^2(1+|\hat{x}|^2
+|\hat{y}|^2)\sum_{i=0}^{\infty}\frac{1}{2^i}\left[\left|(x^i)^A_{t_i,\hat{t}}-\hat{x}\right|^2
+\left|(y^i)^A_{t_i,\hat{t}}- \hat{y}\right|^2\right]
\\
&+12\varepsilon \frac{\nu T-{\hat{t}}}{\nu T}L^2(1+| \hat{x}|^4
+|\hat{y}|^4);
\end{aligned}
\end{eqnarray}
from Assumption \ref{hypstate} (ii),  we have
\begin{eqnarray}\label{j12024}
\begin{aligned}
J_{2}:=& {\left\langle   b(\hat{t}, \hat{x},u),\, \nabla_x\chi^{N}(\hat{t}, \hat{x})\right\rangle_{H}  -\left\langle   b(\hat{t}, \hat{y},u),\, -\nabla_x\hbar^{N}(\hat{t},\hat{y})\right\rangle_{H}}\\
\leq&\, 2\beta{L}\, \left| \hat{x}-\hat{y}\right|^2 +8\varepsilon \frac{\nu T-{\hat{t}}}{\nu T} L(1+|\hat{x}|^4+|\hat{y}|^4)\\
&+4L(1+|\hat{x}|+|\hat{y}|)\sum_{i=0}^{\infty}\frac{1}{2^i}\left[\left|(x^i)^A_{t_i,\hat{t}}-\hat{x}\right|^3
+\left|(x^i)^A_{t_i,\hat{t}}-\hat{y}\right|^3\right];
\end{aligned}
\end{eqnarray}
from Assumption \ref{hypstate} (iii),  we have
\begin{eqnarray}\label{j32024}
\begin{aligned}
J_{3}:=&  q{(}\hat{t}, \hat{x},   W_2(\hat{t}, \hat{y}), \nabla_x\chi^{N}(\hat{t}, \hat{x})  \sigma(\hat{t}, \hat{x},u),u{)}\\
&-
  q(\hat{t}, \hat{y},   W_2(\hat{t}, \hat{y}), -\nabla_x\hbar^{N}(\hat{t}, \hat{y})  \sigma(\hat{t}, \hat{y},u),u{)}\\
\leq&
L|\hat{x}-\hat{y}|
+2\beta L^2|\hat{x}
-\hat{y}|^2
+8\varepsilon \frac{\nu T-{\hat{t}}}{\nu T} L^2(1+|\hat{x}|^4
+|\hat{y}|^4)\\
&
+4L^2(1+|\hat{x}|+|\hat{y}|)\sum_{i=0}^{\infty}\frac{1}{2^i}\left[\left|(x^i)^A_{t_i,\hat{t}}-\hat{x}\right|^3+\left|(y^i)^A_{t_i,\hat{t}}-\hat{y}\right|^3\right]
.
\end{aligned}
\end{eqnarray}
In view of  the property (i) of $(\hat{t}, \hat{x},\hat{y})$, we have
\begin{eqnarray}\label{1002d12024}
4\sum_{i=0}^{\infty}\frac{1}{2^i}(\hat{t}-t_i)
\leq4\sum_{i=0}^{\infty}\frac{1}{2^i}\bigg{(}\frac{1}{2^i\beta}\bigg{)}^{\frac{1}{2}}\leq 8{\bigg{(}\frac{1}{{\beta}}\bigg{)}}^{\frac{1}{2}},
\end{eqnarray}
\begin{eqnarray}\label{1002d2}
\begin{aligned}
\sum_{i=0}^{\infty}\frac{1}{2^i}\left[\left|(x^i)^A_{t_i,\hat{t}}-\hat{x}\right|^3
+\left|(y^i)^A_{t_i,\hat{t}}-\hat{y}\right|^3\right]\leq
2\sum_{i=0}^{\infty}\frac{1}{2^i}\bigg{(}\frac{1}{2^i\beta}\bigg{)}^{\frac{3}{4}}\leq 4{\bigg{(}\frac{1}{{\beta}}\bigg{)}}^{\frac{3}{4}},\quad \mbox{and}
\end{aligned}
\end{eqnarray}
\begin{eqnarray}\label{1002d3}
\begin{aligned}
\sum_{i=0}^{\infty}\frac{1}{2^i}\left[\left|(x^i)^A_{t_i,\hat{t}}-\hat{x}\right|^2
+\left|(y^i)^A_{t_i,\hat{t}}-\hat{y}\right|^2\right]\leq
2\sum_{i=0}^{\infty}\frac{1}{2^i}\bigg{(}\frac{1}{2^i\beta}\bigg{)}^{\frac{1}{2}}\leq 4{\bigg{(}\frac{1}{{\beta}}\bigg{)}}^{\frac{1}{2}};
\end{aligned}
\end{eqnarray}
and since $\hat{x}$ and $\hat{y}$ are independent of $N$, by Assumption  \ref{hypstate5666},
\begin{eqnarray}\label{1002d4}
\qquad\qquad \sup_{u\in U}\left[\left|Q_N  \sigma(\hat{t}, \hat{x},u)\right|_{L_2(\Xi,H)}^2+\left|Q_N  \sigma(\hat{t}, \hat{y},u)\right|_{L_2(\Xi,H)}^2\right]\rightarrow0\ \mbox{as}\ N\rightarrow \infty.
\end{eqnarray}
Combining (\ref{vis1122024})-(\ref{j32024}),  and letting $N\rightarrow\infty$ and then 
by (\ref{5.10jiajiaaaa}) and (\ref{5.10}), we have for sufficiently large $\beta>0$,  
\begin{eqnarray}\label{vis1222024}
\qquad\qquad
c
\leq
-\frac{\varepsilon}{\nu T}(|\hat{x}|^4
+|\hat{y}|^4)+ \varepsilon \frac{\nu T-{\hat{t}}}{\nu T} (20L+8)L(1+|\hat{x}|^4
+|\hat{y}|^4
)+\frac{c}{4}.
\end{eqnarray}
Since $$
\nu=1+\frac{1}{2T(20L+8)L}
\quad \text{and} \quad\bar{a}=\frac{1}{2(20L+8)L}\wedge{T},$$
we see by  (\ref{5.32024}) the following contradiction:
\begin{eqnarray*}\label{vis1222024}
	c\leq
	\frac{\varepsilon}{\nu
		T}+\frac{c}{4}\leq \frac{c}{2}.
\end{eqnarray*}
The proof is  complete.   \end{proof} 

In Lemmas below of this {section}, let $\widetilde{w}_{1}^{\hat{t},N}, \widetilde{w}_{1}^{\hat{t},N,*}$ and $\widetilde{w}_{2}^{\hat{t},N}, \widetilde{w}_{2}^{\hat{t},N,*}$ be defined in Definition \ref{definition06072024} with respect to $w_1$ defined by (\ref{060912024}) and  $w_2$ defined by  (\ref{060922024}), respectively. To complete the previous proof, it remains to state and prove the following lemmas.

\begin{lemma}\label{lemma03302}\ \
	For every fixed $(t, a) \in [0,T)\times H$ and $N\in \mathbb{N}^+$, define   $f:[t,T]\times H\rightarrow \mathbb{R}$ by
	$$f(s,x):=\left|x_N-\left(a^A_{t,s}\right)_N\right|^2,\quad (s,x)\in [t, T]\times H.$$ Then
	$f\in \Phi_t$.
\end{lemma}

\begin{proof}
First, it is clear that
$$
\nabla_xf(s,x)=2x_N-2\left(a^A_{t,s}\right)_N;\ \ \ \nabla^2_{x}f(s,x)=2P_N,\ \ \ (s,x)\in [t, T]\times H,
$$
and $(f,\nabla_xf, \nabla^2_{x}f)$ are continuous and grow in a polynomial way. Moreover, $f\in C^{0+}([t, T]\times H)$.
Second, we consider $f_{t}$. For every $(s,x)\in (t, T)\times H$,
\begin{eqnarray*}
	\begin{aligned}
		f_{t}(s,x)
		&=
		\lim_{h\to 0}\frac{1}{h}\left(\left|x_N-\left(a^A_{t,(s+h)\vee t}\right)_N\right|^2-\left|x_N-\left(a^A_{t,s}\right)_N\right|^2\right)\\
		&=\lim_{h\to 0}\frac{1}{h}\left(\sum^{N}_{i=1}\left\langle x-a^A_{t,(s+h)\vee t},\,  e_i\right\rangle_H^2-\sum^{N}_{i=1}\left\langle x-a^A_{t,s}, \, e_i\right\rangle_H^2\right).
	\end{aligned}
\end{eqnarray*}
Recall that $e^{sA^*}:=(e^{sA})^*$ and  $e^{sA^*}$ is a $C_0$ {semi-group} on $H$ generated by $A^*$. Since, for all $i=1,2,\ldots, N$,
\begin{eqnarray*}
	\begin{aligned}
		&\quad\left\langle x-a^A_{t,(s+h)\vee s}, \, e_i\right\rangle_H^2-\left\langle x-a^A_{t,s},\, e_i\right\rangle_H^2\\
		&=-\left\langle 2x-a^A_{t,(s+h)\vee t}-a^A_{t,s}, \, e_i\right\rangle_H\times \left\langle a^A_{t,(s+h)\vee t}-a^A_{t,s}, \, e_i\right\rangle_H,
	\end{aligned}
\end{eqnarray*}
\begin{eqnarray*}
	\lim_{h\to 0}\left\langle 2x-a^A_{t,(s+h)\vee t}-a^A_{t,s},\,  e_i\right\rangle_H=2\left\langle x-a^A_{t,s},\, e_i\right\rangle_H,
\end{eqnarray*}
and \begin{eqnarray*}
	\lim_{h\to 0^+}\frac{1} {h}\left\langle a^A_{t,(s+h)\vee t}-a^A_{t,s}, \, e_i\right\rangle_H=\lim_{h\to 0^+}\frac{1} {h}\left\langle a^A_{t,s}, \, e^{hA^*}e_i-e_i\right\rangle_H=\left\langle a^A_{t,s}, \, A^*e_i\right\rangle_H,
\end{eqnarray*}
\begin{eqnarray*}
	\lim_{h\to 0^-}\frac{1} {h}\left\langle a^A_{t,(s+h)\vee t}-a^A_{t,s}, \, e_i\right\rangle_H=\lim_{h\to 0^-}\frac{1} {h}\left\langle a^A_{t,(s+h)\vee t}, \, e_i-e^{-hA^*}e_i\right\rangle_H=\left\langle a^A_{t,s}, \, A^*e_i\right\rangle_H,
\end{eqnarray*}
we have
\begin{eqnarray}\label{f-t}
\begin{aligned}
	f_{t}(s,x)=-2\sum^{N}_{i=1}\left\langle x-a^A_{t,s}, \, e_i\right\rangle_H\times\left\langle a^A_{t,s},\,  A^*e_i\right\rangle_H
	=-2\left\langle a^A_{t,s}, \, A^*\left(x-a^A_{t,s}\right)_N\right\rangle_H.
\end{aligned}
\end{eqnarray}
It is easy to see that the above equality also holds at $s=t$ for the right time derivative and at $s=T$ for the left time derivative.

It is sufficient to prove $f_{t}\in C^0_p([t,T]\times H)$ and $A^*\nabla_xf\in C^0_p([t,T]\times H,H)$. Write $\bar{M}:=\max_{1\leq i\leq N}|A^*e_i|$. Then, for every $(s,x; l,y)\in ([t, T]\times H)^2,$
\begin{eqnarray*}
\begin{aligned}
	&\left|A^*\left(x-a^A_{t,s}\right)_N-A^*\left(y-a^A_{t,l}\right)_N\right|=\left|\sum^{N}_{i=1}\left\langle x-a^A_{t,s}-y+a^A_{t,l},\,  e_i\right\rangle_HA^*e_i \right|\\
	\leq&\bar{M}\sum^{N}_{i=1}\left|\left\langle x-a^A_{t,s}-y+a^A_{t,l}, \, e_i\right\rangle_H\right|
\leq N\bar{M}\left[|x-y|+\left|a^A_{t,s}-a^A_{t,l}\right|\right].
\end{aligned}
\end{eqnarray*}
Thus, 
we have $f_{t}\in C^0([t,T]\times H)$ and $A^*\nabla_xf\in C^0([t,T]\times H,H)$. By a simple calculation, we see
that $f_{t}$ and  $A^*\nabla_xf$ grow in a polynomial way. The proof is complete.
\end{proof}

\begin{lemma}\label{lemma4.3}\ \
	There exists a local modulus of continuity  $\rho_1$ 
	such that  the functionals $w_1$ and $w_2$ defined by (\ref{060912024}) and (\ref{060922024}) satisfy condition (\ref{0608a2024}).
\end{lemma}

\begin{proof}
By the definition of $w_{1}$, we have that,
for every $\hat{t}\leq t\leq s\leq T$ and $x\in H$,
\begin{eqnarray*}
\begin{aligned}
	w_1(t,x)-w_1(s,x_{t,s}^{A,N})
	 =&W_1(t,x)-\Pi(t,x)-2\beta\left|x^-_N-\left(\hat{z}^A_{\hat{t},t}\right)^-_N\right|^2-W_1(s,x_{t,s}^{A,N})\\
	 &+\Pi(s,x_{t,s}^{A,N})+2\beta\left|\left(x^A_{t,s}\right)^-_N-\left(\hat{z}^A_{\hat{t},s}\right)^-_N\right|^2,
\end{aligned}
\end{eqnarray*}
where
$$
\Pi(l,x):=\varepsilon\frac{\nu T-l}{\nu
	T}|x|^4
+\varepsilon {\Upsilon}((l,x),(\hat{t},\hat{x}))
+\sum_{i=0}^{\infty}
\frac{1}{2^i}\, {\Upsilon}((t_i,x^i),(l,x)), \ \ (l,x)\in [\hat{t},T]\times H.
$$
Since the operator $A$ generates a $C_0$ contraction
{semi-group}  of bounded linear operators, we have
$$|x|\geq |x^A_{t,s}|, \quad |x-x^A_{\hat{t},t}|\geq |x^A_{t,s}-x^A_{\hat{t},s}|,\quad |x-(x^i)^A_{t_i,t}|\geq |x^A_{t,s}-(x^i)^A_{t_i,s}|.
 $$
Thus,
\begin{eqnarray*}
	&&w_1(t,x)-w_1(s,x_{t,s}^{A,N})\\
	 &\leq&W_1(t,x)-\Pi(s,x_{t,s}^A)+2\beta\left|x_N-(\hat{z}^A_{\hat{t},t})_N\right|^2-W_1(s,x_{t,s}^{A,N})+\Pi(s,x_{t,s}^{A,N})\\
&&-2\beta\left|(x^A_{t,s})_N-(\hat{z}^A_{\hat{t},s})_N\right|^2+\varepsilon\left[(s-\hat{t}\,)^2-(t-\hat{t}\,)^2\right]+\sum_{i=0}^{\infty}
	\frac{1}{2^i}\left[(s-t_i)^2-(t-t_i)^2\right].
\end{eqnarray*}
From (\ref{w12024}),
\begin{eqnarray*}
	&&W_1(t,x)-W_1(s,x_{t,s}^{A,N})= W_1(t,x)-W_1(s,x_{t,s}^A)+ W_1(s,x_{t,s}^A)-W_1(s,x_{t,s}^{A,N})\\
	&\leq& \rho_2(|s-t|,|x|) +4L\, (1+|x|)\left |x_N-(x^A_{t,s})_N\right|;
\end{eqnarray*}
noting that, for all $x,y\in H$,
\begin{eqnarray*}
	\left||x|^4-|y|^4\right|\leq C\left(|x|^3\vee|y|^3\right)|x-y|,
\end{eqnarray*}
we have
\begin{eqnarray*}
	|\Pi(s,x_{t,s}^A)-\Pi(s,x_{t,s}^{A,N})|\leq C\left(|x|^3+|\hat{x}|^3+\sum_{i=0}^{\infty}\frac{1}{2^i}|x^i|^3\right)
      \left|x_N-(x^A_{t,s})_N\right|;
\end{eqnarray*}
since, for $ 0\leq t\leq s\leq T, \ x\in H$,
\begin{eqnarray*}
\begin{aligned}
	&\left|(x^A_{t,s})_N-x_N\right|=\left|\sum^{N}_{i=1}\left\langle e^{\left(s-t\right)A}x-x, e_i\right\rangle_He_i\right|  \\
	=&\left|\sum^{N}_{i=1}\left\langle x, e^{\left(s-t\right)A^*}e_i-e_i\right\rangle_He_i\right|\leq \sum^{N}_{i=1}|x|\left|e^{\left(s-t\right)A^*}e_i-e_i\right|,
\end{aligned}
\end{eqnarray*}
we have
\begin{eqnarray*}
	\begin{aligned}
		&2\beta\left|x_N-(\hat{z}^A_{\hat{t},t})_N\right|^2-2\beta\left|(x^A_{t,s})_N-(\hat{z}^A_{\hat{t},s})_N\right|^2
		 \leq4\beta(|x|+|\hat{z}|)\left(\left|x_N-(x^A_{t,s})_N\right|
		+\left|\hat{z}^A_{t,s}-\hat{z}\right|\right)\\
		&\leq 4\beta(|x|+|\hat{z}|)\left(|x|\sum^{N}_{i=1}\left|e^{\left(s-t\right)A^*}e_i-e_i\right|
		+\left|\hat{z}^A_{t,s}-\hat{z}\right|\right).
	\end{aligned}
\end{eqnarray*}
Taking for $(l,\lambda)\in[0,\infty)\times [0,\infty)$,
\begin{eqnarray*}
\begin{aligned}
\rho_1(l,\lambda):=& \rho_2(l,\lambda)+(2\varepsilon+4)Tl+\biggl [4L(1+\lambda)+C\Bigl(\lambda^3+|\hat{x}|^3+\sum_{i=0}^{\infty}\frac{1}{2^i}|x^i|^3\Bigr)\\
&+
4\beta(\lambda+|\hat{z}|)\biggr]\times  \left(\lambda\sum^{N}_{i=1}\left|e^{lA^*}e_i-e_i\right|
+\left|e^{lA}\hat{z}-\hat{z}\right|\right),
\end{aligned}
\end{eqnarray*}
we see that $\rho_1$ is a local modulus of continuity 
and $w_1$  satisfies condition (\ref{0608a2024}) with it.  In a similar way, we show that  $w_2$  satisfies condition (\ref{0608a2024}) with this $\rho_1$.
The proof is now complete.
\end{proof}

\begin{lemma}\label{lemma4.344} We have $\widetilde{w}_{1}^{\hat{t},N,*}\in \Phi_N(\hat{t},\hat{x}_N)$ and $\widetilde{w}_{2}^{\hat{t},N,*}\in \Phi_N(\hat{t},\hat{y}_N)$. 
\end{lemma}

\begin{proof}  We only prove the first inclusion $\widetilde{w}_{1}^{\hat{t},N,*}\in \Phi_N(\hat{t},\hat{x}_N)$. The second inclusion
$\widetilde{w}_{2}^{\hat{t},N,*}\in \Phi_N(\hat{t},\hat{y}_N)$ is proved in a symmetric  way.

Set $r=\frac{1}{2}(|T-{\hat{t}}|\wedge \hat{t}\,)$. For  given $L>0$, let $\varphi\in C^{1,2}([0,T]\times H_N)$ be a function such that the function
$(\widetilde{w}_{1}^{\hat{t},N,*}-\varphi)(t,x_N)$  attains the  maximum at $(\bar{{t}},\bar{x}_N)\in (0, T)\times H_N$, and moreover, the following are satisfied:
\begin{eqnarray}\label{0815zhou1}
|\bar{{t}}-{\hat{t}}|+\left|\bar{x}_N-\hat{x}_N\right|<r=\frac{1}{2}\left(|T-{\hat{t}}|\wedge \hat{t}\,\right),
\end{eqnarray}
\begin{eqnarray}\label{0815zhou2}
\left|\widetilde{w}_{1}^{\hat{t},N,*}(\bar{{t}},\bar{{x}}_N)\right|+\left|\nabla_x\varphi(\bar{{t}},\bar{{x}}_N)\right|
+\left|\nabla^2_x\varphi(\bar{{t}},\bar{{x}}_N)\right|\leq L.
\end{eqnarray}
By \cite[Lemma 5.4, Chapter 4 ]{yong}, we  modify $\varphi$  such  that $\varphi\in C^{1,2}([0,T]\times H_N)$ bounded from below, ${\varphi}$, ${\varphi_t}$, $\nabla_{x}{\varphi}$ and $\nabla^2_{x}{\varphi}$  grow  in a polynomial way,
$\widetilde{w}_{1}^{\hat{t},N,*}(t,x_N)-\varphi(t,x_N)$  has a strict   maximum 0 at $(\bar{{t}},\bar{x}_N)\in (0, T)\times H_N$ on $[0, T]\times H_N$ and the above two inequalities hold true.
If $\bar{{t}}<\hat{t}$, we have $\varphi_{t}(\bar{{t}},\bar{x}_N)=\frac{1}{2}(\hat{t}-\bar{{t}})^{-\frac{1}{2}} > 0$. If $\bar{{t}}\geq \hat{t}$,  recall that $w_1$ is defined in (\ref{060912024}), we consider the functional
$$
\Gamma(t,x)= w_{1}(t,x)
-\varphi(t,x_N),\ (t,x)\in [\hat{t},T]\times H.
$$
Note that $w_1\in USC^{0+}([0,T]\times H)$
bounded from above and satisfies (\ref{051312024}),  and $\varphi$ is a continuous  function defined on $[0,T]\times H_N$ bounded from below, 
then $\Gamma\in  USC^{0+}([\bar{t},T]\times H)$  bounded from above and satisfies
$$
	\limsup_{|x|\rightarrow\infty}\sup_{t\in[\bar{t},T]}\left[\frac{\Gamma(t,x)}{|x|}\right]<0.
$$
%
%
Take
$\delta_i:=\frac{1}{2^i}$ for all $i\geq0$.
 For every  $0<\delta<1$, by Lemma \ref{theoremleft} we have that,
for every  $(\breve{t}_0,\breve{x}^0)\in [\bar{t},T]\times H$  {satisfying}
\begin{eqnarray}\label{0615a}
\Gamma(\breve{t}_0,\breve{x}^0)\geq \sup_{(s,x)\in [\bar{t},T]\times H}\Gamma(s,x)-\delta,
\end{eqnarray}
there exist $(\breve{t},\breve{x})\in [\bar{t},T]\times H$, a sequence $\{(\breve{t}_i,\breve{x}^i)\}_{i\geq1}\subset
[\bar{t},T]\times H$ and constant $C>0$ such that
\begin{description}
	\item[(i)] 
	${\Upsilon}((\breve{t}_i,\breve{x}^i),(\breve{t},\breve{x}))\leq \frac{\delta}{2^i}$, $|\breve{x}^i|\leq C$ for all $i\geq0$,  and $t_i\uparrow \breve{t}$ as $i\rightarrow\infty$,
	\item[(ii)]  $\mathbf{\Gamma}(\breve{t},\breve{x}):=\Gamma(\breve{t},\breve{x})
	-\sum_{i=0}^{\infty}\frac{1}{2^i}{\Upsilon}((\breve{t}_i,\breve{x}^i),(\breve{t},\breve{x}))
	\geq \Gamma(\breve{t}_0,\breve{x}^0)$, and
	\item[(iii)]  $\mathbf{\Gamma}(s,x)<\mathbf{\Gamma}(\breve{t},\breve{x})$  for all $(s,x)\in [\breve{t},T]\times H\setminus \{(\breve{t},\breve{x})\}$.
\end{description}
{We should note that the point
	$(\breve{t},\breve{x})$ depends on $\delta$.}   By the definitions of $\widetilde{w}^{\hat{t},N}_1$ and $\widetilde{w}_{1}^{\hat{t},N,*}$, we have
\begin{eqnarray}\label{220901a}
\begin{aligned}
\widetilde{w}_{1}^{\hat{t},N,*}(\bar{{t}},\bar{x}_N)-\varphi(\bar{{t}},\bar{x}_N)
=&\limsup_{\scriptsize \begin{matrix}
	(s,y_N)\rightarrow(\bar{{t}},\bar{x}_N)\\
	s\geq \hat{t}
	\end{matrix}}(\widetilde{w}^{\hat{t},N}_1(s,y_N)-\varphi(s,y_N))\\
=&\limsup_{\scriptsize \begin{matrix}
	(s,y_N)\rightarrow(\bar{{t}},\bar{x}_N)\\
	s\geq \hat{t}
	\end{matrix}}\left(\sup_{\scriptsize  z_N=y_N}
w_1(s,z)-\varphi(s,y_N)\right).
\end{aligned}
\end{eqnarray}
Note that by Lemma \ref{lemma4.3}, $w_1$   satisfies condition (\ref{0608a2024}). Then, for every $(s,z)\in [\hat{t},\bar{t}]\times H$,
\begin{eqnarray}\label{220901b}
w_1(s,z)\leq w_1(\bar{t},z_{s,\bar{t}}^{A,N})+\rho_1(|\bar{t}-s|,|z|).
\end{eqnarray}
By the definition of $w_1$, there exists a constant {$M_6>0$} such that
\begin{eqnarray}\label{220901c}
\sup_{\scriptsize z_N=y_N}
w_1(s,z)=\sup_{\scriptsize  z_N=y_N,
	|z|\leq {M_6}}
w_1(s,z).
\end{eqnarray}
Thus, by (\ref{220901b}) and (\ref{220901c}),
\begin{eqnarray}\label{220901d}
\begin{aligned}
&\limsup_{\scriptsize \begin{matrix}  (s,y_N)\rightarrow(\bar{{t}},\bar{x}_N)\\
	s\geq \bar{t}\end{matrix}}\sup_{\scriptsize z_N=y_N}
[w_1(s,z)-\varphi(s,z_N)]\\
&\leq \limsup_{\scriptsize \begin{matrix}
	(s,y_N)\rightarrow(\bar{{t}},\bar{x}_N)\\
	s\geq \hat{t}
	\end{matrix}}
\left(\sup_{\scriptsize z_N=y_N
}
w_1(s,z)-\varphi(s,y_N)\right)\\
&\leq \limsup_{\scriptsize \begin{matrix}
	(s,y_N)\rightarrow(\bar{{t}},\bar{x}_N)\\
	s\geq \hat{t}
	\end{matrix}}\sup_{ z_N=y_N,
	 |z|\leq {M_6}}\!\!\!\!\!\!
\left[w_1(\bar{t},z_{s,\bar{t}}^{A,N})+\rho_1(|s\vee\bar{t}-s|,{M_6})-\varphi(s,y_N)\right]\\
&\leq\limsup_{\scriptsize \begin{matrix}  (s,y_N)\rightarrow(\bar{{t}},\bar{x}_N)\\
s\geq \bar{t}\end{matrix}}\sup_{z_N=y_N}
[w_1(s,z)-\varphi(s,z_N)].
\end{aligned}
\end{eqnarray}
Therefore, by (\ref{220901a}) and (\ref{220901d}),
\begin{eqnarray*}
\begin{aligned}
	\widetilde{w}_{1}^{\hat{t},N,*}(\bar{{t}},\bar{x}_N)-\varphi(\bar{{t}},\bar{x}_N)
	=&\limsup_{\scriptsize \begin{matrix}
			(s,y_N)\rightarrow(\bar{{t}},\bar{x}_N)\\
			s\geq \bar{t}
	\end{matrix}}\sup_{z_N=y_N}
	[w_1(s,z)-\varphi(s,y_N)]\\
	\leq&\sup_{(s,x)\in [\bar{t},T]\times H}\Gamma(s,x).
\end{aligned}
\end{eqnarray*}
Combining with (\ref{0615a}),
\begin{eqnarray}\label{220901e}
\Gamma(\breve{t}_0,\breve{x}^0)\geq \sup_{(s,x)\in [\bar{t},T]\times H}\Gamma(s,x)-\delta
\geq\widetilde{w}_{1}^{\hat{t},N,*}(\bar{{t}},\bar{x}_N)-\varphi(\bar{{t}},\bar{x}_N)-\delta.
\end{eqnarray}
Recall that $\widetilde{w}_{1}^{\hat{t},N,*}\geq\widetilde{w}^{\hat{t},N}_1$.  Then, by 
the definition of $\widetilde{w}^{\hat{t},N}_1$, the property (ii) of $(\breve{t},\breve{x})$ and (\ref{220901e}),
\begin{eqnarray}\label{20210509}
\begin{aligned}
\widetilde{w}_{1}^{\hat{t},N,*}(\breve{t},\breve{x}_N)
-\varphi(\breve{t},\breve{x}_N)
&\geq\widetilde{w}^{\hat{t},N}_1(\breve{t},\breve{x}_N)
-\varphi(\breve{t},\breve{x}_N)
\geq w_1(\breve{t},\breve{x})
-\varphi(\breve{t},\breve{x}_N)\\
&\geq \Gamma(\breve{t}_0,\breve{x}^0)
\geq\widetilde{w}_{1}^{\hat{t},N,*}(\bar{{t}},\bar{x}_N)-\varphi(\bar{{t}},\bar{x}_N)-\delta=-\delta.
\end{aligned}
\end{eqnarray}
Noting that $\nu$ is independent of  $\delta$ and $\varphi$ is a continuous  function bounded from below, 
by the definitions of  $\Gamma$ and $w_1$, and (\ref{w2024}),
there exists a constant  $M_7>0$   depending only on $\varphi$ 
that is sufficiently  large   that
$
\Gamma(t,x)<\sup_{(s,z)\in [\bar{t},T]\times H}\Gamma(s,x)-1
$ for all $t\in [\bar{t},T]$ and $|x|\geq M_7$. Thus, we have $|\breve{x}|\vee
|{\breve{x}}^{0}|<M_7$. In particular, $|\breve{x}_N|<M_7$.
Letting $\delta\rightarrow0$, by  the similar proof procedure of (\ref{4.226}) and (\ref{05231}), we obtain
\begin{eqnarray}\label{delta0}
&& \breve{t}\rightarrow \bar{t},\ \breve{x}_N\rightarrow \bar{x}_N,\  \widetilde{w}_{1}^{\hat{t},N,*}(\breve{t},\breve{x}_N)\rightarrow\widetilde{w}_{1}^{\hat{t},N,*}(\bar{{t}},\bar{x}_N) \ \mbox{as}\ \delta\rightarrow0.
\end{eqnarray}
Noting that $\widetilde{w}_{1}^{\hat{t},N,*}(t,x_N)-\varphi(t,x_N)$  has a strict   maximum 0 at $(\bar{{t}},\bar{x}_N)\in (0, T)\times H_N$ on $[0, T]\times H_N$, by (\ref{0815zhou2})
and (\ref{delta0}) there exists a constant $0<\Delta<1$ such that
for all $0<\delta<\Delta$,
$$
\varphi(\breve{t},\breve{x}_N)\geq \widetilde{w}_{1}^{\hat{t},N,*}(\breve{t},\breve{x}_N)\geq \widetilde{w}_{1}^{\hat{t},N,*}(\bar{{t}},\bar{x}_N)-1\geq -(L+1).
$$
Then, by the definitions of  $\Gamma$ and $w_1$, (\ref{w2024}) and (\ref{20210509}),
there exists a constant  $M_8>0$   depending only on $L$ 
such that, for all $0<\delta<\Delta$ and $|x|\geq M_8$ satisfying $x_N=\breve{x}_N$,
$$
\Gamma(\breve{t},x)<-2\leq \Gamma(\breve{t}_0,\breve{x}^0)-1 \leq\sup_{(s,z)\in [\bar{t},T]\times H}\Gamma(s,z)-1.
$$
Thus, we have $|\breve{x}|<M_8$ for all $0<\delta<\Delta$.   From (\ref{0815zhou1}), it follows that $\bar{t}<{\hat{t}}+\frac{|T-{\hat{t}}|}{2}\leq T$. 
Then, by (\ref{delta0}),  we have  $\breve{t}<T$ 
provided that  $\delta>0$ is small enough.
Thus,  we have from the definition of the viscosity sub-solution
\begin{eqnarray}\label{5.15}
\begin{aligned}
&
{ \partial_t^o}\chi(\breve{t},\breve{x})
+\left\langle A^*\nabla_{x}\varphi(\breve{t}, \breve{x}_N),\, \breve{x}\right\rangle_H-4\beta\left\langle A^*\left(\breve{x}_N-\left(\hat{z}^A_{\hat{t},\breve{t}}\right)_N\right),\, \breve{x}\right\rangle_H\\
&
+{\mathbf{H}}\left(\breve{t},\breve{x}, \, W_1(\breve{t},\breve{x}),\, \nabla_x\chi(\breve{t},\breve{x}),\, \nabla^2_{x}\chi(\breve{t},\breve{x})\right)\geq c,
\end{aligned}
\end{eqnarray}
where, for every $(t,x)\in  {[\breve{t},T]\times{{\Lambda}}}$,  define $\chi(t,x) :=\chi_{1}(t,x)+\chi_{2}(t,x)$ with
\begin{eqnarray*}
	\chi_{1}(t,x)  &:=&\varphi(t,x_N)-2\beta\left|x_N-\left(\hat{z}^A_{\hat{t},t}\right)_N\right|^2\quad \hbox{ \rm and }\\
	\chi_{2}(t,x)  &:=&
	\varepsilon\frac{\nu T-t}{\nu
		T}\, |x|^4
	+\varepsilon \, {\Upsilon}((t,x),(\hat{t},\hat{x}))
	+\sum_{i=0}^{\infty}
	\frac{1}{2^i}\, {\Upsilon}((t_i,x^i),(t,x)\\
	&&
	+\sum_{i=0}^{\infty}
	\frac{1}{2^i}\, {\Upsilon}((\breve{t}_i,\breve{x}^{i}),(t,x))
	+2\beta\left|x-\hat{z}^A_{\hat{t},t}\right|^2;
\end{eqnarray*}
and therefore, by (\ref{f-t}),
\begin{eqnarray*}
\begin{aligned}
	{ \partial_t^o}\chi(t,x):=&(\chi_{1})_t(t,x)+{\partial_t^o}\chi_{2}(t,x)\\
	=&
	\varphi_t({t},x_N)+4\beta\left\langle A^*\left(x_N-\left(\hat{z}^A_{\hat{t},t}\right)_N\right),\,
	\hat{z}^A_{\hat{t},t}\right\rangle_H-\frac{\varepsilon}{\nu T}|x|^4\\
&
	+2\varepsilon({t}-{\hat{t}})+2\sum_{i=0}^{\infty}\frac{1}{2^i}[(t-t_{i})+(t-\breve{t}_i)],
\end{aligned}
\end{eqnarray*}
\begin{eqnarray*}
\begin{aligned}
	\nabla_x\chi(t,x):=&\nabla_x\chi_{1}(t,x)+\nabla_x\chi_{2}(t,x)\\
	=&\nabla_{x}\varphi({t}, x_N)
	+4\beta\left(x^-_N-\left(\hat{z}^A_{\hat{t},t}\right)^-_N\right)
	+\varepsilon\, \frac{\nu T-{t}}{\nu T}\, \nabla_x|x|^4 +\varepsilon\, \nabla_x|x-\hat{x}_{{\hat{t},t}}^A|^4\\
	&+\nabla_x\left[\sum_{i=0}^{\infty}\frac{1}{2^i}
	\, |x-(x^{i})_{t_{i},t}^A|^4
	+\sum_{i=0}^{\infty}\frac{1}{2^i}\, |x-(\breve{x}^i)_{\breve{t}_i,t}^A|^4\right], \quad \hbox{\rm and }
\end{aligned}
\end{eqnarray*}
\begin{eqnarray*}
\begin{aligned}
	\partial^2_{x}\chi(t,x):=&\nabla^2_{x}\chi_{1}(t,x)+\nabla^2_{x}\chi_{2}(t,x)\\
	=&\nabla^2_{x}\varphi({t}, x_N)+4\beta Q_N+\varepsilon\frac{\nu T-{t}}{\nu T}\, \nabla^2_{x}|x|^4 +\varepsilon\, \nabla^2_{x}|x-\hat{x}_{{\hat{t},t}}^A|^4
	\\
	&
	+\nabla^2_{x}\left[\sum_{i=0}^{\infty}\frac{1}{2^i}
	|x-(\gamma^{i})_{t_{i},t}^A|^4
	+\sum_{i=0}^{\infty}\frac{1}{2^i}|x-(\breve{x}^i)_{\breve{t}_i,t}^A|^4\right].
\end{aligned}
\end{eqnarray*}
We notice that  $|\breve{x}|< M_8$ for all $0<\delta<\Delta$ and $M_8$ only depends on $L$.
Then    by 
  (\ref{5.15}) and the definition of ${\mathbf{H}}$, it follows that there exists a constant $C'_0\geq0$ depending only on $L$ such that
$$\varphi_t\left(\breve{t},\breve{x}_N\right)\geq -C_0'\left(1+\left|\nabla_{x}\varphi(\breve{t},\breve{x}_N)\right|+\left|\nabla^2_{x}\varphi(\breve{t},\breve{x}_N)\right|\right).$$
Letting $\delta\rightarrow0$, by (\ref{0815zhou2}) and (\ref{delta0}) , we obtain 
$$
\varphi_{t}(\bar{t},\bar{x}) \geq -{C}_0'(1+|\nabla_{x}\varphi(\bar{t},\bar{x})|+|\nabla^2_{x}\varphi(\bar{t},\bar{x})|)\geq -{C}_0'(1+2L).
$$
Letting $ \tilde{C}_0={C}_0'(1+2L)$, we get $\varphi_{t}(\bar{t},x_1) \geq -{\tilde{C}_0}$.
The proof is now complete.
\end{proof}

\begin{lemma}\label{0611a}\ \
	The functionals $w_1$ and $w_2$ defined by (\ref{060912024}) and (\ref{060922024}) satisfy the conditions of  Theorem \ref{theorem0513} with  $\varphi$, where $\varphi$ is the function defined by (\ref{070632024}).
\end{lemma}

\begin{proof}
From (\ref{w2024}),  the functionals $w_1$ and $w_2$ are continuous and bounded from above and satisfy (\ref{051312024}).
By  Lemmas \ref{lemma4.3} and \ref{lemma4.344}, $w_1$ and $w_2$ satisfy condition  (\ref{0608a2024}),
and  $\widetilde{w}_{1}^{\hat{t},N,*}\in \Phi_N(\hat{t},\hat{x}_N)$ and $\widetilde{w}_{2}^{\hat{t},N,*}\in \Phi_N(\hat{t},\hat{y}_N)$.
Moreover, let $\varphi$ be the function defined by (\ref{070632024}). By (\ref{iii42024}) we obtain that, for all 
$(t,x,y)\in [\hat{t},T]\times H\times H$,
\begin{eqnarray}\label{wv}
\ \ \ \ \ \ \  \ \
\begin{aligned}
&w_{1}(t,x)+w_{2}(t,y)-\varphi(x_N,y_N)+\varepsilon ({\Upsilon}((t,x),(\hat{t},{\hat{x}}))+ {\Upsilon}((t,y),(\hat{t},\hat{y})))  \\
=&w_{1}(t,x)+w_{2}(t,y)-\beta|x_N-y_N|^2+\varepsilon ({\Upsilon}((t,x),(\hat{t},{\hat{x}}))+ {\Upsilon}((t,y),(\hat{t},{\hat{y}})))  \\
\leq&
{\Psi}_1(t,x,y)\leq\Psi_1(\hat{t},{\hat{x}},{\hat{y}})=w_{1}(\hat{t},\hat{x})+w_{2}(\hat{t},\hat{y})
-\beta|\hat{x}_N-\hat{y}_N|^2\\
=&w_{1}(\hat{t},\hat{x})+w_{2}(\hat{t},\hat{y})
-\varphi(\hat{x}_N,\hat{y}_N),
\end{aligned}
\end{eqnarray}
where the last inequality becomes equality if and only if $t={\hat{t}}$, $t,x=\hat{x}, t,y=\hat{y}$.
Then we obtain that $
w_1(t,x)+w_2(t,y)-\varphi((t,x_N, y_N)
$
has a 
maximum over $[\hat{t},T]\times H\times H$ at a point $(\hat{t},{\hat{x}},{\hat{y}})$ with $\hat{t}\in (0,T)$ and satisfies (\ref{101012024}) with  ${\tilde{\rho}}(l)=\varepsilon l, l\in [0,\infty)$. Thus $w_1$ and $w_2$ satisfy the conditions of  Theorem \ref{theorem0513}  with  $\varphi$ defined by (\ref{070632024}).
\end{proof}

\begin{example}\label{202401281}\ \
	Parabolic   equation
\end{example}
Let $\mathcal{O} \subset \mathbb{R}^d$ be a bounded  domain with a smooth boundary $\partial\mathcal{O}$. Let
$$
                  A:=\sum^{n}_{i,j}\partial_i(a_{ij}\partial_j),\ \ D(A):=H^1_0(\mathcal{O})\cap H^2(\mathcal{O}),
$$
where $a_{ij}=a_{ji}\in L^{\infty}(\mathcal{O})$ for $i,j\in \{1,\ldots,d\}$, and there exists a $\theta>0$ such that
$$
                    \sum^{n}_{i,j}a_{ij}\xi_i\xi_j\geq \theta |\xi|^2, \forall x\in \mathbb{R}^d
$$
a.e. in $\mathcal{O}$. Then $A$ generates a contraction semigroup $e^{tA}$ in $H:=L^2(\mathcal{O})$ (see ~\cite[Example 3.16, page 178]{fab1}). We consider a control problem for the stochastic parabolic equation
\begin{eqnarray}\label{202401291}
\begin{cases}
\frac{\partial y}{\partial s}(s,\xi)=Ay(s,\xi)+f(s,\xi,y(s,\xi),u_s)\\
\  \  \ \ \ \ \ \ \  \ \ \   +h(s,\xi,y(s,\xi),u_s)\frac{\partial}{\partial s}W_{Q}(s,\xi),  \quad s\in (t,T], \ \xi\in\mathcal{O},\\
 y(s,\xi)=0,\ \ s\in (t,T]\times \partial \mathcal{O},\\
y(t,\xi)=y^0(\xi), \  \xi\in\mathcal{O},
\end{cases}
\end{eqnarray}
where $Q$ is an operator in $\mathcal{L}^+_1(L^2(\mathcal{O}))$ (see ~\cite[Proposition B.30, page 801]{fab1}) and $W_Q$ is a $Q$-Wiener process
(see ~ \cite[Definition 1.88, page 28]{fab1}), $y^0\in L^2(\mathcal{O})$, $U$ is a Polish space,
$\mathbf{u}=\{u_s, s\in [t,T]\}\in \mathcal{U}[t,T]$.
 Suppose we want to maximize the cost functional
\begin{eqnarray}\label{202401291d}
               ~~~~~  I(t,y^0, \mathbf{u})=\mathbb{E}\left[\int^{T}_{t}\int_{\mathcal{O}}\alpha(s,y(s,\xi),u_s)d\xi ds+\int_{\mathcal{O}}\beta(y(T,\xi))d\xi\right]
\end{eqnarray}
over all $\mathbf{u}\in \mathcal{U}[t,T]$. 
  We assume the following.
\begin{assumption}\label{hypstate20240128}
	\begin{description}
		\item[(i)]
		 $(f, h) :[0,T]\times\mathcal{O}\times \mathbb{R}\times U\to  \mathbb{R}\times  \mathbb{R}$ is continuous, {and
        $(f,h)(\cdot,\cdot,\cdot,u)$ is continuous in $ [0,T]\times \mathcal{O}\times \mathbb{R}$, uniformly in $u\in U$.} Moreover,
		there is a constant $L>0$ such that,  for all $(t,\xi, x, u) \in [0,T]\times \mathcal{O}\times \mathbb{R}\times U$,
		\begin{eqnarray*}
		&&\displaystyle \left|f(t,\xi,x,u)\right|^2\vee\left|h(t,\xi,x,u)\right|^2\leq
		L^2(1+|x|^2),\\
		&&\displaystyle \left|f(t,\xi,x,u)-f(t,\xi,y,u)\right|\vee\left|h(\xi,x,u)-h(\xi,y,u)\right|\leq
		L|x-y|.
		\end{eqnarray*}
\item[(ii)]
	$
	\alpha: [0,T]\times \mathbb{R}\times U\rightarrow \mathbb{R}$ and $\beta:  \mathbb{R}\rightarrow \mathbb{R}$ are continuous. 
         Moreover,     there is a  constant $L>0$
	such that, for all $(t,x,    y,  u)
	\in [0,T]\times\mathbb{R}\times\mathbb{R}\times U$,
	\begin{eqnarray*}
		&&| \alpha(t,x,u)|\leq L(1+|x|),
		\ \ \ \ |\alpha(t,x,u)-\alpha(t,y,u)|\leq L|x-y|,
		\\
		&&|\beta(x)-\beta(y)|\leq L|x-y|.
	\end{eqnarray*}
\item[(iii)] For every $(t,\gamma)\in [0,T)\times H$,
	\begin{eqnarray*}
	\lim_{N\to \infty} \sup_{u\in U}\left|Q_N{h}(t,\cdot,\gamma(\cdot),u)\right|_{H}^2=0.
	\end{eqnarray*}
	\end{description}
\end{assumption}
Define, for $s\in [0,T], \gamma\in H$,
 $z\in H$ and $u\in U$,
\begin{eqnarray*}
                         && b(s,\gamma,u)=f(s,\cdot,\gamma(\cdot),u),\ \ \sigma(s,\gamma,u)z=h(s,\cdot,\gamma(\cdot),u)z(\cdot),\\
                          && q(s,\gamma,u)=\int_{\mathcal{O}}\alpha(s,\gamma(\xi),u)d\xi,\ \  \phi(\gamma)=\int_{\mathcal{O}}\beta(\gamma(\xi))d\xi.
\end{eqnarray*}
Equation (\ref{202401291}) can be rewritten as the following evolution equation in $H$
$$
    dX_s=AX_sds+b(s,X_s,u_s)ds+\sigma(s,X_s,u_s)dW_Q(s),\ \ X_t=y^0,
$$
and the  cost functional $I$ can be rewritten as
$$
                    J(t,y^0,\mathbf{u})=\mathbb{E}\left[\int^T_tq(s,X_s,u_s)ds+\phi(X_T)\right].
$$
It is easy
to check that Assumption \ref{hypstate20240128} implies that Assumption \ref{hypstate} is satisfied. For example,
let $\Xi=Q^{\frac{1}{2}}H$, by  \cite[Theorem 1.87, page 28]{fab1},  $W_Q$ is a   cylindrical Wiener process in Hilbert space $\Xi$, by {Assumption}  \ref{hypstate20240128} (i),
\begin{eqnarray*}
                    &&|\sigma(s,\gamma,u)|^2_{L_2(\Xi,H)}=|h(s,\cdot,\gamma(\cdot)),u)Q^{\frac{1}{2}}|^2_{L_2(H)}\\
                    &\leq& |h(s,\cdot,\gamma(\cdot),u)|^2_{L(H)}|Q^{\frac{1}{2}}|^2_{L_2(H)}
                    =|Q^{\frac{1}{2}}|^2_{L_2(H)}\sup_{|z|_H\leq 1}\left(\int_{\mathcal{O}}h(s,\xi,\gamma(\xi),u)z(\xi)d\xi\right)^2\\
                    &\leq&|Q^{\frac{1}{2}}|^2_{L_2(H)}\int_{\mathcal{O}}h^2(s,\xi,\gamma(\xi),u)d\xi
                    \leq L^2|Q^{\frac{1}{2}}|^2_{L_2(H)}\int_{\mathcal{O}}(1+|\gamma(\xi)|^2)d\xi\\
                    &\leq&L^2|Q^{\frac{1}{2}}|^2_{L_2(H)}\left(m(\mathcal{O})+|\gamma|^2_H\right);
\end{eqnarray*}
and by {Assumption}  \ref{hypstate20240128} (iii),
	\begin{eqnarray*}
	&&\lim_{N\to \infty} \sup_{u\in U}\left|Q_N\sigma(s,\gamma,u)\right|_{L_2(\Xi,H)}^2
=\lim_{N\to \infty} \sup_{u\in U}|Q_Nh(s,\cdot,\gamma(\cdot)),u)Q^{\frac{1}{2}}|^2_{L_2(H)}\\
&\leq&\lim_{N\to \infty} \sup_{u\in U}|Q_Nh(s,\cdot,\gamma(\cdot)),u)|^2_{H}|Q^{\frac{1}{2}}|^2_{L_2(H)}=0.
\end{eqnarray*}
\begin{example}\label{20240129a}\ \
	Hyperbolic  equation
\end{example}
Consider  a control problem for
the stochastic Hyperbolic  equation in  a bounded  domain $\mathcal{O}$ with a smooth boundary $\partial\mathcal{O}$
\begin{eqnarray}\label{20240129a1}
\begin{cases}
\frac{\partial^2 y}{\partial s^2}(s,\xi)=Ay(s,\xi)+f(\xi,y(s,\xi),u_s)\\
\  \  \ \ \ \ \ \ \  \ \ \   +h(\xi,y(s,\xi),u_s)\frac{\partial}{\partial s}{W}_{{Q}}(s,\xi),  \quad s\in (t,T], \ \xi\in\mathcal{O},\\
 y(s,\xi)=0,\ \ s\in (t,T]\times \partial \mathcal{O},\\
y(t,\xi)=y^0(\xi), \qquad
\frac{\partial y}{\partial s} (t,\xi)=z^0(\xi),\ \xi\in\mathcal{O},
\end{cases}
\end{eqnarray}
where $A$ is given in Example \ref{202401281}, $y^0\in H^1_0(\mathcal{O})$ and  $z^0\in L^2(\mathcal{O})$.  Suppose we also want to maximize the cost functional (\ref{202401291d}). Set
$$
\mathbb{H}:=\left(\begin{matrix}
	H^1_0(\mathcal{O})\\
\times\\
	H
	\end{matrix}\right)
$$ equipped with the inner product
$$
\left\langle\left(\begin{matrix}
	y\\
	z
	\end{matrix}\right), \left(\begin{matrix}
	y'\\
	z'
	\end{matrix}\right)\right\rangle_\mathbb{H}=\langle(-A)^{\frac{1}{2}}y,(-A)^{\frac{1}{2}}y'\rangle_H+\langle z,z'\rangle_H,\ \ (x,y), (x',y')\in \mathbb{H}.
$$
The operator
 $$
         \mathcal{A}=\left(\begin{matrix}
	0&I\\
	A&0
	\end{matrix}\right),\ \        D(\mathcal{A})=\left(\begin{matrix}
	D(A)\\
	\times \\
H^1_0(\mathcal{O})
	\end{matrix}\right)
 $$
generates a contraction semigroup $e^{t\mathcal{A}}$ in $\mathbb{H}$ (see ~\cite[Example 3.13, page 176]{fab1}).
Define, for $\left(\begin{matrix}
	\gamma\\
	\eta
	\end{matrix}\right)\in \mathbb{H}$,
 $\left(\begin{matrix}
	x\\
	y
	\end{matrix}\right)\in \mathbb{H}$ and $u\in U$,
\begin{eqnarray*}
                         && b\left(\left(\begin{matrix}
	\gamma\\
	\eta
	\end{matrix}\right),u\right)=\left(\begin{matrix}
	0\\
	f(\cdot,\gamma(\cdot),u)
	\end{matrix}\right),\qquad \sigma\left(\left(\begin{matrix}
	\gamma\\
	\eta
	\end{matrix}\right),u\right)\left(\begin{matrix}
	y\\
	z
	\end{matrix}\right)=\left(\begin{matrix}
	0\\
	h(\cdot,\gamma(\cdot),u)z
	\end{matrix}\right),\\
&&\widetilde{W}_{\tilde{Q}}=\left(\begin{matrix}
	0\\
	W_Q
	\end{matrix}\right),
\ \ \tilde{Q}\left(\begin{matrix}
	y\\
	z
	\end{matrix}\right)=\left(\begin{matrix}
	0\\
	Qz
	\end{matrix}\right),\\
                          && q\left(s,\left(\begin{matrix}
	\gamma\\
	\eta
	\end{matrix}\right),u\right)=\int_{\mathcal{O}}\alpha(s,\gamma(\xi),u)d\xi,\ \  \phi\left(\left(\begin{matrix}
	\gamma\\
	\eta
	\end{matrix}\right)\right)=\int_{\mathcal{O}}\beta(\gamma(\xi))d\xi.
\end{eqnarray*}
Equation (\ref{20240129a1}) can be rewritten in an abstract way in $\mathbb{H}$ as
$$
    dX_s=\mathcal{A}X_sds+b(X_s,u_s)ds+\sigma(X_s,u_s)d\widetilde{W}_Q(s),\ \ X_t=x:=\left(\begin{matrix}
	y^0\\
	z^0
	\end{matrix}\right),
$$
and the  cost functional (\ref{202401291d}) can be rewritten as
$$
                    J(t,x,u(\cdot))=\mathbb{E}\left[\int^T_tq(s,X_s,u_s)ds+\phi(X_T)\right].
$$
It is easy to see that Assumption \ref{hypstate20240128} implies that Assumptions \ref{hypstate} and \ref{hypstate5666} are satisfied.
\begin{remark}\label{202401301}\ \
	In the above two examples, we  assume 
 $A$ generates a contraction semigroup $e^{tA}$ in $L^2(\mathcal{O})$ and $Q$ is an operator in $\mathcal{L}^+_1(L^2(\mathcal{O}))$ without any additional assumptions, while, for example ~\cite[Example 3.69, page 238]{fab1}, the additional assumptions $A=\Delta$, $d>1$ and $(-\Delta)^{\frac{d-1}{4}}Q^{\frac{1}{2}}\in L_2(L^2(\mathcal{O}))$ are  imposed  to ensure  the  $B$-continuity of the coefficients in~\cite[(3.22), page 184]{fab1}.
\end{remark}

\appendix

 \section{Properties of the value functional}\label{appendixc}
\setcounter{equation}{0}
\renewcommand{\theequation}{A.\arabic{equation}}
\begin{proof}[Proof of Theorem \ref{theoremj=y}]
 Let $\{h^n\}$, $n\in N$,
be a dense subset of $H$, $B(h^n,\frac{1}{k})$
be the open sphere in $H$ with radius
$\frac{1}{k}$ and center $h^n$.  Set $B_{n,k}:=B(h^n,\frac{1}{k})\setminus
\bigcup_{m<n}B(h^m,\frac{1}{k})$  and $A_{n,k}:=\{\omega\in \Omega|\xi(\omega)\in
B_{n,k}\}$. Then $\cup_{n=1}^{\infty}A_{n,k}=\Omega$ and the
sequence $f^k(\omega):= \Sigma^{\infty}_{n=1}h^n1_{
	A_{n,k}}(\omega)$ is ${\mathcal
	{F}}_{t}$-measurable and
converges to $\xi$ strongly and
uniformly.
 For every  $n$ and $\mathbf{u}\in {\mathcal {U}}[t,T]$, we put $(X^{n}_s,Y^{n}_s,Z^{n}_s)=(X^{t,h^n,\mathbf{u}}_s,Y^{t,h^n,\mathbf{u}}_s,Z^{t,h^n,\mathbf{u}}_s)$.
Then $X^{n}_s$ is the solution of the SEE
\begin{eqnarray*}
X^n_s=e^{(s-t)A}h^n+\int_{t}^{s}e^{(s-l)A}b(l,X^n_l,u_l)dl
+\int_{t}^{s}e^{(s-l)A}\sigma(l,X^n_l,u_l)dW_l, \ \
s\in [t,T];
\end{eqnarray*}
 and $(Y^{n}_s,Z^{n}_s)$ is the solution of the associated  BSDE
\begin{eqnarray*}
Y^{n}_s=\phi(X_T^{n})+\int^{T}_{s}q(l,X^{n}_l,Y^{n}_l,Z^{n}_l,u_l)dl-\int^{T}_{s}Z^{n}_ldW_l,\  \ s\in [t,T].
\end{eqnarray*}
The above two equations are multiplied by $1_{A_{n,k}}$  and summed up with respect to $n$. Thus, taking into account that $\sum_{n=1}^{\infty}\varphi(h^n)1_{A_{n,k}}=\varphi(\sum_{n=1}^{\infty}h^n1_{A_{n,k}})$, we obtain
\begin{eqnarray*}
	\sum_{n=1}^{\infty}1_{A_{n,k}}X^n_s
	&=&\sum_{n=1}^{\infty}1_{A_{n,k}}e^{(s-t)A}h^n+\int_{t}^{s}e^{(s-l)A}b\left(l,\sum_{n=1}^{\infty}1_{A_{n,k}}X^n_l, u_l\right)dl\nonumber\\
	&& \ \
	+\int_{t}^{s}e^{(s-l)A}\sigma\left(l,\sum_{n=1}^{\infty}1_{A_{n,k}}X^n_l,u_l\right)dW_l, \nonumber
\end{eqnarray*}
and
\begin{eqnarray*}
\quad\sum_{n=1}^{\infty}1_{A_{n,k}}Y^{n}_s
&=&\phi\left(\sum_{n=1}^{\infty}1_{A_{n,k}}X_T^{n}\right)-\int^{T}_{s}\sum_{n=1}^{\infty}1_{A_{n,k}}Z^{n}_ldW_l\\
&&+\int^{T}_{s}q\left(l,\sum_{n=1}^{\infty}1_{A_{n,k}}X_l^{n},
\sum_{n=1}^{\infty}1_{A_{n,k}}Y^{n}_l,\sum_{n=1}^{\infty}1_{A_{n,k}}Z^{n}_l,u_l\right)dl.
\end{eqnarray*}
Then,  the strong uniqueness property of the solution to the SEE and the BSDE yields
\begin{eqnarray*}
 \left(X^{t,f^k,\mathbf{u}}_s,Y^{t,f^k,\mathbf{u}}_s,Z^{t,f^k,\mathbf{u}}_s\right)
=\left(\sum_{n=1}^{\infty}1_{A_{n,k}}X^n_s,\sum_{n=1}^{\infty}1_{A_{n,k}}Y^{n}_s,\sum_{n=1}^{\infty}1_{A_{n,k}}Z^{n}_s\right), \ \ s\in [t,T].
\end{eqnarray*}
Finally, from $J(t,h^n,\mathbf{u})=Y^n_t$ for $n\geq 1$, we deduce that
\begin{eqnarray*}
Y^{t,f^k,\mathbf{u}}_t
=\sum_{n=1}^{\infty}1_{A_{n,k}}Y^{n}_t=\sum_{n=1}^{\infty}1_{A_{n,k}}J(t,h^n,\mathbf{u})
=J\left(t,\sum_{n=1}^{\infty}1_{A_{n,k}}h^n,\mathbf{u}\right)=J(t,f^k,\mathbf{u}).
\end{eqnarray*}
Consequently, from the estimates (\ref{fbjia4}) and (\ref{fbjia41022}),  we get
\begin{eqnarray*}
	\mathbb{E}|Y^{t,\xi,\mathbf{u}}_t-J(t,\xi,\mathbf{u})|^p
	&\leq&  C\mathbb{E}|Y^{t,\xi,\mathbf{u}}_t-Y^{t,f^{k},\mathbf{u}}_t|^p+C \mathbb{E}|J(t,f^k,\mathbf{u})-J(t,\xi,\mathbf{u})|^p\\
	&\leq& C\mathbb{E}|\xi-f^k|_H^p\rightarrow0\ \mbox{as}\ k\rightarrow \infty.
\end{eqnarray*}
Now let us prove $(\ref{0903jia1022})$. By (\ref{fbjia41022}) and (\ref{j=y}),
\begin{eqnarray*}\label{0903jia}
\begin{aligned}
	|Y^{t,\xi,\mathbf{u}}_t-Y^{t,\eta,\mathbf{u}}_t|
	&=|J(t,\xi,\mathbf{u})-J(t,\eta,\mathbf{u})|=\Big{|}Y^{t,x,\mathbf{u}}_t|_{x=\xi}-Y^{t,y,\mathbf{u}}_t|_{y=\eta}\Big{|}\\
	&=\Big{|}(Y^{t,x,\mathbf{u}}_t-Y^{t,y,\mathbf{u}}_t)|_{x=\xi,y=\xi'}\Big{|}\leq C|\xi-\eta|_H.
\end{aligned}
\end{eqnarray*}
The proof is complete.
\end{proof}
\begin{proof}[Proof of Theorem \ref{valuedet}]
First, we show that there exists a sequence $\{\mathbf{u}_i\}_{i\geq1}\subset {\mathcal{U}}[t,T]$ such that $\{Y^{t,x,\mathbf{u}_i}_t\}_{n\geq1}$ is a nondecreasing sequence and
\begin{eqnarray}\label{jiale}
V(t,x)=\mathop{\esssup}\limits_{\mathbf{u}\in {\mathcal{U}}[t,T]}Y^{t,x,\mathbf{u}}_t=\lim_{i\rightarrow\infty} Y^{t,x,\mathbf{u}_i}_t, \ \  \mathbb{P}\mbox{-a.s.}.
\end{eqnarray}
From \cite[Theorem A.3]{kar}, it is sufficient to prove that, for any $\mathbf{u}_1,\mathbf{u}_2\in {\mathcal{U}}[t,T]$, we have
\begin{eqnarray}\label{jiale1}
Y^{t,x,\mathbf{u}_1}_t\vee Y^{t,x,\mathbf{u}_2}_t=Y^{t,x,\mathbf{u}}_t,\ \ \mathbb{P}\mbox{-a.e.,         }
\end{eqnarray}
with $\mathbf{u}\in {\mathcal{U}}[t,T]$ satisfying $\mathbf{u}:=\mathbf{u}_1{\mathbf{1}}_{A_1}+\mathbf{u}_2{\mathbf{1}}_{A_2}$, where $A_1:=\{Y^{t,x,\mathbf{u}_1}_t\geq Y^{t,x,\mathbf{u}_2}_t\}$ and
$A_2:=\{Y^{t,x,\mathbf{u}_1}_t<Y^{t,x,\mathbf{u}_2}_t\}$.
Taking into account that $\sum_{i=1}^{2}\varphi(x_i)1_{A_i}=\varphi(\sum_{i=1}^{2}x_i1_{A_{i}})$, we obtain
\begin{eqnarray*}
\begin{aligned}
\sum_{i=1}^{2}1_{A_{i}}X^{t,x,\mathbf{u}_i}_s=&e^{(s-t)A}x+\int_{t}^{s}e^{(s-l)A}b\left(l,\sum_{i=1}^{2}1_{A_{i}}X^{t,x,\mathbf{u}_i}_l,\sum_{i=1}^{2}1_{A_{i}}(u_i)_l\right)dl\\
	&+\int_{t}^{s}e^{(s-l)A}\sigma\left(l,\sum_{i=1}^{2}1_{A_{i}}X^{t,x,\mathbf{u}_i}_l,\sum_{i=1}^{2}1_{A_{i}}(u_i)_l\right)dW_l, \quad \mbox{and}
\end{aligned}
\end{eqnarray*}
\begin{eqnarray*}
\begin{aligned}
	\sum_{i=1}^{2}1_{A_{i}}Y^{t,x,\mathbf{u}_i}_s
	=&\phi\left(\sum_{i=1}^{2}1_{A_{i}}X_T^{t,x,\mathbf{u}_i}\right)-\int^{T}_{s}\sum_{i=1}^{2}1_{A_i}Z^{t,x,\mathbf{u}_i}_ldW_l\\
	&+\int^{T}_{s}q\left(l,\sum_{i=1}^{2}1_{A_i}X_l^{t,x,\mathbf{u}_i}, \sum_{i=1}^{2}1_{A_{i}}Y^{t,x,\mathbf{u}_i}_l,\sum_{i=1}^{2}1_{A_i}Z^{t,x,\mathbf{u}_i}_l,\sum_{i=1}^{2}1_{A_{i}}(u_i)_l\right)dl.
\end{aligned}
\end{eqnarray*}
Then the strong uniqueness property of the solution to FSEE and BSDE yields
\begin{eqnarray}\label{10081118}
Y^{t,x,\mathbf{u}}_t=\sum_{i=1}^{2}1_{A_i}Y^{t,x,\mathbf{u}_i}_t=Y^{t,x,\mathbf{u}_1}_t\vee Y^{t,x,\mathbf{u}_2}_t,\ \ \mathbb{P}\mbox{-a.e.}.
\end{eqnarray}

Suppose that $\{\mathbf{u}_i\}_{i\geq1}\subset {\mathcal{U}}[t,T]$ satisfy (\ref{jiale}). By (\ref{fbjia4}),
 we  suppose without lost of generality that $\mathbf{u}_i$ takes the following form:
$$
\mathbf{u}_i=\sum^{n}_{j=1}\mathbf{1}_{A_{ij}}\mathbf{u}_{ij}.
$$
Here $\mathbf{u}_{ij}$ is ${\mathcal{F}}^t$-progressively measurable and $\{A_{ij}\}_{j=1}^{n}$ is a partition of $(\Omega,{\mathcal{F}}_t)$. {Like (\ref{10081118})}, we show that
$$
J(t,x, \mathbf{u}_i)=\sum^{n}_{j=1}\mathbf{1}_{A_{ij}}J(t,x, \mathbf{u}_{ij}).
$$
It is clear that $X^{t,x,\mathbf{u}_{ij}}_s$ is ${\mathcal{F}}_t^s$-measurable for all $s\in[t,T]$, then $Y^{t,x,\mathbf{u}_{ij}}_s$ is ${\mathcal{F}}_t^s$-measurable for all $s\in[t,T]$.
In particular, $J(t,x, \mathbf{u}_{ij})=Y^{t,x,\mathbf{u}_{ij}}_t$ is deterministic. Without lost of generality, we may assume
$$
J(t,x, \mathbf{u}_{ij})\leq J(t,x, \mathbf{u}_{i1}), \ \ j\geq 1.
$$
Then we have $J(t,x,\mathbf{u_i})\leq J(t,x, \mathbf{u}_{i1})$. Combining (\ref{jiale}), we get
$$
\lim_{i\rightarrow\infty}J(t,x, \mathbf{u}_{i1})=V(t,x).
$$
Therefore, $V(t,x)$ is deterministic.
The proof is complete.
\end{proof}

 \section{Consistency and stability  for  viscosity
solutions}\label{appendixd}

\setcounter{equation}{0}
\renewcommand{\theequation}{B.\arabic{equation}}
\begin{proof}[Proof of Theorem \ref{theorem1223}]  Assume that $v$ is a viscosity solution. It is clear that $v(T,x)=\phi(x)$ for all $x\in H$. For any $(t,x)\in [0,T)\times H$, since $v\in C_p^{1,2}([0,T]\times H)$ and $A^*\nabla_xv\in C_p^0([0,T]\times H,H)$, by definition of viscosity solutions,  we have $$
 v_{t}(t,x)+\langle A^*\nabla_xv(t,x), x\rangle_H + {\mathbf{H}}(t,x,v(t,x),\nabla_x v(t,x),\nabla^2_{x} v(t,x))=0.
$$
On the other hand, assume $v$ is a  classical solution. Let  $(\varphi, g)\in \Phi_t\times {\mathcal{G}}_t$ with $t\in [0,T)$
such that for some  $x\in H$,
$$
0=(v-\varphi-g)(t,x)=\sup_{(s,y)\in [t,T]\times H}
(v- \varphi-g)(s,y).
$$ For every $\alpha\in H$ and $\beta\in L_2(\Xi,H)$, let $\xi=x$, $\vartheta(\cdot)\equiv \alpha$ and $\varpi(\cdot)\equiv \beta$ in (\ref{formular1}), 
applying  It\^o formula  to $\varphi$ and inequality (\ref{statesop020240206a})
 to $g$ and noticing that $(v-\varphi-g)(t,x)=0$, we have, for every $0<\delta\leq T-t$,
\begin{eqnarray}\label{1224}
\begin{aligned}
0&\leq\mathbb{E}(\varphi+g-v)(t+\delta,X_{t+\delta})\\
&\leq\mathbb{E}\int^{t+\delta}_{t}[(\varphi-v)_{t}(X_l)+{ \partial_t^o}g(l,X_l)
+\langle A^*\nabla_x(\varphi-v)(l,X_l), X_l\rangle_H\\
&\quad+\langle\nabla_x(\varphi+g-v)(l,X_l), \alpha\rangle_{H}]\,dl+\frac{1}{2}\mathbb{E}\int^{t+\delta}_{t}\mbox{Tr}((\nabla^2_{x}(\varphi+g-v)(l,X_l))\beta\beta^*)\, dl\\
&=
\mathbb{E}\int^{t+\delta}_{t}\widetilde{{\mathcal{H}}}(l,X_l)\, dl,
\end{aligned}
\end{eqnarray}
where, for  $(s,y)\in [0,T]\times H$,
\begin{eqnarray*}
	\widetilde{{\mathcal{H}}}(s,y)&:=&(\varphi-v)_{t}(s,y)+{ \partial_t^o}g(s,y)
	+\langle A^*\nabla_x(\varphi-v)(s,y), y\rangle_H\\
	&&
	+\langle\nabla_x(\varphi+g-v)(s,y), \alpha\rangle_H+\frac{1}{2}\mbox{Tr}((\nabla^2_{x}(\varphi+g-v)(s,y))\beta\beta^*).
\end{eqnarray*}
Letting $\delta\rightarrow0$ in (\ref{1224}),
\begin{eqnarray}\label{040912}
\widetilde{{\mathcal{H}}}(t,x)\geq0.
\end{eqnarray}
Let $\beta=\mathbf{0}$, by the arbitrariness of  $\alpha$,
\begin{eqnarray*}
\begin{aligned}
&\varphi_{t}(t,x)+{ \partial_t^o}g(t,x)+\langle A^*\nabla_x(\varphi-v)(t,x), x\rangle_H\geq v_{t}(t,x), \\
&\nabla_x\varphi(t,x)+\nabla_xg(t,x)=\nabla_xv(t,x).
\end{aligned}
\end{eqnarray*}
Then, for every $u\in U$, let $\beta=\sigma(t,x,u)$ in (\ref{040912}).
Noting that $\varphi(t,x)+g(t,x)=v(t,x)$,
we have
\begin{eqnarray*}
	&&\varphi_{t}(t,x)+{ \partial_t^o}g(t,x)+\langle \nabla_x(\varphi+g)(t,x), b(t,x,u)\rangle_{H}+\langle A^*\nabla_x\varphi(t,x), x\rangle_H\\
	&&
	+\frac{1}{2}\mbox{Tr}((\nabla^2_{x}(\varphi+g)(t,x))\sigma(t,x,u)\sigma^*(t,x,u))\\
	&&+q(t,x,({\varphi}+g)(t,x),({\nabla_x({\varphi}+g)(t,x)})\sigma(t,x,u),u)\\
	&\geq& v_{t}(t,x)+\langle \nabla_xv(t,x), b(t,x,u)\rangle_{H}+\langle A^*\nabla_xv(t,x),  x\rangle_H
	\\&&+\frac{1}{2}\mbox{Tr}(\nabla^2_{x}v(t,x)\sigma(t,x,u)\sigma^*(t,x,u))+q(t,x,v(t,x),{\nabla_xv(t,x)}\sigma(t,x,u),u).
\end{eqnarray*}
Since
$$v_{t}(t,x)+\langle A^*\nabla_xv(t,x), x\rangle_H
+{\mathbf{H}}(t,x,v(t,x),\nabla_xv(t,x),\nabla^2_{x}v(t,x))=0, $$   taking the supremum over $u\in U$, we see that
\begin{eqnarray*}
	&&\varphi_{t}(t,x)+{\partial_t^o}g(t,x)+\langle A^*\nabla_x\varphi(t,x), x\rangle_H\\
	&&+{\mathbf{H}}(t,x,(\varphi+g)(t,x),\nabla_x(\varphi+g)(t,x),\nabla^2_{x}(\varphi+g)(t,x))\\
	&\geq& v_{t}(t,x)+\langle A^*\nabla_xv(t,x), x\rangle_H+{\mathbf{H}}(t,x,v(t,x),\nabla_xv(t,x),\nabla^2_{x}v(t,x))=0.
\end{eqnarray*}
Thus,  $v$ is a viscosity  {sub-solution} of equation (\ref{hjb1}). In a symmetric way, we show that $v$ is also a viscosity  {super-solution} to equation (\ref{hjb1}).
\end{proof}
\begin{proof}[Proof of Theorem \ref{theoremstability}]  Without loss of generality, we shall only prove the viscosity {sub-solution} property.
First,  from $v^{\varepsilon}$ is a viscosity {sub-solution} of  equation (\ref{hjb1}) with generators $b^{\varepsilon}, \sigma^{\varepsilon},  q^{\varepsilon}, \phi^{\varepsilon}$, it follows that
$$
v^{\varepsilon}(T,x)\leq \phi^{\varepsilon}(x),\ \ x\in  H.
$$
Letting $\varepsilon\rightarrow0$, we  have
$$
v(T,x)\leq \phi(x),\ \ x\in H.
$$
Next,  let  $\varphi\in \Phi_{\hat{t}}$ and $g\in {\mathcal{G}}_{\hat{t}}$ with $\hat{t}\in [0,T)$
such that
$$
0=(v-\varphi-g)(\hat{t},\hat{x})=\sup_{(s,y)\in[\hat{t},T]\times H}
(v- \varphi-g)(s,y),
$$
where $\hat{x}\in H$.
By (\ref{sss}), there exists a constant $\delta>0$ such that for all $\varepsilon\in (0,\delta)$,
\begin{eqnarray}\label{09032025g}
\sup_{(t,x)\in [\hat{t},T]\times H}(v^\varepsilon(t,x)-\varphi(t,x)-g(t,x))\leq 1.
\end{eqnarray}
Recall (\ref{upsilon3}), denote $g_{1}(t,x):=g(t,x)+\Upsilon((t,x),(\hat{t},\hat{x}))$
for all
$(t,x)\in [\hat{t},T]\times H$.
Then we have  $g_{1}\in {\mathcal{G}}_{\hat{t}}$. For every $\varepsilon\in (0,\delta)$, by (\ref{09032025g}), it is clear that $v^{\varepsilon}-{{\varphi}}-g_1\in USC^{0+}([\hat{t},T]\times H)$
 bounded from above and satisfies (\ref{09032025c}).
Take
$\delta_i:=\frac{1}{2^i}$ for all $i\geq0$. From Lemma \ref{theoremleft} it follows that,
for every  $(t_0,x^0)\in [\hat{t},T]\times H$ {satisfying}
$$
(v^{\varepsilon}-\varphi-g_1)(t_0,x^0)\geq \sup_{(s,y)\in [\hat{t},T]\times H}(v^{\varepsilon}-\varphi-g_1)(s,y)-\varepsilon, \quad \mbox{and}
$$
$$ (v^{\varepsilon}-\varphi-g_1)(t_0,x^0)\geq (v^{\varepsilon}-\varphi-g_1)(\hat{t},\hat{x}),
$$
there exist $(t_{\varepsilon},x^{\varepsilon})\in [\hat{t},T]\times H$, a sequence $\{(t_i, x^i)\}_{i\geq1}\subset [\hat{t},T]\times H$ and constant $C>0$ such that
\begin{description}
	\item[(i)] 
$\Upsilon((t_i, x^i),(t_{\varepsilon},x^{\varepsilon}))\leq \frac{\varepsilon}{2^i}$, $|x^i|\leq C$ for all $i\geq0$, and $t_i\uparrow t_{\varepsilon}$ as $i\rightarrow\infty$.
	\vskip1mm
	\item[(ii)]  $\displaystyle \Psi(t_{\varepsilon},x^{\varepsilon}):=(v^{\varepsilon}-\varphi-g_1)(t_{\varepsilon},x^{\varepsilon})-\sum_{i=0}^{\infty}\frac{1}{2^i}\Upsilon((t_i, x^i),(t_{\varepsilon},x^{\varepsilon}))\geq (v^{\varepsilon}-\varphi-g_1)(t_0,x^0)$, \quad and
		\vskip1mm
	\item[(iii)]  $\displaystyle \Psi(s,y)
	<\Psi(t_{\varepsilon},x^{\varepsilon})$  for all $(s,y)\in [t_{\varepsilon},T]\times H\setminus \{(t_{\varepsilon},x^{\varepsilon})\}$.
	
\end{description}
We claim that
\begin{eqnarray}\label{gamma}
|t^{\varepsilon}-\hat{t}|+|x^{\varepsilon}-\hat{x}^A_{\hat{t},t^{\varepsilon}}|\rightarrow0  \ \ \mbox{as} \ \ \varepsilon\rightarrow0.
\end{eqnarray}
Otherwise,  there is a constant  $\nu_0>0$
such that
$$
\Upsilon((t^{\varepsilon},x^{\varepsilon}),(\hat{t},\hat{x}))
\geq\nu_0.
$$
Thus, {by the property (ii) of $(t_{\varepsilon},x^{\varepsilon})$,}  we obtain that
\begin{eqnarray*}
\begin{aligned}
	0&=(v- {{\varphi}}-g)(\hat{t},\hat{x})= \lim_{\varepsilon\rightarrow0}(v^\varepsilon-\varphi-g_1)(\hat{t},\hat{x})\\
	&\leq \limsup_{\varepsilon\rightarrow0}\bigg{[}(v^{\varepsilon}-\varphi-g_1)(t_{\varepsilon},x^{\varepsilon})-\sum_{i=0}^{\infty}\frac{1}{2^i}\Upsilon((t_i, x^i),(t_{\varepsilon},x^{\varepsilon}))\bigg{]}\\
	 &\leq\limsup_{\varepsilon\rightarrow0}{[}(v-{{\varphi}}-g)(t_{\varepsilon},x^{\varepsilon})+(v^\varepsilon-v)(t_{\varepsilon},x^{\varepsilon})
	{]}-\nu_0\leq (v- {{\varphi}}-g)(\hat{t},\hat{x})-\nu_0=-\nu_0,
\end{aligned}
\end{eqnarray*}
contradicting $\nu_0>0$. Thus,
$$
d((t^{\varepsilon},x^{\varepsilon}),(\hat{t},\hat{x}))\leq |t^{\varepsilon}-\hat{t}|+|x^{\varepsilon}-\hat{x}^A_{\hat{t},t^{\varepsilon}}|
+|\hat{x}^A_{\hat{t},t^{\varepsilon}}-\hat{x}|\rightarrow0  \ \ \mbox{as} \ \ \varepsilon\rightarrow0.
$$ We note that, by  the definition of $\Upsilon$ and the property (i) of $(t_{\varepsilon},x^{\varepsilon})$,
	there is a generic constant $C>0$ such that
\begin{eqnarray*}
\begin{aligned}
	2\sum_{i=0}^{\infty}\frac{1}{2^i}({t_{\varepsilon}}-{t}_{i})
	&\leq 2\sum_{i=0}^{\infty}\frac{1}{2^i}\bigg{(}\frac{\varepsilon}{2^i}\bigg{)}^{\frac{1}{2}}\leq C\varepsilon^{\frac{1}{2}},\\[3mm]
	\left|\nabla_x|x^{\varepsilon}-\hat{x}_{\hat{t},{t}_{\varepsilon}}^A|^4\right|
	&\leq C|\hat{x}^A_{\hat{t},t_{\varepsilon}}-x^{\varepsilon}|^3,\quad
	\left|\nabla^2_{x}|x^{\varepsilon}-\hat{x}_{\hat{t},{t}_{\varepsilon}}^A|^4\right|
	\leq  C|\hat{x}^A_{\hat{t},t_{\varepsilon}}-x^{\varepsilon}|^2,
\end{aligned}
\end{eqnarray*}
\begin{eqnarray*}
	\bigg{|}\nabla_x\sum_{i=0}^{\infty}\frac{1}{2^i}|x^{\varepsilon}-(x^i)_{t_i,t_{\varepsilon}}^A|^4
	\bigg{|}
	\leq C\sum_{i=0}^{\infty}\frac{1}{2^i}|(x^i)_{t_i,t_{\varepsilon}}^A-x^{\varepsilon}|^3
	\leq C\sum_{i=0}^{\infty}\frac{1}{2^i}\bigg{(}\frac{\varepsilon}{2^i}\bigg{)}^{\frac{3}{4}}
	\leq C{\varepsilon}^{\frac{3}{4}}, \ \ \mbox{and}
\end{eqnarray*}
\begin{eqnarray*}
	\bigg{|}\nabla^2_{x}\sum_{i=0}^{\infty}\frac{1}{2^i}|x^{\varepsilon}-(x^i)_{t_i,t_{\varepsilon}}^A|^4
	\bigg{|}
	\leq C\sum_{i=0}^{\infty}\frac{1}{2^i}|(x^i)_{t_i,t_{\varepsilon}}^A-x^{\varepsilon}|^2
\leq C\sum_{i=0}^{\infty}\frac{1}{2^i}\bigg{(}\frac{\varepsilon}{2^i}\bigg{)}^{\frac{1}{2}}
	\leq C{\varepsilon}^{\frac{1}{2}}.
\end{eqnarray*}
Then for any $\varrho>0$, by (\ref{sss}) and (\ref{gamma}), there is sufficiently small $\varepsilon>0$  such that
$$
2|{t}_{\varepsilon}-\hat{t}|+2\sum_{i=0}^{\infty}\frac{1}{2^i}({t_{\varepsilon}}-{t}_{i})+|\varphi_{t}(t_{\varepsilon},x^{\varepsilon})-\varphi_{t}(\hat{t},\hat{x})|
+|{\partial_t^o}g(t_{\varepsilon},x^{\varepsilon})-{ \partial_t^o}g(\hat{t},\hat{x})|\leq \frac{\varrho}{3},$$
$$
\hat{t}\leq {t}_{\varepsilon}< T,  \ \left|\langle A^*\nabla_x\varphi(t_{\varepsilon},x^{\varepsilon}),  x^{\varepsilon}\rangle_H-\langle A^*\nabla_x{\varphi}(\hat{t},\hat{x}),\hat{x}\rangle_H\right| \leq \frac{\varrho}{3}, \ \mbox{and}\  |I|\leq \frac{\varrho}{3},
$$
where
\begin{eqnarray*}
\begin{aligned}
	I:=&{\mathbf{H}}^{\varepsilon}(t_{\varepsilon},x^{\varepsilon}, v^\varepsilon(t_{\varepsilon},x^{\varepsilon}),
	\nabla_x\varphi(t_{\varepsilon},x^{\varepsilon})+ \nabla_xg_2(t_{\varepsilon},x^{\varepsilon}),\nabla^2_{x}\varphi(t_{\varepsilon},x^{\varepsilon})+ \nabla^2_{x}g_2(t_{\varepsilon},x^{\varepsilon}))\\[2mm]
	 &-{\mathbf{H}}(\hat{t},\hat{x},v(\hat{t},\hat{x}),\nabla_x{\varphi}(\hat{t},\hat{x})+\nabla_xg(\hat{t},\hat{x}),\nabla^2_{x}{\varphi}(\hat{t},\hat{x})+\nabla^2_{x}g(\hat{t},\hat{x})),\\
	 g_2(t_{\varepsilon},x^{\varepsilon}):=&g(t_{\varepsilon},x^{\varepsilon})
+{{\Upsilon}}((t_{\varepsilon},x^{\varepsilon}),(\hat{t},\hat{x}))
	+\sum_{i=0}^{\infty}\frac{1}{2^i}{{\Upsilon}}((t_{\varepsilon},x^{\varepsilon}),(t_i, x^i)), \quad \mbox{and}
\end{aligned}
\end{eqnarray*}
\begin{eqnarray*}
\begin{aligned}
	&{\mathbf{H}}^{\varepsilon}(t,x,r,p,l)\\&:=\sup_{u\in{
			{U}}}\left[
	\langle p, b^{\varepsilon}(t,x,u)\rangle_{H}+\frac{1}{2}\mbox{Tr}(l \sigma^{\varepsilon}(t,x,u){\sigma^{\varepsilon}}^*(t,x,u))+q^{\varepsilon}(t,x,r,p{\sigma^{\varepsilon}}(t,x,u),u)\right],  \\
	&\qquad\qquad\qquad\qquad\qquad\qquad (t,x,r,p,l)\in [0,T]\times H\times \mathbb{R}\times H\times {\mathcal{S}}(H).
\end{aligned}
\end{eqnarray*}
Since $v^{\varepsilon}$ is a viscosity {sub-solution} of HJB equation (\ref{hjb1}) with coefficients $(b^{\varepsilon}, \sigma^{\varepsilon}, $ $  q^{\varepsilon}, \phi^{\varepsilon})$, we have
\begin{eqnarray*}
\begin{aligned}
	&\varphi_{t}(t_{\varepsilon},x^{\varepsilon})+{ \partial_t^o}g_2(t_{\varepsilon},x^{\varepsilon})
	+\langle A^*\nabla_x\varphi(t_{\varepsilon},x^{\varepsilon}), x^{\varepsilon}\rangle_H\\
	&+{\mathbf{H}}^{\varepsilon}(t_{\varepsilon},x^{\varepsilon},
	\nabla_x\varphi(t_{\varepsilon},x^{\varepsilon})+\nabla_xg_2(t_{\varepsilon},x^{\varepsilon}),
	\nabla^2_{x}\varphi(t_{\varepsilon},x^{\varepsilon})+\nabla^2_{x}g_2(t_{\varepsilon},x^{\varepsilon}))\geq0.
\end{aligned}
\end{eqnarray*}
Thus
\begin{eqnarray*}
\begin{aligned}
	0\leq&  \varphi_{t}(t_{\varepsilon},x^{\varepsilon})+{ \partial_t^o}g(t_{\varepsilon},x^{\varepsilon})
	+2({t}_{\varepsilon}-\hat{t}\,)+2\sum_{i=0}^{\infty}\frac{1}{2^i}({t_{\varepsilon}}-{t}_{i})
	+\langle A^*\nabla_x\varphi(t_{\varepsilon},x^{\varepsilon}), x^{\varepsilon}\rangle_H\\
	 &+{\mathbf{H}}(\hat{t},\hat{x},v(\hat{t},\hat{x}),\nabla_x{\varphi}(\hat{t},\hat{x})+\nabla_xg(\hat{t},\hat{x}),\nabla^2_{x}{\varphi}(\hat{t},\hat{x})+\nabla^2_{x}g(\hat{t},\hat{x}))+I\\[3mm]
	\leq&\varphi_{t}(\hat{t},\hat{x})+{\partial_t^o}g(\hat{t},\hat{x})+\langle A^*\nabla_x{\varphi}(\hat{t},\hat{x}), \hat{x}\rangle_H\\
	&
	 +{\mathbf{H}}(\hat{t},\hat{x},v(\hat{t},\hat{x}),\nabla_x{\varphi}(\hat{t},\hat{x})+\nabla_xg(\hat{t},\hat{x}),\nabla^2_{x}{\varphi}(\hat{t},\hat{x})+\nabla^2_{x}g(\hat{t},\hat{x}))+\varrho.
\end{aligned}
\end{eqnarray*}
Letting $\varrho\downarrow 0$, we show that
\begin{eqnarray*}
\begin{aligned}
&\varphi_{t}(\hat{t},\hat{x})+{ \partial_t^o}g(\hat{t},\hat{x})+\langle A^*\nabla_x{\varphi}(\hat{t},\hat{x}), \hat{x}\rangle_H\\
	 &+{\mathbf{H}}(\hat{t},\hat{x},v(\hat{t},\hat{x}),\nabla_x{\varphi}(\hat{t},\hat{x})+\nabla_xg(\hat{t},\hat{x}),\nabla^2_{x}{\varphi}(\hat{t},\hat{x})+\nabla^2_{x}g(\hat{t},\hat{x}))\geq0.
\end{aligned}
\end{eqnarray*}
Since ${\varphi}\in \Phi_{\hat{t}}$ and $g\in {\mathcal{G}}_{\hat{t}}$ with ${\hat{t}}\in [0,T)$  are arbitrary, we see that $v$ is a viscosity sub-solution of HJB equation (\ref{hjb1}) with the coefficients $b, \sigma, q, \phi$.  The proof  is complete.
\end{proof}



\end{document}